\documentclass[a4paper]{article}

\usepackage[utf8]{inputenc} 
\usepackage[T1]{fontenc} 
\usepackage{amsmath,amsthm,latexsym} 
\usepackage[fixamsmath]{mathtools}
\mathtoolsset{showmanualtags,mathic,centercolon} 
\usepackage{amssymb,amsfonts} 
\usepackage{booktabs} 
\usepackage{enumitem}
\usepackage[UKenglish]{babel} 
\usepackage[svgnames]{xcolor} 
\usepackage{algpseudocode} 
\usepackage{algorithm} 
\usepackage{bm} 
\usepackage{url}

\usepackage[margin=1.25in]{geometry} 
\usepackage[labelfont=sl,textfont=sl]{subcaption} 
\usepackage[font={small,sl},labelsep=period]{caption} 

\numberwithin{equation}{section} 
\theoremstyle{plain}
\newtheorem{theorem}{Theorem}[section] 
\newtheorem{lemma}[theorem]{Lemma} 
\newtheorem{corollary}[theorem]{Corollary}

\theoremstyle{definition}
\newtheorem{definition}[theorem]{Definition}

\newtheorem{assumption}[theorem]{Assumption}

\theoremstyle{remark}
\newtheorem{remark}[theorem]{Remark}

\newcommand{\appref}[1]{Appendix~\ref{#1}}
\newcommand{\secref}[1]{Section~\ref{#1}}
\newcommand{\defref}[1]{Definition~\ref{#1}}
\newcommand{\thmref}[1]{Theorem~\ref{#1}}
\newcommand{\lemref}[1]{Lemma~\ref{#1}}
\newcommand{\corref}[1]{Corollary~\ref{#1}}
\renewcommand{\algref}[1]{Algorithm~\ref{#1}}
\newcommand{\assref}[1]{Assumption~\ref{#1}}
\newcommand{\figref}[1]{Figure~\ref{#1}}

\newcommand{\remref}[1]{Remark~\ref{#1}}
\newcommand{\tabref}[1]{Table~\ref{#1}}

\newcommand{\R}{\mathbb{R}} 
\newcommand{\bigO}{\mathcal{O}} 
\DeclareMathOperator*{\argmin}{arg\,min} 
\newcommand{\defeq}{:=} 
\newcommand{\grad}{\nabla} 
\newcommand{\col}{\operatorname{col}} 
\newcommand{\Y}{\mathcal{Y}} 

\def\be{\begin{equation}}
\def\ee{\end{equation}}

\renewcommand{\b}[1]{\bm{#1}} 
\renewcommand{\t}[1]{\widetilde{#1}} 

\newcommand{\bx}{\b{x}}
\newcommand{\by}{\b{y}}
\newcommand{\bs}{\b{s}}
\newcommand{\bv}{\b{v}}
\newcommand{\br}{\b{r}}
\newcommand{\bee}{\b{e}}
\newcommand{\bg}{\b{g}}
\newcommand{\bem}{\b{m}}

\newcommand{\Prob}[1]{\mathbb{P}\left[#1\right]} 
\newcommand{\E}[1]{\mathbb{E}\left[#1\right]} 

\algrenewcommand\algorithmicrequire{\textbf{Input:}}
\algrenewcommand\algorithmicensure{\textbf{Output:}}

\begin{document}
\title{Scalable Subspace Methods for Derivative-Free Nonlinear Least-Squares Optimization}
\author{
Coralia Cartis\thanks{Mathematical Institute, University of Oxford, Radcliffe Observatory Quarter, Woodstock Road, Oxford, OX2 6GG, United Kingdom (\texttt{cartis@maths.ox.ac.uk}). This work was supported by the EPSRC Centre for Doctoral Training in Industrially Focused Mathematical Modelling (EP/L015803/1) in collaboration with the Numerical Algorithms Group Ltd.}
\and 
Lindon Roberts\thanks{Mathematical Sciences Institute, Building 145, Science Road, Australian National University, Canberra ACT 2601, Australia (\texttt{lindon.roberts@anu.edu.au}). We note that the work in Sections \ref{sec_implementation} and \ref{sec_numerics} originally appeared in this author's thesis \cite[Chapter 7]{Roberts2019}.}}

\date{\today}
\maketitle

\begin{abstract}
	We introduce a general framework for large-scale model-based derivative-free optimization based on iterative minimization within random subspaces.
	We present a probabilistic worst-case complexity analysis for our method, where in particular we prove high-probability bounds on the number of iterations before a given optimality is achieved.
	This framework is specialized to nonlinear least-squares problems, with a model-based framework based on the Gauss-Newton method.
	This method achieves scalability by constructing local linear interpolation models to approximate the Jacobian, and computes new steps at each iteration in a subspace with user-determined dimension.
	We then describe a practical implementation of this framework, which we call DFBGN.
	We outline efficient techniques for selecting the interpolation points and search subspace, yielding an implementation that has a low per-iteration linear algebra cost (linear in the problem dimension) while also achieving fast objective decrease as measured by evaluations.
	Extensive numerical results demonstrate that DFBGN has improved scalability, yielding strong performance on large-scale nonlinear least-squares problems.
\end{abstract}

\\


\section{Introduction} \label{sec_introduction}
An important class of nonlinear optimization methods is so-called derivative-free optimization (DFO).
In DFO, we consider problems where derivatives of the objective (and/or constraints) are not available to be evaluated, and we only have access to function values.
This topic has received growing attention in recent years, and is primarily used for objectives which are black-box (so analytic derivatives or algorithmic differentiation are not available), and expensive to evaluate or noisy (so finite differencing is impractical or inaccurate).
There are many types of DFO methods, such as model-based, direct and pattern search, implicit filtering and others (see \cite{Larson2019} for a recent survey), and these techniques have been used in a variety of applications \cite{Alarie2020}.

Here, we consider model-based DFO methods for unconstrained optimization, which are based on iteratively constructing and minimizing interpolation models for the objective.
We also specialize these methods for nonlinear least-squares problems, by constructing interpolation models for each residual term rather than for the full objective \cite{Zhang2010,Wild2017,Cartis2019a}.

This paper aims to provide a method that attempts to answer a key question regarding model-based DFO: how to improve the scalability of this class.
Existing model-based DFO techniques are primarily designed for small- to medium-scale problems, as the linear algebra cost of each iteration---largely due to the cost of constructing interpolation models---means that their runtime increases rapidly for large problems.
There are several settings where scalable DFO algorithms may be useful, such as data assimilation \cite{Bergou2016,Arter2018a}, machine learning \cite{Salimans2017,Ghanbari2017}, generating adversarial examples for deep neural networks \cite{Alzantot2018,Ughi2019}, image analysis \cite{Ehrhardt2020}, and as a possible proxy for global optimization methods \cite{Cartis2019}.

To address this, we introduce RSDFO, a scalable algorithmic framework for model-based DFO. 
At each iteration of RSDFO we select a random low-dimensional subspace, build and minimize a model to compute a step in this space, then change the subspace at the next iteration.
We provide a probabilistic worst-case complexity analysis of RSDFO.
To our knowledge, this is the first \emph{subspace model-based DFO method with global complexity and convergence guarantees}.
We then describe how this general framework can be specialized to the case of nonlinear least-squares minimization through a model construction technique inspired by the Gauss-Newton method, yielding a new algorithm RSDFO-GN with associated worst-case complexity bounds.
We then present an efficient implementation of RSDFO-GN, which we call DFBGN.
DFBGN is available on Github\footnote{\url{https://github.com/numericalalgorithmsgroup/dfbgn}} and includes several algorithmic features that yield strong performance on large-scale problems and a low per-iteration linear algebra cost that is typically linear in the problem dimension.

\subsection{Existing Literature}
The contributions in this paper are connected to several areas of research.
We briefly review these topics below.

\paragraph{Block Coordinate Descent}
There is a large body of work on (derivative-based) block coordinate descent (BCD) methods, typically motivated by machine learning applications.
BCD extends coordinate search methods \cite{Wright2015} by updating a subset of the variables at each iteration, typically using a coordinate-wise variant of a first-order method.
For nonconvex problems, the first convergence result for a randomized coordinate descent method based on proximal gradient descent was given in \cite{Patrascu2015}.
Here, the sampling of coordinates was uniform and required step sizes based on Lipschitz constants associated with the objective.
This was extended in \cite{Lu2018} to general randomized block selection with a nonomonotone linesearch-type method (to allow for unknown Lipschitz constants), and to a (possibly deterministic) `essentially cyclic' block selection and extrapolation (but requiring Lipschitz constants) in \cite{Xu2017}.
Several extensions of this approach have been developed, including the use of stochastic gradients \cite{Xu2015}, parallel block updating \cite{Facchinei2015} and inexact step calculations \cite{Facchinei2015,Yang2020}.

BCD methods have been extended to nonlinear least-squares problems, leading to so-called Subspace Gauss-Newton methods.
These are derivative-based methods where a Gauss-Newton step is computed for a subset of variables.
This approach was initially proposed in \cite{Tett2017} for parameter estimation in climate models---where derivative estimates were computed using implicit filtering \cite{Kelley1999}---and analyzed in quadratic regularization and trust-region settings for general unconstrained objectives in \cite{Fowkes2020,Shao2021}.

\paragraph{Sketching}
Sketching is an alternative dimensionality reduction technique for least-squares problems, reducing the number of residuals rather than the number of variables.
Sketching ideas have been applied to linear \cite{Halko2011,Mahoney2011,Woodruff2014} and nonlinear \cite{Ergen2019} least-squares problems, as well as model-based DFO for nonlinear least-squares \cite{Cartis2020}, as well as subsampling algorithms for finite sum-of-functions minimization such as Newton's method \cite{Roosta-Khorasani2019,Berahas2020}.

There are also alternative approaches to sketching which lead to subspace-type methods, where local gradients and Hessians are estimated only within a subspace (possibly used in conjunction with random subsampling).
Sketching in this context has been applied to, for example, Newton's method \cite{Pilanci2015,Gower2019,Berahas2020}, BFGS \cite{Gower2016a}, and SAGA \cite{Gower2020}, as well as to trust-region and quadratic regularization methods \cite{Fowkes2020,Shao2021}.

\paragraph{Random Embeddings for Global Optimization}
Some global optimization methods have been proposed which randomly project a high-dimensional problem into a low-dimensional subspace and solve this smaller problem using existing (global or local) methods. 
Though applicable to general global optimization problems (as a more sophisticated variant of random search), this technique has been explored particularly for defeating the curse of dimensionality when optimising functions which have low effective dimensionality \cite{Qian2016,Wang2016,Otemissov2020}.
For the latter class, often only one random subspace projection is needed, though the addition of constraints leads to multiple embeddings being required \cite{Otemissov2020}. 
Our approach here differs from these works in both theoretical and numerical aspects, as it is focused on a specific random subspace technique for local optimization.

\paragraph{Probabilistic Model-Based DFO}
For model-based DFO, several algorithms have been developed and analyzed where the local model at each iteration is only sufficiently accurate with a certain probability \cite{Bandeira2014,Chen2016,Blanchet2016}.
Similar analysis also exists for derivative-based algorithms \cite{Cartis2018a,Gratton2017b}.
Our approach is based on deterministic model-based DFO within subspaces, and we instead require a very weak probabilistic condition on the (randomly chosen) subspaces (\assref{ass_sketching}).

\paragraph{Randomized Direct Search DFO}
In randomized direct search methods, iterates are perturbed in a random subset of directions (rather than a positive spanning set) when searching for local improvement.
In this framework, effectively only a random subspace is searched in each iteration.
Worst-case complexity bounds for this technique are given under predetermined step length regimes in \cite{Bergou2019,Golovin2019}, and with adaptive step sizes in \cite{Gratton2015,Gratton2019b}, where \cite{Gratton2019b} extends \cite{Gratton2015} to linearly constrained problems.

\paragraph{Large-Scale DFO}
There have been several alternative approaches considered for improving the scalability of DFO.
These often consider problems with specific structure which enable efficient model construction, such as partial separability \cite{Colson2005,Porcelli2020}, sparse Hessians \cite{Bandeira2012}, and minimization over the convex hull of finitely many points \cite{Cristofari2020}.
On the other hand, there is a growing body of literature on `gradient sampling' techniques for machine learning problems. 
These methods typically consider stochastic first-order methods but with a gradient approximation based on finite differencing in random directions \cite{Nesterov2017}.
This framework has lead to variants of methods such as stochastic gradient descent \cite{Ghadimi2013}, SVRG \cite{Liu2018} and Adam \cite{Chen2019}, for example.
We note that linear interpolation to orthogonal directions---more similar to traditional model-based DFO---has been shown to outperform gradient sampling as a gradient estimation technique \cite{Berahas2019,Berahas2019c}.

\paragraph{Subspace DFO Methods}
A model-based DFO method with similarities to our subspace approach is the moving ridge function method from \cite{Gross2020}.
Here, existing objective evaluations are used to determine an `active subspace' which captures the largest variability in the objective and build an interpolation model within this subspace.
We also note the VXQR method from \cite{Neumaier2011}, which performs line searches along a direction chosen from a subspace determined by previous iterates.
Both of these methods do not include convergence theory.
By comparison, aside from our focus on nonlinear least-squares problems, both our general theoretical framework and our implemented method select their working subspaces randomly, and we provide (probabilistic) convergence guarantees.

\subsection{Contributions}
We introduce RSDFO (Randomized Subspace Derivative-Free Optimization), a generic model-based DFO framework that relies on constructing a model in a subspace at each iteration.
Our novel approach enables model-based DFO methods to be applied in a large-scale regime by giving the user explicit control over the subspace dimension, and hence control over the per-iteration linear algebra cost of the method.
This framework is then specialized to the case of nonlinear least-squares problems, yielding a new algorithm RSDFO-GN (Randomized Subspace DFO with Gauss-Newton).
The subspace model construction framework of RSDFO-GN is based on DFO Gauss-Newton methods \cite{Cartis2019a,Cartis2018}, and retains the same theoretical guarantees as RSDFO.
We then describe a practical implementation of RSDFO-GN, which we call DFBGN (Derivative-Free Block Gauss-Newton).\footnote{Technically, DFBGN is not a block method as its subspaces are not coordinate-aligned, but has already been released with this name.}
Compared to existing methods, DFBGN reduces the linear algebra cost of model construction and the initial objective evaluation cost by allowing fewer interpolation points at every iteration.
In order for DFBGN to have both scalability and a similar evaluation efficiency to existing methods (i.e.~objective reduction achieved for a given number of objective evaluations), several modifications to the theoretical framework, regarding the selection of interpolation points and the search subspace, are necessary.

\paragraph{Theoretical Results}
We consider a generic theoretical framework RSDFO, where the subspace dimension is a user-chosen algorithm hyperparameter, and no specific model construction approach is specified.
Our framework is not specific to a least-squares problem structure, and holds for any objective with Lipschitz continuous gradient, and allows for a general class of random subspace constructions (not relying on a specific class of embeddings or projections).
The theoretical results here extend the approach and techniques in \cite{Fowkes2020,Shao2021} to model-based DFO methods.
In particular, we use the notion of a well-aligned subspace (\defref{def_well_aligned}) from \cite{Fowkes2020,Shao2021}, one in which sufficient decrease is achievable, and assume that our search subspace is well-aligned with some probability (\assref{ass_sketching}).
This is achieved provided we select a sufficiently large subspace dimension (depending on the desired failure probability and subspace alignment quality).

We derive a high probability worst-case complexity bound for RSDFO. 
Specifically, our main bounds are of the form $\Prob{\min_{j\leq k} \|\grad f(\bx_j)\| \leq C k^{-1/2}} \geq 1 - e^{-ck}$ and $\Prob{K_{\epsilon} \leq C\epsilon^{-2}}  \leq 1-e^{-c\epsilon^{-2}}$, where $K_{\epsilon}$ is the first iteration to achieve first-order optimality $\epsilon$ (see \thmref{thm_high_prob_complexity} and \corref{cor_high_prob_complexity}).
This result then implies a variety of alternative convergence results, such as expectation bounds and almost-sure convergence.
Based on \cite{Fowkes2020,Shao2021}, we give several constructions for determining our random subspace, and show that for constructions based on Johnson-Lindenstrauss transformations, that we can achieve convergence with a subspace dimension that is \emph{independent of the ambient dimension}.

Our analysis matches the standard deterministic $\bigO(\epsilon^{-2})$ complexity bounds for model-based DFO methods (e.g.~\cite{Garmanjani2016,Cartis2019a}).
Compared to the analysis of derivative-based methods (e.g.~BCD \cite{Xu2017} and probabilistically accurate models \cite{Cartis2018a}) we need to incorporate the possibility that the interpolation model is not accurate (not fully linear, see \defref{def_reduced_fully_linear}).
However, unlike \cite{Bandeira2014,Chen2016,Blanchet2016} we do not assume that full linearity is a stochastic property; instead, our stochasticity comes from the subspace selection and we explicitly handle non-fully linear models similar to \cite{Conn2009,Cartis2019a}.
This gives us a framework which is similar to standard model-based DFO and with weak probabilistic conditions.
Compared to the analysis of derivative-based random subspace methods in \cite{Fowkes2020,Shao2021}, our analysis is complicated substantially by the possibility of inaccurate models and the intricacies of model-based DFO algorithms.

We then consider RSDFO-GN, which explicitly describes how interpolation models can be constructed for nonlinear least-squares problems, thus providing a concrete implementation of RSDFO in this context.
We prove that RSDFO-GN retains the same complexity bounds as RSDFO.

\paragraph{Implementation}
Although RSDFO-GN gives a reduced linear algebra cost of model construction compared to existing methods, it is important that we have an implementation that achieves this reduced cost without sacrificing overall performance (in the sense of objective decrease achieved within a given number of objective evaluations).

We introduce a practical, implementable variant of RSDFO-GN called DFBGN, which is based on the solver DFO-LS \cite{Cartis2018}.
DFBGN includes an efficient, geoemtry-aware approach for selecting interpolation points, and hence directions in our subspace, for removal, an adaptive randomized approach for selecting new interpolation points/subspace directions.
We study the per-iteration linear algebra cost of DFBGN, and show that it is \emph{linear in the problem dimension}, a substantial improvement over existing methods, which are cubic in the problem dimension.
Our per-iteration linear algebra costs are also linear in the number of residuals, the same as existing methods, but with a substantially smaller constant (quadratic in the subspace dimension, which is user-determined, rather than quadratic in the problem dimension).

\paragraph{Numerical Results}
We compare DFBGN with DFO-LS (which itself is shown to have state-of-the-art performance in \cite{Cartis2018}) on collections of both medium-scale (approx.~100 dimensions) and large-scale test problems (approx.~1000 dimensions).
We show that DFBGN with a full-sized subspace has similar performance to DFO-LS in terms of objective evaluations, but shows improved performance on runtime.
As the dimension of the subspace reduces (i.e.~the size of the interpolation set reduces), we demonstrate a tradeoff between reduced linear algebra costs and increased evaluation counts required to achieve a given objective reduction.
The flexibility of DFBGN allows this tradeoff to be explicitly managed.
When tested on large-scale problems, DFO-LS frequently reaches a reasonable runtime limit without making substantial progress, whereas DFBGN with small subspace size can perform many more iterations and hence make better progress than DFO-LS.
In the case of expensive objectives with small evaluation budgets, we show that DFBGN can make progress with few objective evaluations in a similar way to DFO-LS (which has a mechanism to make progress from as few as 2 objective evaluations independent of problem dimension), but with substantially lower linear algebra costs.

\paragraph{Structure of paper}
In \secref{sec_rsdfo} we describe RSDFO and provide our probabilistic worst-case complexity analysis.
We specialize RSDFO to RSDFO-GN in \secref{sec_rsdfogn}.
Then we describe the practical implementation DFBGN and its features in \secref{sec_implementation}.
Our numerical results are given in \secref{sec_numerics}.

\paragraph{Implementation}
A Python implementation of DFBGN is available on Github.\footnote{\url{https://github.com/numericalalgorithmsgroup/dfbgn}}

\paragraph{Notation}
We use $\|\cdot\|$ to refer to the Euclidean norm of vectors and the operator 2-norm of matrices, and $B(\bx,\Delta)$ for $\bx\in\R^n$ and $\Delta>0$ to be the closed ball $\{\by\in\R^n : \|\by-\bx\|\leq\Delta\}$.

\section{Random Subspace Model-Based DFO} \label{sec_rsdfo}
In this section we outline our general model-based DFO algorithmic framework based on minimization in random subspaces.
We consider the nonconvex problem
\begin{align}
	\min_{\bx\in\R^n} f(\bx), \label{eq_opt_generic}
\end{align}
where we assume that $f:\R^n\to\R$ is continuously differentiable, but that access to its gradient is not possible (e.g.~for the reasons described in \secref{sec_introduction}).
In a standard model-based DFO framework (e.g.~\cite{Conn2009,Larson2019}), at each iteration $k$ we construct a quadratic model $m_k:\R^n\to\R$ which approximates $f$ near our iterate $\bx_k$:
\begin{align}
	f(\bx_k + \bs) \approx m_k(\bs) \defeq f(\bx_k) + \bg_k^T \bs + \frac{1}{2}\bs^T H_k \bs, \label{eq_f_model_generic}
\end{align}
for some $\bg_k\in\R^n$ and $H_k\in\R^{n\times n}$ symmetric.
Based on this model, we build a globally convergent algorithm using a trust-region framework \cite{Conn2000}.
This algorithmic framework is suitable providing that---when necessary---we can guarantee $m_k$ is a sufficiently accurate model for $f$ near $\bx_k$.
Details about how to construct sufficiently accurate models based on interpolation are given in \cite{Conn2009}.

Our core idea here is to construct interpolation models which only approximate the objective in a subspace, rather than in the full space $\R^n$.
This allows us to use interpolation sets with fewer points, since we do not have to capture the objective's behaviour outside our subspace, which improves the scalability of the method.

In this section, we outline our general algorithmic framework and provide a worst-case complexity analysis showing convergence to first-order stationary points with high probability.
We then describe how this framework may be specialized to the case of nonlinear least-squares minimization.

\subsection{RSDFO Algorithm}
In our general framework, which we call RSDFO (Randomized Subspace DFO), we modify the above approach by randomly choosing a $p$-dimensional subspace (where $p<n$ is user-chosen) and constructing an interpolation model defined only in that subspace.\footnote{Formally, we define our model in an affine space, but we call it a subspace throughout as this fits with an intuitive view of what RSDFO aims to achieve.}
Specifically, in each iteration $k$ we randomly choose a $p$-dimensional affine space $\Y_k \subset \R^n$ given by the range of $Q_k\in\R^{n\times p}$, i.e.~
\be \Y_k = \{\bx_k + Q_k \hat{\bs} : \hat{\bs}\in\R^p\}. \label{eq_subspace_definition} \ee

We then construct a model which interpolate $f$ at points in $\Y_k$ and ultimately construct a local quadratic model for $f$ only on $\Y_k$.
That is, given $Q_k$, we assume that we have $\hat{m}_k:\R^p\to\R$ given by
\be f(\bx_k + Q_k\hat{\bs}) \approx \hat{m}_k(\hat{\bs}) := f(\bx_k) + \hat{\bg}_k^T\hat{\bs} + \frac{1}{2}\hat{\bs}^T \hat{H}_k \hat{\bs}, \label{eq_reduced_model_generic} \ee
where $\hat{\bg}_k\in\R^p$ and $\hat{H}_k\in\R^{p\times p}$ are the low-dimensional model gradient and Hessian respectively, adopting the convention of using hats on variables to denote low-dimensional quantities.
In \secref{sec_rsdfogn} we specialize this to a model construction process for nonlinear least-squares problems.

For our trust-region algorithm, we (approximately) minimize $\hat{m}_k$ inside the trust region to get a tentative step
\be \hat{\bs}_k \approx \argmin_{\hat{\bs}\in\R^p} \hat{m}_k(\hat{\bs}), \quad \text{s.t.} \quad \|\hat{\bs}\|\leq\Delta_k, \label{eq_reduced_trs_generic} \ee
for the current trust-region radius $\Delta_k>0$, yielding a tentative step $\bs_k = Q_k \hat{\bs}_k \in\R^n$.
We thus also get the computational advantage coming from solving a $p$-dimensional trust-region subproblem.

In our setting we are only interested in the approximation properties of $\hat{m}_k$ in the space $\Y_k$, and so we introduce the following notion of a ``sufficiently accurate'' model:

\begin{definition} \label{def_reduced_fully_linear}
Given $Q\in\R^{n\times p}$, a model $\hat{m}:\R^p\to\R$ is $Q$-fully linear in $B(\bx,\Delta)\subset\R^n$ if
\begin{subequations}
\begin{align}
    |f(\bx+Q\hat{\bs}) - \hat{m}(\hat{\bs})| &\leq \kappa_{\rm ef}\Delta^2, \label{eq_fully_linear_obj} \\
    \|Q^T \grad f(\bx+Q\hat{\bs}) - \grad\hat{m}(\hat{\bs})\| &\leq \kappa_{\rm eg}\Delta, \label{eq_fully_linear_grad}
\end{align}
\end{subequations}
for all $\bs\in\R^p$ with $\|\hat{\bs}\|\leq\Delta$.
The constants $\kappa_{\rm ef}$ and $\kappa_{\rm eg}$ must be independent of $Q$, $\hat{m}$, $\bx$ and $\Delta$.
\end{definition}

The gradient condition \eqref{eq_fully_linear_grad} comes from noting that if $\hat{f}(\hat{\bs})\defeq f(\bx+Q\hat{\bs})$ then $\grad \hat{f}(\hat{\bs}) = Q^T \grad f(\bx+Q\hat{\bs})$.
We note that if we have full-dimensional subspaces $p=n$ and take $Q=I$, then we recover the standard notion of fully linear models \cite[Definition 6.1]{Conn2009}.

\begin{algorithm}[tb]
	\footnotesize{
	\begin{algorithmic}[1]
		\Require Starting point $\bx_0\in\R^n$, initial trust region radius $\Delta_0>0$, and subspace dimension $p\in\{1,\ldots,n\}$. 
		\vspace{0.2em}
		\Statex \underline{Parameters}: maximum trust-region radius $\Delta_{\rm max}\geq\Delta_0$, trust-region radius scalings $0<\gamma_{\rm dec}<1<\gamma_{\rm inc}\leq\overline{\gamma}_{\rm inc}$, criticality constants $\epsilon_C,\mu>0$ and trust-region scaling $0<\gamma_C<1$, safety step threshold $\beta_F>0$ and trust-region scaling $0<\gamma_F<1$, and acceptance thresholds $0 < \eta_1 \leq \eta_2 < 1$.
		\vspace{0.5em}
		\State Set flag \texttt{CHECK\_MODEL}=\texttt{FALSE}.
		\For{$k=0,1,2,\ldots$}
		    \If{\texttt{CHECK\_MODEL}=\texttt{TRUE}}
				\State Set $Q_k=Q_{k-1}$.
		        \State Construct a reduced model $\hat{m}_k:\R^p\to\R$ \eqref{eq_reduced_model_generic} which is $Q_k$-fully linear in $B(\bx_k,\Delta_k)$. 
		    \Else
				\State Define a subspace by randomly sampling $Q_k\in\R^{n\times p}$.
		        \State Construct a reduced model $\hat{m}_k:\R^p\to\R$ \eqref{eq_reduced_model_generic} which need not be $Q_k$-fully linear.
		    \EndIf
		    \If{$\|\hat{\bg}_k\|<\epsilon_C$ \textbf{and} \textit{($\|\hat{\bg}_k\|<\mu^{-1}\Delta_k$ \textbf{or} $\hat{m}_k$ is not $Q_k$-fully linear in $B(\bx_k,\Delta_k)$)}} \label{ln_main_start}
			    \State (Criticality step) Set $\bx_{k+1}=\bx_k$, $\Delta_{k+1}=\gamma_C\Delta_k$ and \texttt{CHECK\_MODEL}=\texttt{TRUE}.
			\Else \quad $\leftarrow$ \textit{$\|\hat{\bg}_k\|\geq\epsilon_C$ or ($\|\hat{\bg}_k\|\geq\mu^{-1}\Delta_k$ and $\hat{m}_k$ is $Q_k$-fully linear in $B(\bx_k,\Delta_k)$)}
				\State Approximately solve the subspace trust-region subproblem in $\R^{p}$ \eqref{eq_reduced_trs_generic} and calculate the step $\bs_k = Q_k \hat{\bs}_k \in\R^n$.
				\If{$\|\hat{\bs}_k\| < \beta_F \Delta_k$}
    			    \State (Safety step) Set $\bx_{k+1}=\bx_k$ and $\Delta_{k+1}=\gamma_{F}\Delta_k$.
    			    \State If $\hat{m}_k$ is $Q_k$-fully linear in $B(\bx_k,\Delta_k)$, then set \texttt{CHECK\_MODEL}=\texttt{FALSE}, otherwise \texttt{CHECK\_MODEL}=\texttt{TRUE}.
			    \Else
    				\State Evaluate $f(\bx_k+\bs_k)$ and calculate ratio 
    			    \be \rho_k \defeq \frac{f(\bx_k)-f(\bx_k+\bs_k)}{\hat{m}_k(\b{0})-\hat{m}_k(\hat{\bs}_k)}. \label{eq_ratio_generic} \ee
    			    \State Accept/reject step and update trust region radius: set
    			    \be \bx_{k+1} = \begin{cases}\bx_k + \bs_k, & \rho_k \geq \eta_1, \\ \bx_k, & \rho_k < \eta_1, \end{cases} \quad \text{and} \quad \Delta_{k+1} = \begin{cases}\min(\max(\gamma_{\rm inc}\Delta_k, \overline{\gamma}_{\rm inc}\|\hat{\bs}_k\|), \Delta_{\rm max}), & \rho_k \geq \eta_2, \\ \max(\gamma_{\rm dec}\Delta_k, \|\hat{\bs}_k\|), & \eta_1 \leq \rho_k < \eta_2, \\ \min(\gamma_{\rm dec}\Delta_k, \|\hat{\bs}_k\|), & \rho_k < \eta_1. \end{cases} \label{eq_dfbgn_tr_updating_generic} \ee
    			    \State If $\rho_k\geq\eta_2$ or $\hat{m}_k$ is $Q_k$-fully linear in $B(\bx_k,\Delta_k)$, then set \texttt{CHECK\_MODEL}=\texttt{FALSE}, otherwise set \texttt{CHECK\_MODEL}=\texttt{TRUE}.
    	        \EndIf
			\EndIf \label{ln_main_end}
		\EndFor
	\end{algorithmic}
	} 
	\caption{RSDFO (Randomized Subspace Derivative-Free Optimization) for solving \eqref{eq_opt_generic}.}
	\label{alg_rsdfo}
\end{algorithm}

\paragraph{Complete RSDFO Algorithm}
The complete RSDFO algorithm is stated in \algref{alg_rsdfo}.
The overall structure is common to model-based DFO methods \cite{Conn2009}.
In particular, we assume that we have procedures to verify whether or not a model is $Q_k$-fully linear in $B(\bx_k,\Delta_k)$ and (if not) to generate a $Q_k$-fully linear model.
When we specialize RSDFO to nonlinear least-squares problems in \secref{sec_rsdfogn}, we will describe how we can obtain such procedures.

After defining our subspace and model, we first perform a criticality step, which guarantees that---whenever we suspect we are close to first-order stationarity, as measured by $\|\hat{\bg}_k\|$---we have an accurate model and an appropriately sized trust-region radius.
Often this criticality step is formulated as its own subroutine with an extra inner loop \cite{Conn2009}, but following \cite{Garmanjani2016} we only perform one criticality step per iteration and avoid nested loops.

The remainder of \algref{alg_rsdfo} broadly follows standard trust-region methods.
If the computed step is much shorter than the trust-region radius, we enter a safety step---originally due to Powell \cite{Powell2003}---which is similar to an unsuccessful iteration (i.e.~$\rho_k<\eta_1$, so we reject the step and decrease $\Delta_k$) but without evaluating $f(\bx_k+\bs_k)$ and hence saving one objective evaluation.
Our updating mechanism for $\Delta_k$ \eqref{eq_dfbgn_tr_updating_generic} also takes into account the computed step size, and is the same as in \cite{Powell2009,Cartis2019a}.

An important feature of RSDFO is that in some iterations, we reuse the previous subspace, $\Y_k=\Y_{k-1}$, corresponding to the flag \texttt{CHECK\_MODEL}=\texttt{TRUE}.
In this case, we had an inaccurate model in iteration $k-1$ and require that our new model $\hat{m}_k$ is accurate ($Q_k$-fully linear).
This mechanism essentially ensures that $\Delta_k$ is not decreased too quickly as a result of inaccurate models, and is mostly decreased to achieve sufficient objective reduction.

We now give our convergence and worst-case complexity analysis of \algref{alg_rsdfo}.

\subsection{Assumptions and Preliminary Results}
We begin our analysis with some basic assumptions and preliminary results.

\begin{assumption}[Smoothness] \label{ass_smoothness}
The objective function $f:\R^n\to\R$ is bounded below by $f_{\rm low}$ and continuously differentiable, and $\grad f$ is $L_{\grad f}$-Lipschitz continuous in the extended level set $\{\by\in\R^n : \|\by-\bx\| \leq\Delta_{\max} \: \text{for some} \: f(\bx) \leq f(\bx_0)\}$, for some constant $L_{\grad f}>0$.
\end{assumption}

We also need two standard assumptions for trust-region methods: uniformly bounded above model Hessians and sufficiently accurate solutions to the trust-region subproblem \eqref{eq_reduced_trs_generic}.

\begin{assumption}[Bounded model Hessians] \label{ass_bdd_hess}
We assume that $\|\hat{H}_k\|\leq\kappa_H$ for all $k$, for some $\kappa_H\geq1$.
\end{assumption}

\begin{assumption}[Cauchy decrease] \label{ass_cauchy_decrease}
Our method for solving the trust-region subproblem \eqref{eq_reduced_trs_generic} gives a step $\hat{\bs}_k$ satisfying the sufficient decrease condition
\be \hat{m}_k(\b{0}) - \hat{m}_k(\hat{\bs}_k) \geq c_1 \|\hat{\bg}_k\| \min\left(\Delta_k, \frac{\|\hat{\bg}_k\|}{\max(\|\hat{H}_k\|,1)}\right), \label{eq_cauchy_decrease} \ee
for some $c_1\in[1/2, 1]$ independent of $k$.
\end{assumption}

A useful consequence, needed for the analysis of our trust-region radius updating scheme, is the following.

\begin{lemma}[Lemma 3.6, \cite{Cartis2019a}] \label{lem_step_lower_bound}
Suppose \assref{ass_cauchy_decrease} holds.
Then
\be \|\hat{\bs}_k\| \geq c_2 \min\left(\Delta_k, \frac{\|\hat{\bg}_k\|}{\max(\|\hat{H}_k\|, 1)}\right), \ee
where $c_2 := 2c_1 / (1+\sqrt{1+2c_1})$.
\end{lemma}

\begin{lemma} \label{lem_eventually_successful}
Suppose Assumptions \ref{ass_bdd_hess} and \ref{ass_cauchy_decrease} hold, and we run RSDFO with $\beta_F \leq c_2$.
If $\hat{m}_k$ is $Q_k$-fully linear in $B(\bx_k,\Delta_k)$ and
\be \Delta_k \leq c_0 \|\hat{\bg}_k\|, \qquad \text{where} \qquad c_0 := \min\left(\mu, \frac{1}{\kappa_H}, \frac{c_1(1-\eta_2)}{2\kappa_{\rm ef}}\right), \ee
then the criticality and safety steps are not called, and $\rho_k\geq\eta_2$.
\end{lemma}
\begin{proof}
Since $\hat{m}_k$ is $Q_k$-fully linear and $\Delta_k \leq \mu\|\hat{\bg}_k\|$, the criticality step is not called.
From \lemref{lem_step_lower_bound} and $\Delta_k \leq \|\hat{\bg}_k\|/\kappa_H$, we have $\|\hat{\bs}_k\| \geq c_2 \Delta_k \geq \beta_F \Delta_k$ and so the safety step is not called.

From Assumptions \ref{ass_bdd_hess} and \ref{ass_cauchy_decrease}, we have
\be  \hat{m}_k(\b{0}) - \hat{m}_k(\hat{\bs}_k) \geq c_1 \|\hat{\bg}_k\| \min\left(\Delta_k, \frac{\|\hat{\bg}_k\|}{\kappa_H}\right) = c_1 \|\hat{\bg}_k\| \Delta_k, \ee
since $\Delta_k \leq \|\hat{\bg}_k\|/\kappa_H$ by assumption.
Next, since $\hat{m}_k$ is $Q_k$-fully linear, from \eqref{eq_fully_linear_obj} we have
\begin{align}
    |f(\bx_k) - \hat{m}_k(\b{0})| &= |f(\bx_k + Q_k\b{0}) - \hat{m}_k(\b{0})| \leq \kappa_{\rm ef}\Delta_k^2, \\
    |f(\bx_k+\bs_k) - \hat{m}_k(\hat{\bs}_k)| &= |f(\bx_k + Q_k\hat{\bs}_k) - \hat{m}_k(\hat{\bs}_k)| \leq \kappa_{\rm ef}\Delta_k^2.
\end{align}
Hence we have
\begin{align}
    |\rho_k-1| &\leq \frac{|f(\bx_k) - \hat{m}_k(\b{0})|}{|\hat{m}_k(\b{0}) - \hat{m}_k(\hat{\bs}_k)|} + \frac{|f(\bx_k+\bs_k) - \hat{m}_k(\hat{\bs}_k)|}{|\hat{m}_k(\b{0}) - \hat{m}_k(\hat{\bs}_k)|}
    \leq \frac{2\kappa_{\rm ef}\Delta_k^2}{c_1 \|\hat{\bg}_k\| \Delta_k}
    \leq 1-\eta_2,
\end{align}
since $\Delta_k \leq c_0\|\hat{\bg}_k\| \leq c_1 (1-\eta_2)\|\hat{\bg}_k\|/(2\kappa_{\rm ef})$.
Thus $\rho_k\geq\eta_2$, which is the claim of the lemma.
\end{proof}

\begin{remark}
	The requirement $\beta_F \leq c_2$ in \lemref{lem_eventually_successful} is not restrictive.
	Since have $c_1 \geq 1/2$ in \assref{ass_cauchy_decrease}, it suffices to choose $\beta_F \leq \sqrt{2}-1$, for example.
\end{remark}

Our key new assumption is on the quality of our subspace selection, as suggested in \cite{Fowkes2020,Shao2021}:

\begin{definition} \label{def_well_aligned}
The matrix $Q_k$ is well-aligned if
\be \|Q_k^T \grad f(\bx_k)\| \geq \alpha_Q \|\grad f(\bx_k)\|, \label{eq_sketched_gradient} \ee
for some $\alpha_Q\in(0,1)$ independent of $k$.
\end{definition}

\begin{assumption}[Subspace quality] \label{ass_sketching}
Our subspace selection (determined by $Q_k$) satisfies the following two properties:
\begin{enumerate}[label=(\alph*)]
	\item At each iteration $k$ of RSDFO in which \texttt{CHECK\_MODEL} = \texttt{FALSE}, our subspace selection $Q_k$ is well-aligned for some fixed $\alpha_Q\in(0,1)$ with probability at least $1-\delta_S$, for some $\delta_S\in(0,1)$, independently of $\{Q_0,\ldots,Q_{k-1}\}$. \label{item_sketched_aligned_prob}
	\item $\|Q_k\| \leq Q_{\max}$ for all $k$ and some $Q_{\max}>0$. \label{item_sketch_bounded}
\end{enumerate}
\end{assumption}

Of these two properties, \ref{item_sketched_aligned_prob} is needed for our complexity analysis, while \ref{item_sketch_bounded} is only needed in order to construct $Q_k$-fully linear models (in \secref{sec_rsdfogn}).
We will discuss how to achieve \assref{ass_sketching} in more detail in \secref{sec_achieving_sketch}.

\begin{lemma} \label{lem_epsilon_g}
In all iterations $k$ of RSDFO where the criticality step is not called, we have $\|\hat{\bg}_k\| \geq \min(\epsilon_C, \mu^{-1}\Delta_k)$.
If the criticality step is not called in iteration $k$, $Q_k$ is well-aligned and $\|\grad f(\bx_k)\|\geq \epsilon$, then
\be \|\hat{\bg}_k\| \geq \epsilon_g(\epsilon) := \min\left(\epsilon_C, \frac{\alpha_Q \epsilon}{\kappa_{\rm eg}\mu + 1}\right) > 0. \label{eq_epsilon_g} \ee
\end{lemma}
\begin{proof}
The first part follows immediately from the entry condition of the criticality step.
To prove \eqref{eq_epsilon_g}, suppose the criticality step is not called in iteration $k$ and $\|\hat{\bg}_k\|<\epsilon_C$.
Then we have $\Delta_k \leq \mu\|\hat{\bg}_k\|$ and $\hat{m}_k$ is $Q_k$-fully linear, and so from \eqref{eq_fully_linear_grad} we have
\be \|Q_k^T \grad f(\bx_k)\| \leq \|Q_k^T \grad f(\bx_k) - \hat{\bg}_k\| + \|\hat{\bg}_k\| \leq \kappa_{\rm eg}\Delta_k + \|\hat{\bg}_k\| \leq (\kappa_{\rm eg}\mu + 1)\|\hat{\bg}_k\|. \label{eq_tmp0} \ee
Since $Q_k$ is well-aligned, we conclude from \eqref{eq_sketched_gradient} and \eqref{eq_tmp0} that
\be \alpha_Q\|\grad f(\bx_k)\| \leq \|Q_k^T \grad f(\bx_k)\| \leq (\kappa_{\rm eg}\mu+1)\|\hat{\bg}_k\|, \ee
and we are done, since $\|\grad f(\bx_k)\|\geq\epsilon$.
\end{proof}

\subsection{Counting iterations}
We now provide a series of results counting the number of iterations of RSDFO of different types, following the style of analysis from \cite{Cartis2018a,Shao2021}. 
First we introduce some notation to enumerate our iterations.
Suppose we run RSDFO until the end of iteration $K$.
We then define the following subsets of $\{0,\ldots,K\}$:
\begin{itemize}
    \item $\mathcal{C}$ is the set of iterations in $\{0,\ldots,K\}$ where the criticality step is called.
    \item $\mathcal{F}$ is the set of iterations in $\{0,\ldots,K\}$, where the safety step is called (i.e.~$\|\hat{\bs}_k\|<\beta_F \Delta_k$).
    \item $\mathcal{VS}$ is the set of very successful iterations in $\{0,\ldots,K\}$, where $\rho_k \geq \eta_2$.
    \item $\mathcal{S}$ is the set of successful iterations in $\{0,\ldots,K\}$, where $\rho_k \geq \eta_1$. Note that $\mathcal{VS}\subset \mathcal{S}$.
    \item $\mathcal{U}$ is the set of unsuccessful iterations in $\{0,\ldots,K\}$, where $\rho_k<\eta_1$. 
    \item $\mathcal{A}$ is the set of well-aligned iterations in $\{0,\ldots,K\}$, where \eqref{eq_sketched_gradient} holds.
    \item $\mathcal{A}^C$ is the set of poorly aligned iterations in $\{0,\ldots,K\}$, where $\eqref{eq_sketched_gradient}$ does not hold.
    \item $\mathcal{D}(\Delta)$ is the set of iterations in $\{0,\ldots,K\}$ where $\Delta_k \geq \Delta$ for some $\Delta>0$.
    \item $\mathcal{D}^C(\Delta)$ is the set of iterations in $\{0,\ldots,K\}$ where $\Delta_k < \Delta$.
    \item $\mathcal{L}$ is the set of iterations in $\{0,\ldots,K\}$ where $\hat{m}_k$ is $Q_k$-fully linear in $B(\bx_k,\Delta_k)$.
    \item $\mathcal{L}^C$ is the set of iterations in $\{0,\ldots,K\}$ where $\hat{m}_k$ is not $Q_k$-fully linear in $B(\bx_k,\Delta_k)$.
\end{itemize}
In particular, we have the partitions, for any $\Delta>0$,
\be \{0,\ldots,K\} = \mathcal{C} \cup \mathcal{F} \cup \mathcal{S} \cup \mathcal{U} = \mathcal{A}\cup\mathcal{A}^C = \mathcal{D}(\Delta)\cup\mathcal{D}^C(\Delta) = \mathcal{L}\cup\mathcal{L}^C. \ee

First, we bound the number of successful iterations with large $\Delta_k$ using standard arguments from trust-region methods.
Throughout, we use $\#(\cdot)$ to refer to the cardinality of a set of iterations.

\begin{lemma} \label{lem_success}
Suppose Assumptions \ref{ass_smoothness}, \ref{ass_bdd_hess} and \ref{ass_cauchy_decrease} hold. 
If $\|\grad f(\bx_k)\| \geq \epsilon$ for all $k=0,\ldots,K$, then
\be \#(\mathcal{A}\cap\mathcal{D}(\Delta)\cap\mathcal{S}) \leq \phi(\Delta, \epsilon) := \frac{f(\bx_0)-f_{\rm low}}{\eta_1 c_1\epsilon_g(\epsilon) \min(\epsilon_g(\epsilon)/\kappa_H, \Delta)}, \ee
for all $\Delta>0$.
\end{lemma}
\begin{proof}
Since $\|\grad f(\bx_k)\|\geq\epsilon$, from \lemref{lem_epsilon_g} we have $\|\hat{\bg}_k\|\geq\epsilon_g(\epsilon)$ for all $k\in\mathcal{A}\cap\mathcal{S}$ (noting that $k\in\mathcal{S}$ implies $k\in\mathcal{C}^C$).
Then since $\rho_k\geq\eta_1$, from Assumptions \ref{ass_bdd_hess} and \ref{ass_cauchy_decrease} we get
\begin{align}
    f(\bx_k) - f(\bx_{k+1}) &\geq \eta_1 [\hat{m}_k(\b{0}) - \hat{m}_k(\hat{\bs}_k)], \\
    &\geq \eta_1 c_1 \|\hat{\bg}_k\|\min\left(\Delta_k, \frac{\|\hat{\bg}_k\|}{\kappa_H}\right), \\
    &\geq \eta_1 c_1 \epsilon_g(\epsilon) \min\left(\Delta, \frac{\epsilon_g(\epsilon)}{\kappa_H}\right),
\end{align}
where the last line follows from $\|\hat{\bg}_k\|\geq\epsilon_g(\epsilon)$ and $\Delta_k\geq\Delta$ (from $k\in\mathcal{D}(\Delta)$).
Since our step acceptance guarantees our algorithm is monotone (i.e.~$f(\bx_{k+1}) \leq f(\bx_k)$ for all $k$), we get
\be f(\bx_0) - f_{\rm low} \geq \sum_{k\in\mathcal{A}\cap\mathcal{D}(\Delta)\cap\mathcal{S}} [f(\bx_k) - f(\bx_{k+1})] \geq \left[\eta_1 c_1 \epsilon_g(\epsilon) \min\left(\Delta, \frac{\epsilon_g(\epsilon)}{\kappa_H}\right)\right] \cdot \#(\mathcal{A}\cap\mathcal{D}(\Delta)\cap\mathcal{S}), \ee
from which the result follows.
\end{proof}

\begin{lemma} \label{lem_small_unsuccessful}
	Suppose Assumptions \ref{ass_smoothness}, \ref{ass_bdd_hess} and \ref{ass_cauchy_decrease} hold, and $\beta_F\leq c_2$.
	If $\|\grad f(\bx_k)\| \geq \epsilon$ for all $k=0,\ldots,K$, then
	\be \#(\mathcal{A}\cap\mathcal{D}^C(\Delta)\cap\mathcal{L}\setminus\mathcal{VS}) = 0, \ee
	\be \Delta \leq \Delta^*(\epsilon) := \min\left(c_0\epsilon_g(\epsilon), \frac{\alpha_Q\epsilon}{\kappa_{\rm eg}+\mu^{-1}}\right). \ee
\end{lemma}
\begin{proof}
	To find a contradiction, first suppose $k\in\mathcal{A}\cap\mathcal{D}^C(\Delta)\cap\mathcal{L}\cap\mathcal{C}^C\setminus\mathcal{VS}$.
	Then since $k\in\mathcal{A}\cap\mathcal{C}^C$ and $\|\grad f(\bx_k)\|\geq\epsilon$ by assumption, we have $\|\hat{\bg}_k\|\geq\epsilon_g(\epsilon)$ from \lemref{lem_epsilon_g}. 
	Since $k\in\mathcal{D}^C(\Delta)$, we have $\Delta_k< \Delta \leq \Delta^*(\epsilon) \leq c_0\epsilon_g(\epsilon) \leq c_0\|\hat{\bg}_k\|$ by definition of $\Delta^*(\epsilon)$.
	From this and $k\in\mathcal{L}$, the assumptions of \lemref{lem_eventually_successful} are met, so $k\notin\mathcal{F}$ and $\rho_k\geq\eta_2$; that is, $k\in\mathcal{VS}$, a contradiction.
	Hence we have $\#(\mathcal{A}\cap\mathcal{D}^C(\Delta)\cap\mathcal{L}\cap\mathcal{C}^C\setminus\mathcal{VS})=0$.
	
	Next, we suppose $k\in\mathcal{A}\cap\mathcal{D}^C(\Delta)\cap\mathcal{L}\cap\mathcal{C}$ and again look for a contradiction.
	In this case, we have $\Delta_k < \Delta \leq \Delta^*(\epsilon) \leq \alpha_Q\epsilon/(\kappa_{\rm eg}+\mu^{-1})$, and so from $k\in\mathcal{A}\cap\mathcal{L}$ and $\|\grad f(\bx_k)\|\geq \epsilon$ we have
	\be \|\hat{\bg}_k\| \geq \|Q_k^T \grad f(\bx_k)\| - \|Q_k^T \grad f(\bx_k) - \hat{\bg}_k\| \geq \alpha_Q\epsilon - \kappa_{\rm eg}\Delta_k > \mu^{-1}\Delta_k. \ee
	This means we have $\|\hat{\bg}_k\| > \mu^{-1}\Delta_k$ and $k\in\mathcal{L}$, so the criticality step is not entered; i.e.~$k\in\mathcal{C}^C$, a contradiction.
	Hence we have $\#(\mathcal{A}\cap\mathcal{D}^C(\Delta)\cap\mathcal{L}\cap\mathcal{C})=0$ and we are done.
\end{proof}

\begin{lemma} \label{lem_large_unsuccessful}
	Suppose Assumptions \ref{ass_smoothness}, \ref{ass_bdd_hess} and \ref{ass_cauchy_decrease} hold.
	Then we have
	\begin{align}
		\#(\mathcal{D}(\max(\gamma_C, \gamma_F, \gamma_{\rm dec})^{-1}\Delta)\setminus\mathcal{S}) &\leq  C_1 \#(\mathcal{D}(\overline{\gamma}_{\rm inc}^{-1}\Delta)\cap\mathcal{S}) + C_2, \label{eq_unsuccessful_bound}
	\end{align}
	for all $\Delta\leq\Delta_0$, where
	\be C_1 := \frac{\log(\overline{\gamma}_{\rm inc})}{\log(1/\max(\gamma_C, \gamma_F, \gamma_{\rm dec}))} \qquad \text{and} \qquad C_2 := \frac{\log(\Delta_0/\Delta)}{\log(1/\max(\gamma_C, \gamma_F, \gamma_{\rm dec}))}. \ee
\end{lemma}
\begin{proof}
If $k\in\mathcal{S}$, we always have $\Delta_{k+1} \leq \overline{\gamma}_{\rm inc}\Delta_k$.
On the other hand, if $k\in\mathcal{U}$, we have
\be \Delta_{k+1} = \min(\gamma_{\rm dec}\Delta_k, \|\hat{\bs}_k\|) \leq \gamma_{\rm dec}\Delta_k, \ee
Hence, 
\be \Delta_{k+1} \leq \max(\gamma_C, \gamma_F, \gamma_{\rm dec})\Delta_k, \qquad \text{for all $k\in\mathcal{C}\cup\mathcal{F}\cup\mathcal{U}$}. \ee
We now consider the value of $\log(\Delta_k)$ for $k=0,\ldots,K$, so at each iteration we have an additive change:
\begin{itemize}
    \item Since $\Delta\leq\Delta_0$, the threshold value $\log(\Delta)$ is $\log(\Delta_0/\Delta)$ below the starting value $\log(\Delta_0)$.
    \item If $k\in\mathcal{S}$, then $\log(\Delta_k)$ increases by at most $\log(\overline{\gamma}_{\rm inc})$. In particular, $\Delta_{k+1}\geq\Delta$ is only possible if $\Delta_k \geq \overline{\gamma}_{\rm inc}^{-1}\Delta$.
    \item If $k\notin\mathcal{S}=\mathcal{C}\cup\mathcal{F}\cup\mathcal{U}$, then $\log(\Delta_k)$ decreases by at least $|\log(\max(\gamma_C, \gamma_F, \gamma_{\rm dec}))| = \log(1/\max(\gamma_C, \gamma_F, \gamma_{\rm dec}))$.
\end{itemize}
Now, any decrease in $\Delta_k$ coming from $k\in\mathcal{D}(\max(\gamma_C, \gamma_F, \gamma_{\rm dec})^{-1}\Delta)\setminus\mathcal{S}$ yields $\Delta_{k+1}\geq\Delta$.
Hence the total decrease in $\log(\Delta_k)$ must be fully matched by the initial gap $\log(\Delta_0/\Delta)$ plus the maximum possible amount that $\log(\Delta_k)$ can be increased above $\log(\Delta)$.
That is, we must have
\begin{align}
	\log(1/\max(\gamma_C, \gamma_F, \gamma_{\rm dec})) &\cdot  \#(\mathcal{D}(\max(\gamma_C, \gamma_F, \gamma_{\rm dec})^{-1}\Delta)\setminus\mathcal{S}) \nonumber \\ 
	&\qquad \leq \log(\Delta_0/\Delta) + \log(\overline{\gamma}_{\rm inc}) \cdot \#(\mathcal{D}(\overline{\gamma}_{\rm inc}^{-1}\Delta)\cap\mathcal{S}),
\end{align}
which gives us \eqref{eq_unsuccessful_bound}.
\end{proof}

\begin{lemma} \label{lem_small_delta_successful_bound}
	Suppose Assumptions \ref{ass_smoothness}, \ref{ass_bdd_hess} and \ref{ass_cauchy_decrease} hold.
	Then
	\be \#(\mathcal{D}^C(\gamma_{\rm inc}^{-1}\Delta)\cap\mathcal{VS}) \leq C_3\cdot \#(\mathcal{D}^C(\min(\gamma_C, \gamma_F, \gamma_{\rm dec}, \beta_F)^{-1}\Delta)\setminus \mathcal{VS}), \ee
	for all $\Delta\leq\min(\Delta_0, \gamma_{\rm inc}^{-1}\Delta_{\max})$, where
	\be C_3 := \frac{\log(1/\min(\gamma_C, \gamma_F, \gamma_{\rm dec}, \beta_F))}{\log(\gamma_{\rm inc})}. \ee
\end{lemma}
\begin{proof}
	We follow a similar reasoning to the proof of \lemref{lem_large_unsuccessful}.
	For every iteration $k\in\mathcal{VS}\cap\mathcal{D}^C(\Delta)$, we increase $\Delta_k$ by a factor of at least $\gamma_{\rm inc}$, since $\Delta_k < \Delta \leq \gamma_{\rm inc}^{-1}\Delta_{\max}$.
	Equivalently, we increase $\log(\Delta_k)$ by at least $\log(\gamma_{\rm inc})$.
	In particular, if $\Delta_k<\gamma_{\rm inc}^{-1}\Delta$, then $\Delta_{k+1}<\Delta$.
	
	Alternatively, if $k\in\mathcal{S}\setminus\mathcal{VS}$, we set
	\be \Delta_{k+1} = \max(\gamma_{\rm dec}\Delta_k, \|\hat{\bs}_k\|) \geq \gamma_{\rm dec}\Delta_k. \ee
	If $k\in\mathcal{U}$ we set
	\be \Delta_{k+1} = \min(\gamma_{\rm dec}\Delta_k, \|\hat{\bs}_k\|) \geq \min(\gamma_{\rm dec}, \beta_F)\Delta_k, \ee
	since $\|\hat{\bs}_k\|\geq\beta_F\Delta_k$ from $k\notin\mathcal{F}$.
	Hence, for every iteration $k\notin\mathcal{VS}$, we decrease $\Delta_k$ by a factor of at most $\min(\gamma_C, \gamma_F, \gamma_{\rm dec}, \beta_F)$, or equivalently we decrease $\log(\Delta_k)$ by at most the amount $|\log(\min(\gamma_C, \gamma_F, \gamma_{\rm dec}, \beta_F))|=\log(1/\min(\gamma_C, \gamma_F, \gamma_{\rm dec}, \beta_F))$.
	Then, to have $\Delta_{k+1}<\Delta$ we require $\Delta_k<\min(\gamma_C, \gamma_F, \gamma_{\rm dec}, \beta_F)^{-1}\Delta$.
	
	Therefore, since $\Delta_0\geq\Delta$, the total increase in $\log(\Delta_k)$ from $k\in\mathcal{VS}\cap\mathcal{D}^C(\gamma_{\rm inc}^{-1}\Delta)$ must be fully matched by the total decrease in $\log(\Delta_k)$ from $k\in \mathcal{D}^C(\min(\gamma_C, \gamma_F, \gamma_{\rm dec}, \beta_F)^{-1}\Delta)\setminus \mathcal{VS}$.
	That is,
	\begin{align}
		\log(\gamma_{\rm inc}) &\#(\mathcal{VS}\cap\mathcal{D}^C(\gamma_{\rm inc}^{-1}\Delta)) \nonumber \\
		&\qquad\qquad \leq \log(1/\min(\gamma_C, \gamma_F, \gamma_{\rm dec}, \beta_F)) \#(\mathcal{D}^C(\min(\gamma_C, \gamma_F, \gamma_{\rm dec}, \beta_F)^{-1}\Delta)\setminus \mathcal{VS}),
	\end{align}
	and we are done.
\end{proof}

\begin{lemma} \label{lem_check_model_bound}
	Suppose Assumptions \ref{ass_smoothness}, \ref{ass_bdd_hess} and \ref{ass_cauchy_decrease} hold.
	Then
	\be \#(\mathcal{A}\cap \mathcal{D}^C(\Delta)\cap\mathcal{L}^C\setminus\mathcal{VS}) \leq \#(\mathcal{A}\cap\mathcal{D}^C(\Delta)\cap\mathcal{L}) + 1, \ee
	for all $\Delta> 0$.
\end{lemma}
\begin{proof}
	After every iteration $k$ where $\hat{m}_k$ is not $Q_k$-fully linear and either the criticality step is called or $\rho_k<\eta_2$, we always set $\Delta_{k+1}\leq \Delta_k$ and \texttt{CHECK\_MODEL}=\texttt{TRUE}.
	This means that $Q_{k+1}=Q_k$, so $Q_{k+1}$ is well-aligned if and only if $Q_k$ is well-aligned.
	Hence if $k\in\mathcal{A}\cap \mathcal{D}^C(\Delta)\cap\mathcal{L}^C\setminus\mathcal{VS}$ then either $k=K$ or $k+1\in \mathcal{A}\cap\mathcal{D}^C(\Delta)\cap\mathcal{L}$, and we are done.
\end{proof}

We are now in a position to bound the total number of well-aligned iterations.

\begin{lemma} \label{lem_aligned_bound}
	Suppose Assumptions \ref{ass_smoothness}, \ref{ass_bdd_hess} and \ref{ass_cauchy_decrease} hold, and both $\beta_F\leq c_2$ and $\gamma_{\rm inc} > \min(\gamma_C, \gamma_F, \gamma_{\rm dec}, \beta_F)^{-2}$ hold.
	Then if $\|\grad f(\bx_k)\| \geq \epsilon$ for all $k=0,\ldots,K$, we have 
	\begin{align}
		\#(\mathcal{A}) &\leq \psi(\epsilon) + \frac{C_4}{1+C_4} (K+1),
	\end{align}
	where
	\begin{align}
		\psi(\epsilon) &\defeq \frac{1}{1+C_4}\left[(C_1+2)\phi(\Delta_{\min}(\epsilon), \epsilon) + \frac{4\phi(\gamma_{\rm inc}^{-1}\min(\gamma_C, \gamma_F, \gamma_{\rm dec}, \beta_F)\Delta_{\min}(\epsilon), \epsilon)}{1-2C_3} \right. \\
		&\qquad\qquad\qquad\qquad \left. + C_2 + \frac{2}{1-2C_3} + 1\right], \\
		\Delta_{\min}(\epsilon) &\defeq \min\left(\overline{\gamma}_{\rm inc}^{-1}\Delta_0, \min(\gamma_C, \gamma_F, \gamma_{\rm dec})\overline{\gamma}_{\rm inc}^{-1}\Delta^*(\epsilon)\right), \\
		C_4 &\defeq \max\left(C_1, \frac{4C_3}{1-2C_3}\right) > 0.
	\end{align}
	In these expressions, the values $C_1$ and $C_2$ are defined in \lemref{lem_large_unsuccessful}, $C_3$ is defined in \lemref{lem_small_delta_successful_bound}, $\phi(\cdot,\epsilon)$ is defined in \lemref{lem_success}, and $\epsilon_g(\epsilon)$ and $\Delta^*(\epsilon)$ are defined in Lemmas \ref{lem_epsilon_g} and \ref{lem_small_unsuccessful} respectively.
\end{lemma}
\begin{proof}
    For ease of notation, we will write $\Delta_{\min}$ in place of $\Delta_{\min}(\epsilon)$.
    We begin by noting that $\gamma_{\rm inc} > \min(\gamma_C, \gamma_F, \gamma_{\rm dec}, \beta_F)^{-2}$ implies that $C_3\in(0,1/2)$, which we will use later.
    
    Next, we have
	\begin{align}
		\#(\mathcal{A}\cap\mathcal{D}(\Delta_{\min})) &= \#(\mathcal{A}\cap\mathcal{D}(\Delta_{\min})\cap\mathcal{S}) + \#(\mathcal{A}\cap\mathcal{D}(\Delta_{\min})\setminus\mathcal{S}), \\
		&\leq \phi(\Delta_{\min}, \epsilon) + \#(\mathcal{A}\cap\mathcal{D}(\max(\gamma_C, \gamma_F, \gamma_{\rm dec})^{-1}\overline{\gamma}_{\rm inc}\Delta_{\min})\setminus\mathcal{S}) \nonumber \\
		&\qquad + \#(\mathcal{A}\cap\mathcal{D}(\Delta_{\min})\cap\mathcal{D}^C(\max(\gamma_C, \gamma_F, \gamma_{\rm dec})^{-1}\overline{\gamma}_{\rm inc}\Delta_{\min})\setminus\mathcal{S}), \\
		&\leq \phi(\Delta_{\min}, \epsilon) + C_1 \#(\mathcal{D}(\Delta_{\min})\cap\mathcal{S}) + C_2 \nonumber \\
		&\qquad + \#(\mathcal{A}\cap\mathcal{D}^C(\max(\gamma_C, \gamma_F, \gamma_{\rm dec})^{-1}\overline{\gamma}_{\rm inc}\Delta_{\min})\setminus\mathcal{S}), \\
		&= \phi(\Delta_{\min}, \epsilon) + C_1 \#(\mathcal{D}(\Delta_{\min})\cap\mathcal{S}) + C_2 \nonumber \\
		&\qquad + \#(\mathcal{A}\cap\mathcal{D}^C(\max(\gamma_C, \gamma_F, \gamma_{\rm dec})^{-1}\overline{\gamma}_{\rm inc}\Delta_{\min})\cap\mathcal{L}\setminus\mathcal{S}) \nonumber \\
		&\qquad + \#(\mathcal{A}\cap\mathcal{D}^C(\max(\gamma_C, \gamma_F, \gamma_{\rm dec})^{-1}\overline{\gamma}_{\rm inc}\Delta_{\min})\cap\mathcal{L}^C\setminus\mathcal{S}), \\
		&\leq \phi(\Delta_{\min}, \epsilon) + C_1 \#(\mathcal{D}(\Delta_{\min})\cap\mathcal{S}) + C_2 \nonumber \\
		&\qquad + 2 \#(\mathcal{A}\cap\mathcal{D}^C(\max(\gamma_C, \gamma_F, \gamma_{\rm dec})^{-1}\overline{\gamma}_{\rm inc}\Delta_{\min})\cap\mathcal{L}\setminus\mathcal{S}) \nonumber \\
		&\qquad + \#(\mathcal{A}\cap\mathcal{D}^C(\max(\gamma_C, \gamma_F, \gamma_{\rm dec})^{-1}\overline{\gamma}_{\rm inc}\Delta_{\min})\cap\mathcal{L}\cap\mathcal{S}) + 1,
	\end{align}
	where the first inequality follows from \lemref{lem_success}, the second inequality follows from \lemref{lem_large_unsuccessful} and $\Delta_{\min}\leq\overline{\gamma}_{\rm inc}^{-1}\Delta_0$, and the last line follows from \lemref{lem_check_model_bound} and $\mathcal{VS}\subset\mathcal{S}$.
	Now we use \lemref{lem_small_unsuccessful} with $\Delta_{\min}\leq\max(\gamma_C, \gamma_F, \gamma_{\rm dec})\overline{\gamma}_{\rm inc}^{-1}\Delta^*(\epsilon)$ to get
	\begin{align}
		\#(\mathcal{A}\cap\mathcal{D}(\Delta_{\min})) &\leq \phi(\Delta_{\min}, \epsilon) + C_1 \#(\mathcal{D}(\Delta_{\min})\cap\mathcal{S}) + C_2 \nonumber \\
		&\qquad + \#(\mathcal{A}\cap\mathcal{D}^C(\max(\gamma_C, \gamma_F, \gamma_{\rm dec})^{-1}\overline{\gamma}_{\rm inc}\Delta_{\min})\cap\mathcal{L}\cap\mathcal{S}) + 1, \\
		&= \phi(\Delta_{\min}, \epsilon) + C_1 \#(\mathcal{A}\cap\mathcal{D}(\Delta_{\min})\cap\mathcal{S}) + C_1 \#(\mathcal{A}^C\cap\mathcal{D}(\Delta_{\min})\cap\mathcal{S}) + C_2 \nonumber \\
		&\qquad + \#(\mathcal{A}\cap\mathcal{D}^C(\Delta_{\min})\cap\mathcal{L}\cap\mathcal{S}) \nonumber \\
		&\qquad + \#(\mathcal{A}\cap\mathcal{D}(\Delta_{\min})\cap\mathcal{D}^C(\max(\gamma_C, \gamma_F, \gamma_{\rm dec})^{-1}\overline{\gamma}_{\rm inc}\Delta_{\min})\cap\mathcal{L}\cap\mathcal{S}) + 1, \\
		&\leq \phi(\Delta_{\min}, \epsilon) + C_1 \#(\mathcal{A}\cap\mathcal{D}(\Delta_{\min})\cap\mathcal{S}) + C_1 \#(\mathcal{A}^C\cap\mathcal{D}(\Delta_{\min})\cap\mathcal{S}) + C_2 \nonumber \\
		&\qquad + \#(\mathcal{A}\cap\mathcal{D}^C(\Delta_{\min})) + \#(\mathcal{A}\cap\mathcal{D}(\Delta_{\min})\cap\mathcal{S}) + 1, \\
		&\leq (C_1+2)\phi(\Delta_{\min}, \epsilon) + C_1 \#(\mathcal{A}^C\cap\mathcal{D}(\Delta_{\min})\cap\mathcal{S}) + C_2 \nonumber \\
		&\qquad + \#(\mathcal{A}\cap\mathcal{D}^C(\Delta_{\min})) + 1, \label{eq_tmp15}
	\end{align}
	where the last line follows from \lemref{lem_success}.
	
	Separately, we use \lemref{lem_check_model_bound}, and apply \lemref{lem_small_unsuccessful} with $\Delta_{\min}\leq\Delta^*(\epsilon)$ to get
	\begin{align}
		\#(\mathcal{A}\cap\mathcal{D}^C(\Delta_{\min})) &= \#(\mathcal{A}\cap\mathcal{D}^C(\Delta_{\min})\cap\mathcal{VS}) + \#(\mathcal{A}\cap\mathcal{D}^C(\Delta_{\min})\cap\mathcal{L}\setminus\mathcal{VS}) \nonumber \\
		&\qquad + \#(\mathcal{A}\cap\mathcal{D}^C(\Delta_{\min})\cap\mathcal{L}^C\setminus\mathcal{VS}), \\
		&\leq \#(\mathcal{A}\cap\mathcal{D}^C(\Delta_{\min})\cap\mathcal{VS}) + \#(\mathcal{A}\cap\mathcal{D}^C(\Delta_{\min})\cap\mathcal{L}\setminus\mathcal{VS}) \nonumber \\
		&\qquad + \#(\mathcal{A}\cap\mathcal{D}^C(\Delta_{\min})\cap\mathcal{L}) + 1, \\
		&= \#(\mathcal{A}\cap\mathcal{D}^C(\Delta_{\min})\cap\mathcal{VS}) + 2 \#(\mathcal{A}\cap\mathcal{D}^C(\Delta_{\min})\cap\mathcal{L}\setminus\mathcal{VS}) \nonumber \\
		&\qquad + \#(\mathcal{A}\cap\mathcal{D}^C(\Delta_{\min})\cap\mathcal{L}\cap\mathcal{VS}) + 1, \\
		&= \#(\mathcal{A}\cap\mathcal{D}^C(\Delta_{\min})\cap\mathcal{VS}) + \#(\mathcal{A}\cap\mathcal{D}^C(\Delta_{\min})\cap\mathcal{L}\cap\mathcal{VS}) + 1, \\
		&\leq 2\#(\mathcal{A}\cap\mathcal{D}^C(\Delta_{\min})\cap\mathcal{VS}) + 1.
	\end{align}
	We then get
	\begin{align}
		\#(\mathcal{A}\cap\mathcal{D}^C(\Delta_{\min})) &\leq 2\#(\mathcal{A}\cap\mathcal{D}^C(\gamma_{\rm inc}^{-1}\min(\gamma_C, \gamma_F, \gamma_{\rm dec}, \beta_F)\Delta_{\min})\cap\mathcal{VS}) \nonumber \\
		&\qquad + 2\#(\mathcal{A}\cap\mathcal{D}(\gamma_{\rm inc}^{-1}\min(\gamma_C, \gamma_F, \gamma_{\rm dec}, \beta_F)\Delta_{\min})\cap\mathcal{D}^C(\Delta_{\min})\cap\mathcal{VS}) + 1, \\
		&\leq 2\#(\mathcal{D}^C(\gamma_{\rm inc}^{-1}\min(\gamma_C, \gamma_F, \gamma_{\rm dec}, \beta_F)\Delta_{\min})\cap\mathcal{VS}) \nonumber \\
		&\qquad + 2\#(\mathcal{A}\cap\mathcal{D}(\gamma_{\rm inc}^{-1}\min(\gamma_C, \gamma_F, \gamma_{\rm dec}, \beta_F)\Delta_{\min})\cap\mathcal{VS}) + 1, \\
		&\leq 2 C_3 \#(\mathcal{D}^C(\Delta_{\min})\setminus\mathcal{VS}) + 2\phi(\gamma_{\rm inc}^{-1}\min(\gamma_C, \gamma_F, \gamma_{\rm dec}, \beta_F)\Delta_{\min}, \epsilon) + 1, \\
		&= 2 C_3 \#(\mathcal{A}\cap\mathcal{D}^C(\Delta_{\min})\setminus\mathcal{VS}) + 2 C_3 \#(\mathcal{A}^C\cap\mathcal{D}^C(\Delta_{\min})\setminus\mathcal{VS}) \nonumber \\
		&\qquad + 2\phi(\gamma_{\rm inc}^{-1}\min(\gamma_C, \gamma_F, \gamma_{\rm dec}, \beta_F)\Delta_{\min}, \epsilon) + 1, \\
		&\leq 2 C_3 \#(\mathcal{A}\cap\mathcal{D}^C(\Delta_{\min})) + 2 C_3 \#(\mathcal{A}^C\cap\mathcal{D}^C(\Delta_{\min})\setminus\mathcal{VS}) \nonumber \\
		&\qquad + 2\phi(\gamma_{\rm inc}^{-1}\min(\gamma_C, \gamma_F, \gamma_{\rm dec}, \beta_F)\Delta_{\min}, \epsilon) + 1, \label{eq_tmp14}
	\end{align}
	where the third inequality follows from \lemref{lem_success} and \lemref{lem_small_delta_successful_bound} with 
	\be \Delta_{\min} \leq \overline{\gamma}_{\rm inc}^{-1}\Delta_0 \leq \gamma_{\rm inc}^{-1}\Delta_0 \leq \min(\Delta_0, \gamma_{\rm inc}^{-1}\Delta_{\max}) \leq \min(\gamma_C, \gamma_F, \gamma_{\rm dec}, \beta_F)^{-1}\min(\Delta_0, \gamma_{\rm inc}^{-1}\Delta_{\max}). \ee
	Since $C_3\in(0,1/2)$, we can rearrange \eqref{eq_tmp14} to conclude that
	\begin{align}
		&\#(\mathcal{A}\cap\mathcal{D}^C(\Delta_{\min})) \nonumber \\
		&\qquad\qquad \leq \frac{1}{1-2C_3}\left[2 C_3 \#(\mathcal{A}^C\cap\mathcal{D}^C(\Delta_{\min})\setminus\mathcal{VS}) + 2\phi(\gamma_{\rm inc}^{-1}\min(\gamma_C, \gamma_F, \gamma_{\rm dec}, \beta_F)\Delta_{\min}, \epsilon) + 1\right]. \label{eq_tmp16}
	\end{align}
	Now, we combine \eqref{eq_tmp15} and \eqref{eq_tmp16} to get
	\begin{align}
		\#(\mathcal{A}) &= \#(\mathcal{A}\cap\mathcal{D}(\Delta_{\min})) + \#(\mathcal{A}\cap\mathcal{D}^C(\Delta_{\min})), \\
		&\leq (C_1+2)\phi(\Delta_{\min}, \epsilon) + C_1 \#(\mathcal{A}^C\cap\mathcal{D}(\Delta_{\min})\cap\mathcal{S}) + C_2  + 2\#(\mathcal{A}\cap\mathcal{D}^C(\Delta_{\min})) + 1, \\
		&\leq (C_1+2)\phi(\Delta_{\min}, \epsilon) + \frac{4\phi(\gamma_{\rm inc}^{-1}\min(\gamma_C, \gamma_F, \gamma_{\rm dec}, \beta_F)\Delta_{\min}, \epsilon)}{1-2C_3} +  C_1 \#(\mathcal{A}^C\cap\mathcal{D}(\Delta_{\min})\cap\mathcal{S}) \nonumber \\
		&\qquad + C_2 + \frac{2}{1-2C_3} + 1 + \frac{4C_3}{1-2C_3} \#(\mathcal{A}^C\cap\mathcal{D}^C(\Delta_{\min})\setminus\mathcal{VS}), \\
		&\leq (C_1+2)\phi(\Delta_{\min}, \epsilon) + \frac{4\phi(\gamma_{\rm inc}^{-1}\min(\gamma_C, \gamma_F, \gamma_{\rm dec}, \beta_F)\Delta_{\min}, \epsilon)}{1-2C_3} + C_2 + \frac{2}{1-2C_3} + 1 \nonumber \\
		&\qquad + \max\left(C_1, \frac{4C_3}{1-2C_3}\right)\left[\#(\mathcal{A}^C\cap\mathcal{D}(\Delta_{\min})\cap\mathcal{S}) + \#(\mathcal{A}^C\cap\mathcal{D}^C(\Delta_{\min})\setminus\mathcal{VS})\right]. \label{eq_tmp17}
	\end{align}
	Since $\mathcal{A}^C\cap\mathcal{D}(\Delta_{\min})\cap\mathcal{S}$ and $\mathcal{A}^C\cap\mathcal{D}^C(\Delta_{\min})\setminus\mathcal{VS}$ are disjoint subsets of $\mathcal{A}^C$, we have
	\be \#(\mathcal{A}^C\cap\mathcal{D}(\Delta_{\min})\cap\mathcal{S}) + \#(\mathcal{A}^C\cap\mathcal{D}^C(\Delta_{\min})\setminus\mathcal{VS}) \leq \#(\mathcal{A}^C) = (K+1) - \#(\mathcal{A}). \ee
	Substituting this into \eqref{eq_tmp17} and rearranging, we get the desired result.
	That $C_4>0$ follows from $C_1>0$ and $C_3\in(0,1/2)$.
\end{proof}

\subsection{Overall Complexity Bound}
The key remaining step is to compare $\#(\mathcal{A})$ with $K$.
Since each event ``$Q_k$ is well aligned'' is effectively an independent Bernoulli trial with success probability at least $1-\delta_S$, we derive the below result based on a concentration bound for Bernoulli trials \cite[Lemma 2.1]{Chung2002}.

\begin{lemma} \label{lem_chernoff}
Suppose Assumptions \ref{ass_smoothness}, \ref{ass_bdd_hess}, \ref{ass_cauchy_decrease} and \ref{ass_sketching} hold.
Then we have
\be \Prob{\#(\mathcal{A}) + 1 \leq (1-\delta_S)(1-\delta)(K+1)} \leq e^{-\delta^2 (1-\delta_S)K / 4}, \ee
for all $\delta\in(0,1)$.
\end{lemma}
\begin{proof}
The \texttt{CHECK\_MODEL}=\texttt{FALSE} case of this proof has a general framework based on \cite[Lemma 4.5]{Gratton2015}---also followed in \cite{Shao2021}---with a probabilistic argument from \cite[Lemma 2.1]{Chung2002}.

First, we consider only the subsequence of iterations $\mathcal{K}_1\defeq\{k_0,\ldots,k_J\}\subset\{0,\ldots,K\}$ when $Q_k$ is resampled (i.e.~where \texttt{CHECK\_MODEL}=\texttt{FALSE}, so $Q_k\neq Q_{k-1}$).
For convenience, we define $\mathcal{A}_1 \defeq \mathcal{A}\cap\mathcal{K}_1$ and $\mathcal{A}^C_1 \defeq \mathcal{A}^C\cap\mathcal{K}_1$.

Let $T_{k_j}$ be the indicator function for the event ``$Q_{k_j}$ is well-aligned'', and so
$\#(\mathcal{A}_1) = \sum_{j=0}^{J}T_{k_j}$.
Since $T_{k_j}\in\{0,1\}$, and denoting $p_{k_j}:=\Prob{T_{k_j}=1 |\, \bx_{k_j}}$, for any $t>0$ we have 
\be \E{e^{-t (T_{k_j}-p_{k_j})} |\, \bx_{k_j}} = p_{k_j} e^{-t (1-p_{k_j})} + (1-p_{k_j})e^{t p_{k_j}} = e^{t p_{k_j} + \log(1-p_{k_j}+p_{k_j} e^{-t})} \leq e^{t^2 p_{k_j} / 2}, \label{eq_tmp6} \ee
where the inequality from the identity
$ px + \log(1-p+pe^{-x}) \leq px^2/2$,
for all $p\in[0,1]$ and $x\geq 0$, shown in \cite[Lemma 2.1]{Chung2002}.

Using the tower property of conditional expectations and the fact that, since $k_j\in\mathcal{K}_1$, $T_{k_j}$ only depends on $\bx_{k_j}$ and not any previous iteration, we then get
\begin{align}
    \E{e^{-t (\#(\mathcal{A}_1)-\sum_{j=0}^{J} p_{k_j})}} &= \E{e^{-t \sum_{j=0}^{J} (T_{k_j} - p_{k_j})}}, \\
    &= \E{\E{ e^{-t \sum_{j=0}^{J}(T_{k_j} - p_{k_j})} | \, Q_0,\ldots,Q_{k_J-1},\bx_0,\ldots,\bx_{k_J} }}, \\
    &= \E{ e^{-t \sum_{j=0}^{J-1} (T_{k_j} - p_{k_j})} \E{e^{-t (T_{k_J} - p_{k_J})} | \, Q_0,\ldots,Q_{k_J-1},\bx_0,\ldots,\bx_{k_J}} }, \\
    &= \E{ e^{-t \sum_{j=0}^{J-1} (T_{k_j} - p_{k_j})} \E{e^{-t (T_{k_J} - p_{k_J})} | \, \bx_{k_J}} }, \\
    &\leq e^{t^2 p_{k_J}/2} \E{ e^{-t \sum_{j=0}^{J-1} (T_{k_j} - p_{k_j})}}, \\
    &\leq e^{t^2 (\sum_{j=0}^{J} p_{k_j}) /2}, \label{eq_tmp7}
\end{align}
where the second-last line follows from \eqref{eq_tmp6} and the last line follows by induction.
This means that
\begin{align}
    \Prob{\#(\mathcal{A}_1) \leq \sum_{j=0}^{J} p_{k_j} - \lambda} &= \Prob{ e^{-t \left(\#(\mathcal{A}_1)-\sum_{j=0}^{J} p_{k_j}\right)} > e^{t\lambda} }, \\
    &\leq e^{-t\lambda} \E{e^{-t \left(\#(\mathcal{A}_1)-\sum_{j=0}^{J} p_{k_j}\right)}}, \\
    &\leq e^{t^2 \left(\sum_{j=0}^{J} p_{k_j}\right) /2 - t\lambda},
\end{align}
where the inequalities follow from Markov's inequality and \eqref{eq_tmp7} respectively.
Taking $t=\lambda / \sum_{j=0}^{J} p_{k_j}$, we get
\be \Prob{\#(\mathcal{A}_1) \leq \sum_{j=0}^{J} p_{k_j} - \lambda} \leq e^{-\lambda^2 / \left(2\sum_{j=0}^{J} p_{k_j}\right)}. \ee
Finally, we take $\lambda=\delta \sum_{j=0}^{J} p_{k_j}$ for some $\delta\in(0,1)$ and note that $p_{k_j}\geq (1-\delta_S)$ (from \assref{ass_sketching}), to conclude
\be \Prob{\#(\mathcal{A}_1) \leq (1-\delta)(1-\delta_S)(J+1)} \leq \Prob{\#(\mathcal{A}_1) \leq (1-\delta)\sum_{j=0}^{J} p_{k_j}} \leq e^{-\delta^2 \left(\sum_{j=0}^{J} p_{k_j}\right)/2}, \ee 
or equivalently, using the partition $\mathcal{K}_1 = \mathcal{A}_1 \cup \mathcal{A}_1^C$,
\be \Prob{\#(\mathcal{A}_1) \leq (1-\delta)(1-\delta_S)[\#(\mathcal{A}_1) + \#(\mathcal{A}_1^C)]} \leq e^{-\delta^2 (1-\delta_S)[\#(\mathcal{A}_1) + \#(\mathcal{A}_1^C)]/2}. \label{eq_tmp18} \ee

Now we must consider the iterations for which \texttt{CHECK\_MODEL}=\texttt{TRUE} (so $Q_k=Q_{k-1}$), which we denote $\mathcal{K}_1^C$.
The algorithm ensures that if $k\in\mathcal{K}_1^C$, then $k+1\in\mathcal{K}_1$ (unless we are in the last iteration we consider, $k=K$).
Futher, the algorithm guarantees that if $k\in\mathcal{K}_1^C$, then $k>0$ and $k\in\mathcal{A}$ if and only if $k-1\in\mathcal{A}$.
These are the key implications of RSDFO that we will now use.

Firstly, we have $\#(\mathcal{K}_1^C) \leq \#(\mathcal{K}_1)+1$, and so 
\be K+1 = \#(\mathcal{K}_1) + \#(\mathcal{K}_1^C) \leq 2[\#(\mathcal{A}_1) + \#(\mathcal{A}_1^C)] + 1, \ee
which means \eqref{eq_tmp18} becomes
\be \Prob{\#(\mathcal{A}_1) \leq (1-\delta)(1-\delta_S)[\#(\mathcal{A}_1) + \#(\mathcal{A}_1^C)]} \leq e^{-\delta^2 (1-\delta_S)K/4}. \ee
Setting $\alpha \defeq \delta + \delta_S + \delta\delta_S$, we have $(1-\delta)(1-\delta_S) = 1-\alpha$, and so
\be \Prob{\#(\mathcal{A}_1) \leq \frac{1-\alpha}{\alpha} \#(\mathcal{A}_1^C)} = \Prob{\#(\mathcal{A}_1) \leq (1-\alpha)[\#(\mathcal{A}_1) + \#(\mathcal{A}_1^C)]} \leq e^{-\delta^2 (1-\delta_S)K/4}. \ee
Secondly, we have $\#(\mathcal{K}_1^C \cap\mathcal{A}^C) \leq \#(\mathcal{A}_1^C)+1$, and so $\#(\mathcal{A}^C) \leq 2\#(\mathcal{A}_1^C)+1$.
This and $\mathcal{A}_1\subset\mathcal{A}$ give
\be \Prob{\#(\mathcal{A}) \leq \frac{1-\alpha}{2\alpha} [\#(\mathcal{A}^C)-1]} \leq e^{-\delta^2 (1-\delta_S)K/4}. \ee
We then note that $K+1=\#(\mathcal{A})+\#(\mathcal{A}^C)$, and so
\begin{align}
	\Prob{\#(\mathcal{A}) \leq \frac{1-\alpha}{2\alpha} [K+1-\#(\mathcal{A})-1]} \leq e^{-\delta^2 (1-\delta_S)K/4}, \\
	\Prob{\#(\mathcal{A}) + \frac{1-\alpha}{1+\alpha} \leq \frac{1-\alpha}{1+\alpha} (K+1)} \leq e^{-\delta^2 (1-\delta_S)K/4}, \\
	\Prob{\#(\mathcal{A}) + 1 \leq (1-\alpha)(K+1)} \leq e^{-\delta^2 (1-\delta_S)K/4},
\end{align}
since $\alpha>0$.
\end{proof}

\begin{theorem} \label{thm_high_prob_complexity}
Suppose Assumptions \ref{ass_smoothness}, \ref{ass_bdd_hess}, \ref{ass_cauchy_decrease} and \ref{ass_sketching} hold, and we have $\beta_F\leq c_2$, $\delta_S < 1/(1+C_4)$ for $C_4$ defined in \lemref{lem_aligned_bound}, and $\gamma_{\rm inc} > \min(\gamma_C, \gamma_F, \gamma_{\rm dec}, \beta_F)^{-2}$.
Then for any $\epsilon>0$ and
\be k \geq \frac{2(\psi(\epsilon)+1)}{1-\delta_S-C_4/(1+C_4)}, \label{eq_high_prob_k} \ee
we have
\be \Prob{\min_{j\leq k} \|\grad f(\bx_j)\| \leq \epsilon} \geq 1 - \exp\left(-k\frac{(1-\delta_S-C_4/(1+C_4))^2}{16(1-\delta_S)}\right). \label{eq_high_prob1} \ee
Alternatively, if $K_{\epsilon} \defeq \min\{k : \|\grad f(\bx_k)\|\leq\epsilon\}$ for any $\epsilon>0$, then
\be \Prob{K_{\epsilon} \leq \left\lceil \frac{2(\psi(\epsilon)+1)}{1-\delta_S-C_4/(1+C_4)}  \right\rceil} \geq 1 - \exp\left(-\frac{(\psi(\epsilon)+1)[1-\delta_S-C_4/(1+C_4)]}{8(1-\delta_S)}\right), \label{eq_high_prob2} \ee
where $\psi(\epsilon)$ is defined in \lemref{lem_aligned_bound}.
\end{theorem}
\begin{proof}
First, fix some arbitrary $k\geq 0$.
Let $\epsilon_k := \min_{j\leq k}\|\grad f(\bx_j)\|$ and $A_k$ be the number of well-aligned iterations in $\{0,\ldots,k\}$.
If $\epsilon_k>0$, from \lemref{lem_aligned_bound}, we have
\be A_k \leq \psi(\epsilon_k) + \frac{C_4}{1+C_4} (k+1). \ee
For any $\delta>0$ such that
\be \delta < 1 - \frac{C_4}{(1+C_4)(1-\delta_S)}, \ee
we have $(1-\delta_S)(1-\delta) > C_4 / (1+C_4)$, and so we can compute
\begin{align}
    \Prob{\psi(\epsilon_k) \leq \left[(1-\delta_S)(1-\delta)-\frac{C_4}{1+C_4}\right](k+1) - 1} &\leq \Prob{A_k \leq (1-\delta_S)(1-\delta)(k+1)}, \\
    &\leq e^{-\delta^2 (1-\delta_S)k/4},
\end{align}
using \lemref{lem_chernoff}.
Defining
\be \delta := \frac{1}{2}\left[1 - \frac{C_4}{(1+C_4)(1-\delta_S)}\right], \ee
we have
\be (1-\delta_S)(1-\delta) = \frac{1}{2}\left[1-\delta_S + \frac{C_4}{1+C_4}\right] > \frac{C_4}{1+C_4}, \ee
since $1-\delta_S > C_4 / (1+C_4)$ from our assumption on $\delta_S$.
Hence we get
\be \Prob{\psi(\epsilon_k) \leq \frac{1}{2}\left(1-\delta_S-\frac{C_4}{1+C_4}\right)(k+1) - 1} \leq e^{-k [1-\delta_S-C_4/(1+C_4)]^2 / \left[16(1-\delta_S)\right]}, \ee
and we note that this result is still holds if $\epsilon_k=0$, as $\lim_{\epsilon\to 0}\psi(\epsilon)=\infty$.

Now we fix $\epsilon>0$ and choose $k$ satisfying \eqref{eq_high_prob_k}.
We use the fact that $\psi(\cdot)$ is non-increasing to get
\begin{align}
    \Prob{\epsilon_k \geq \epsilon} &\leq \Prob{\psi(\epsilon_k) \leq \psi(\epsilon)}, \\
    &\leq \Prob{\psi(\epsilon_k) \leq \frac{1}{2}(1-\delta_S-C_4/(1+C_4))k - 1}, \\
    &\leq \Prob{\psi(\epsilon_k) \leq \frac{1}{2}(1-\delta_S-C_4/(1+C_4))(k+1) - 1}, 
\end{align}
and \eqref{eq_high_prob1} follows.
Lastly, we fix
\be k = \left\lceil \frac{2(\psi(\epsilon)+1)}{1-\delta_S-C_4/(1+C_4)} \right\rceil, \ee
and we use \eqref{eq_high_prob1} and the definition of $K_{\epsilon}$ to get
\begin{align}
	\Prob{K_{\epsilon} \geq k} &= \Prob{\epsilon_k \geq \epsilon}, \\
	&\leq e^{-k[1-\delta_S-C_4/(1+C_4)]^2 / [16(1-\delta_S)]}, \\
	&\leq \exp\left(-\frac{(\psi(\epsilon)+1)[1-\delta_S-C_4/(1+C_4)]}{8(1-\delta_S)}\right),
\end{align}
and we get \eqref{eq_high_prob2}.
\end{proof}

\begin{corollary} \label{cor_high_prob_complexity}
	Suppose the assumptions of \thmref{thm_high_prob_complexity} hold.
	Then for $k \geq k_0$ for some $k_0$, we have
	\be \Prob{\min_{j\leq k} \|\grad f(\bx_j)\| \leq \frac{C}{\sqrt{k}}} \geq 1 - e^{-ck}, \label{eq_high_prob3} \ee
	for some constants $c,C>0$.
	Alternatively, for $\epsilon\in(0,\epsilon_0)$ for some $\epsilon_0$, we have
	\be \Prob{K_{\epsilon} \leq \t{C}\epsilon^{-2}} \geq 1 - e^{-\t{c}\epsilon^{-2}}, \label{eq_high_prob4} \ee
	for constants $\t{c},\t{C}>0$.
\end{corollary}
\begin{proof}
	For $\epsilon$ sufficiently small, both $\epsilon_g(\epsilon)$ and $\Delta_{\min}(\epsilon)$ are equal to a multiple of $\epsilon$, and so $\psi(\epsilon)=\alpha_1 \epsilon^{-2} + \alpha_2=\Theta(\epsilon^{-2})$, for some constants $\alpha_1,\alpha_2>0$.
	
	Therefore for $k$ sufficiently large, the choice
	\be \epsilon = \sqrt{\frac{2\alpha_1}{(1-\delta_S-C_4/(1+C_4))k - 2-2\alpha_2}} = \Theta(k^{-1/2}), \ee
	is sufficiently small that $\psi(\epsilon)=\alpha_1 \epsilon^{-2} + \alpha_2$, and gives \eqref{eq_high_prob_k} with equality.
	The first result then follows from \eqref{eq_high_prob1}.
	
	The second result follows immediately from $\psi(\epsilon)=\Theta(\epsilon^{-2})$ and \eqref{eq_high_prob2}.
\end{proof}

\begin{remark}
	All the above analysis holds with minimal modifications if we replace the trust-region mechanisms in RSDFO with more standard trust-region updating mechanisms.
	This includes, for example, having no safety step (i.e.~$\beta_F=0$), and replacing \eqref{eq_dfbgn_tr_updating_generic} with
	\be \bx_{k+1} = \begin{cases}\bx_k + \bs_k, & \rho_k \geq \eta, \\ \bx_k, & \rho_k < \eta, \end{cases} \quad \text{and} \quad \Delta_{k+1} = \begin{cases}\min(\gamma_{\rm inc}\Delta_k, \Delta_{\rm max}), & \rho_k \geq \eta, \\ \gamma_{\rm dec}\Delta_k, & \rho_k < \eta, \end{cases} \ee
	for some $\eta\in(0,1)$.
	The corresponding requirement on the trust-region updating parameters to prove a version of \thmref{thm_high_prob_complexity} is simply $\gamma_{\rm inc} > \gamma_{\rm dec}^{-2}$ (provided we also set $\gamma_C=\gamma_{\rm dec}$).
\end{remark}

\subsection{Remarks on Complexity Bound}
Our final complexity bounds for RSDFO in \corref{cor_high_prob_complexity} are comparable to probabilistic direct search \cite[Corollary 4.9]{Gratton2015}.
They also match---in the order of $\epsilon$---the standard bounds for (full space) model-based DFO methods for general objective \cite{Vicente2013,Garmanjani2016} and nonlinear least-squares \cite{Cartis2019} problems.

Following \cite{Gratton2015}, we may also derive complexity bounds on the expected first-order optimality measure (of $\bigO(k^{-1/2})$) and the expected worst-case complexity (of $\bigO(\epsilon^{-2})$ iterations) for RSDFO.

\begin{theorem}
	Suppose the assumptions of \thmref{thm_high_prob_complexity} hold.
	Then for $k\geq k_0$, the iterates of RSDFO satisfy 
	\be \E{\min_{j\leq k} \|\grad f(\bx_j)\|} \leq Ck^{-1/2} + \|\grad f(\bx_0)\| e^{-ck}, \ee
	for $c,C>0$ from \eqref{eq_high_prob3}, and for $\epsilon\in(0,\epsilon_0)$ we have
	\be \E{K_{\epsilon}} \leq \t{C}_1\epsilon^{-2} + \frac{1}{\t{c}_1}, \ee
	for constants $\t{c}_1,\t{C}_1>0$.
	Here, $k_0$ and $\epsilon_0$ are the same as in \corref{cor_high_prob_complexity}.
\end{theorem}
\begin{proof}
	First, for $k\geq k_0$ define the random variable $H_k$ as
	\be H_k \defeq \begin{cases} Ck^{-1/2}, & \text{if $\min_{j\leq k} \|\grad f(\bx_j)\| \leq Ck^{-1/2}$}, \\ \|\grad f(\bx_0)\| & \text{otherwise}. \end{cases} \ee
	Then since $\min_{j\leq k} \|\grad f(\bx_j)\| \leq H_k$, we get
	\be \E{\min_{j\leq k} \|\grad f(\bx_j)\|} \leq \E{H_k} \leq Ck^{-1/2} + \|\grad f(\bx_0)\|\: \Prob{\min_{j\leq k} \|\grad f(\bx_j)\| > Ck^{-1/2}}, \ee
	and we get the first result by applying \corref{cor_high_prob_complexity}.
	
	Next, if $\epsilon\in(0,\epsilon_0)$ then
	\be k\geq k_0(\epsilon) \defeq \frac{2(\psi(\epsilon)+1)}{1-\delta_S-C_4/(1+C_4)} = \Theta(\epsilon^{-2}), \ee
	and so from \thmref{thm_high_prob_complexity} we have
	\be \Prob{K_{\epsilon} \leq k} = \Prob{\min_{j\leq k} \|\grad f(\bx_j)\| \leq \epsilon} \geq 1 - e^{-\t{c}_1 k}, \ee
	where $\t{c}_1 \defeq (1-\delta_S-C_4/(1+C_4))^2 /[16(1-\delta_S)]$.
	We use the identity $\E{X} = \int_{0}^{\infty} \Prob{X>t} dt$ for non-negative random variables $X$ (e.g.~\cite[eqn.~(1.9)]{Tao2011}) to get
	\be \E{K_{\epsilon}} \leq k_0(\epsilon) + \int_{k_0(\epsilon)}^{\infty} \Prob{K_{\epsilon} > t} dt \leq k_0(\epsilon) + \sum_{k=k_0(\epsilon)}^{\infty} e^{-\t{c}_1 k} = k_0(\epsilon) + \frac{e^{-\t{c}_1 k_0(\epsilon)}}{1-e^{-\t{c}_1}}, \ee
	where $\t{C}_1$ comes from $k_0(\epsilon)=\Theta(\epsilon^{-2})$, which concludes our proof.
\end{proof}

Furthermore, we also get almost-sure convergence of $\liminf$ type, similar to \cite[Theorem 10.12]{Conn2009} in the deterministic case.

\begin{theorem} \label{thm_almost_sure}
	Suppose the assumptions of \thmref{thm_high_prob_complexity} hold.
	Then the iterates of RSDFO satisfy $\inf_{k\geq 0} \|\grad f(\bx_k)\| = 0$ almost surely.
\end{theorem}
\begin{proof}
	From \thmref{thm_high_prob_complexity}, for any $\epsilon>0$ we have
	\be \lim_{k\to\infty} \Prob{\min_{j\leq k} \|\grad f(\bx_j)\| > \epsilon} = 0. \ee
	However, $\Prob{\inf_{k\geq 0}\|\grad f(\bx_k)\| > \epsilon} \leq \Prob{\min_{j\leq k} \|\grad f(\bx_j)\| > \epsilon}$ for all $k$, and so
	\be \Prob{\inf_{k\geq 0}\|\grad f(\bx_k)\| > \epsilon} = 0. \ee
	The result follows from the union bound applied to any sequence $\epsilon_n = n^{-1}$, for example.
\end{proof}

In particular, if $\|\grad f(\bx_k)\| > 0$ for all $k$, then \thmref{thm_almost_sure} implies $\liminf_{k\to\infty} \|\grad f(\bx_k)\| = 0$ almost surely.

\subsection{Selecting a Subspace Dimension} \label{sec_achieving_sketch}
We now specify how to generate our subspaces $Q_k$ to be probabilistically well-aligned and uniformly bounded (\assref{ass_sketching}).
These requirements are quite weak, and so there are several possible approaches for constructing $Q_k$.
First, we discuss the case where $Q_k$ is chosen to be a random matrix with orthonormal columns, and show that we need to choose our subspace dimension $p\sim\sqrt{n}$.
This is related to how we ultimately select $Q_k$ in the practical implementation DFBGN (\secref{sec_implementation}).
However, we then show that by instead taking $Q_k$ to be a Johnson-Lindenstrauss transform, we can choose a value of $p$ independent of the ambient dimension $n$.

\subsubsection{Random Orthogonal Basis}
First, we consider the case where $Q_k$ is a randomly generated matrix with orthonormal columns.
We have the below result, a consequence of \cite[Theorem 9]{Mahoney2016}.

\begin{theorem}
	Suppose the columns of $Q_k\in\R^{n\times p}$ form an orthonormal basis for a randomly generated $p$-dimensional subspace of $\R^n$.
	Then for any fixed vector $\bv$,
	\be \Prob{\|Q_k^T \bv\| > \left(\frac{p}{n} - \theta\sqrt{\frac{p}{n}}\right) \|\bv\| } > 1- 3\exp(-p\theta^2 / 64), \ee
	for all $\theta\in(0,1)$.
\end{theorem}

For some $\delta_S\in(0,1)$, if we set
\be \theta = \frac{8\sqrt{\log(3/\delta_S)}}{\sqrt{p}}, \ee
then we get
\be \Prob{\|Q_k^T \grad f(\bx_k)\| > \left(\frac{p}{n} - \theta\sqrt{\frac{p}{n}}\right) \|\grad f(\bx_k)\| } > 1- \delta_S. \ee
Therefore \assref{ass_sketching} is achieved provided we set
\be p \geq \alpha_Q n + \sqrt{64\log(3/\delta_S) n\,}. \ee
In particular, since our theory holds if we take $\alpha_Q>0$ arbitrarily small, our only limitation on the subspace dimension $p$ is $\delta_S < 1/(1+C_4)$ from \thmref{thm_high_prob_complexity}, yielding the minimum requirement
\be p > \sqrt{64\log(3+3C_4) n\,}, \label{eq_orthogonal_Q_sketch_dimension} \ee
where $C_4$ depends only on our choice of trust-region algorithm parameters. 
The boundedness condition \assref{ass_sketching}\ref{item_sketch_bounded} holds with $Q_{\max}=1$ automatically.
This gives us considerable scope to use very small subspace dimensions, which means our subspace approach has a strong chance of providing a substantial reduction in linear algebra costs.

\subsubsection{Johnson-Lindenstrauss Embeddings}
We can improve on the requirement \eqref{eq_orthogonal_Q_sketch_dimension} on $p$ by using $Q_k$ with non-orthonormal columns.
Specifically, we take $Q_k$ to be a Johnson-Lindenstrauss transform (JLT) \cite{Woodruff2014}.
The application of these techniques to random subspace optimization algorithms follows \cite{Fowkes2020,Shao2021}.

\begin{definition}
	A random matrix $S\in\R^{p\times n}$ is an $(\epsilon,\delta)$-JLT if, for any point $\bv\in\R^n$, we have
	\begin{align}
		\Prob{(1-\epsilon)\|\bv\|^2 \leq \|S \bv\|^2 \leq (1+\epsilon) \|\bv\|^2} \geq 1-\delta.
	\end{align}
\end{definition}

There have been many different approaches for constructing $(\epsilon,\delta)$-JLT matrices proposed.
Two common examples are:
\begin{itemize}
	\item If $S$ is a random Gaussian matrix with independent entries $S_{i,j}\sim N(0,1/p)$ and $p = \Omega(\epsilon^{-2}|\log\delta|)$, then $S$ is an $(\epsilon,\delta)$-JLT (see \cite[Theorem 2.13]{Boucheron2012}, for example).
	\item We say that $S$ is an $s$-hashing matrix if it has exactly $s$ nonzero entries per column (indices sampled independently), which take values $\pm 1/\sqrt{s}$ selected independently with probability 1/2. If $S$ is an $s$-hashing matrix with $s=\Theta(\epsilon^{-1}|\log\delta|)$ and $p=\Omega(\epsilon^{-2}|\log\delta|)$, then $S$ is an $(\epsilon,\delta)$-JLT \cite{Kane2014}.
\end{itemize}

By taking $\bv = \grad f(\bx_k)$ in iteration $k$, and noting $(1-\epsilon)^2 \leq 1-\epsilon$ for all $\epsilon\in(0,1)$, we have that \assref{ass_sketching}\ref{item_sketched_aligned_prob} holds if we take $Q_k=S^T$, where $S$ is any $(1-\alpha_Q, \delta_S)$-JLT.
That is, \assref{ass_sketching}\ref{item_sketched_aligned_prob} is satisfied using either of the constructions above and $p=\Omega((1-\alpha_Q)^{-2}|\log\delta_S|)$.
Importantly, these constructions allow us to choose a subspace dimension $p$ which has \emph{no dependence on the ambient dimension} $n$ (i.e.~$p=\bigO(1)$ as $n\to\infty$), an improvement on \eqref{eq_orthogonal_Q_sketch_dimension}.
We note that the requirement $\delta_S < 1/(1+C_4)$ in \thmref{thm_high_prob_complexity} yields a very mild dependence of $p$ on the choice of trust-region updating parameters.

We conclude by noting that the uniform boundedness property \assref{ass_sketching}\ref{item_sketch_bounded} is trivial if $S$ is a hashing matrix.
If $S$ is Gaussian, by Bernstein's inequality we can choose $Q_{\max}$ large enough that $\Prob{\|S\|> Q_{\max}}$ is small.
Then, we generate our final $Q_k$ by sampling $S$ independently until $\|S\|\leq Q_{\max}$ holds (which almost-surely takes finite time).
A union bound argument then gives us \assref{ass_sketching}\ref{item_sketched_aligned_prob} for this choice of $Q_k$, and so \assref{ass_sketching} is completely satisfied with the same (asymptotic) requirements on $p$.

\section{Random Subspace Nonlinear Least-Squares Method} \label{sec_rsdfogn}
We now describe how RSDFO (\algref{alg_rsdfo}) can be specialized to the unconstrained nonlinear least-squares problem
\be \min_{\bx\in\R^n} f(\bx) \defeq \frac{1}{2}\|\br(\bx)\|^2 = \frac{1}{2}\sum_{i=1}^{m} r_i(\bx)^2, \label{eq_ls_definition} \ee
where $\br:\R^n\to\R^m$ is given by $\br(\bx)\defeq[r_1(\bx), \ldots, r_m(\bx)]^T$.
We assume that $\br$ is differentiable, but that access to the Jacobian $J:\R^n\to\R^{m\times n}$ is not possible.
In addition, we typically assume that $m\geq n$ (regression), but everything here also applies to the case $m<n$ (inverse problems).
We now introduce the algorithm RSDFO-GN (Randomized Subspace DFO with Gauss-Newton), which is a randomized subspace version of a model-based DFO variant of the Gauss-Newton method \cite{Cartis2019a}.

Following the construction from \cite{Cartis2019a}, we assume that we have selected the $p$-dimensional search space $\Y_k$ defined by $Q_k\in\R^{n\times p}$ (as in RSDFO above).
Then, we suppose that we have evaluated $\br$ at $p+1$ points $Y_k \defeq \{\bx_k,\by_1,\ldots,\by_n\} \subset \Y_k$ (which typically are all close to $\bx_k$).
Since $\by_t\in \Y_k$ for each $t=1,\ldots,p$, from \eqref{eq_subspace_definition} we have $\by_t = \bx_k + Q_k \hat{\bs}_t$ for some $\hat{\bs}_t\in \R^p$.

Given this interpolation set, we first wish to construct a local subspace linear model for $\br$:
\begin{align}
	\br(\bx_k + Q_k \hat{\bs}) \approx \hat{\bem}_k(\hat{\bs}) = \br(\bx_k) + \hat{J}_k \hat{\bs}. \label{eq_r_model}
\end{align}
To do this, we choose the approximate subspace Jacobian $\hat{J}_k\in\R^{m\times p}$ by requiring that $\hat{\bem}_k$ interpolate $\br$ at our interpolation points $Y_k$.
That is, we impose
\begin{align}
	\hat{\bem}_k(\hat{\bs}_t) = \br(\by_t), \qquad \forall t=1,\ldots,p, \label{eq_interp_conditions}
\end{align}
which yields the $p\times p$ linear system (with $m$ right-hand sides)
\begin{align}
	\hat{W}_k \hat{J}_k^T \defeq \begin{bmatrix} \hat{\bs}_1^T \\ \vdots \\ \hat{\bs}_p^T \end{bmatrix} \hat{J}_k^T = \begin{bmatrix} (\br(\by_1)-\br(\bx_k))^T \\ \vdots \\ (\br(\by_p)-\br(\bx_k))^T \end{bmatrix}. \label{eq_gn_interp_system}
\end{align}
Our linear subspace model $\hat{\bem}_k$ \eqref{eq_r_model} naturally yields a local subspace quadratic model for $f$, as in the classical Gauss-Newton method, namely (c.f.~\eqref{eq_reduced_model_generic}),
\be f(\bx_k+Q_k \hat{\bs}) \approx \hat{m}_k(\hat{\bs}) \defeq \frac{1}{2}\|\hat{\bem}_k(\hat{\bs})\|^2 = f(\bx_k) + \hat{\bg}_k^T \hat{\bs} + \frac{1}{2} \hat{\bs}^T \hat{H}_k \hat{\bs}, \label{eq_f_model} \ee
where $\hat{\bg}_k \defeq \hat{J}_k^T \br(\bx_k)$ and $\hat{H}_k \defeq \hat{J}_k^T \hat{J}_k$.

\subsection{Constructing $Q_k$-Fully Linear Models}
We now describe how we can achieve $Q_k$-fully linear models of the form \eqref{eq_f_model} in RSDFO-GN.

As in \cite{Cartis2019a}, we will need to define the Lagrange polynomials and $\Lambda$-poisedness of an interpolation set.
Given our interpolation set $Y_k$ lies inside $\Y_k$, we consider the (low-dimensional) Lagrange polynomials associated with $Y_k$.
These are the linear functions $\hat{\ell}_0,\ldots,\hat{\ell}_p:\R^p\to\R$, defined by the interpolation conditions
\begin{align}
	\hat{\ell}_t(\hat{\bs}_{t'}) = \delta_{t, t'}, \qquad \forall t,t'=0,\ldots,p,
\end{align}
with the convention $\hat{\bs}_0=\b{0}$ corresponding to the interpolation point $\bx_k$.
The Lagrange polynomials exist and are unique whenever $\hat{W}_k$ \eqref{eq_gn_interp_system} is invertible, which we typically ensure through judicious updating of $Y_k$ at each iteration.

\begin{definition} \label{def_lambda_poised}
	For any $\Lambda>0$, the set $Y_k$ is $\Lambda$-poised in the $p$-dimensional ball $B(\bx_k,\Delta_k)\cap\Y_k$ if
	\be \max_{t=0,\ldots,p} \: \max_{\|\hat{\bs}\|\leq \Delta_k} |\hat{\ell}_t(\hat{\bs})| \leq \Lambda. \ee
\end{definition}

Note that since $\hat{\ell}_0(\b{0})=1$, for the set $Y_k$ to be $\Lambda$-poised we require $\Lambda\geq 1$.
In general, a larger $\Lambda$ indicates that $Y_k$ has ``worse'' geometry, which leads to a less accurate approximation for $f$.
This notion of $\Lambda$-poisedness (in a subspace) is sufficient to construct $Q_k$-fully linear models \eqref{eq_f_model} for $f$.

\begin{lemma} \label{lem_ls_interp_fully_linear}
	Suppose \assref{ass_sketching}\ref{item_sketch_bounded} holds, $J(\bx)$ is Lipschitz continuous, and $\br$ and $J$ are uniformly bounded above in $\cup_{k\geq 0} B(\bx_k,\Delta_{\max})$.
	If $Y_k \subset B(\bx_k,\Delta_k)\cap\Y_k$ and $Y_k$ is $\Lambda$-poised in $B(\bx_k,\Delta_k)\cap\Y_k$, then $\hat{m}_k$ \eqref{eq_f_model} is a $Q_k$-fully linear model for $f$, with $\kappa_{\rm ef},\kappa_{\rm eg}=\bigO(p^2 \Lambda^2)$.
\end{lemma}
\begin{proof}
	Consider the low-dimensional functions $\hat{\br}:\R^p\to\R^m$ and $\hat{f}:\R^p\to\R$ given by $\hat{\br}_k(\hat{\bs})\defeq \br(\bx_k+Q_k\hat{\bs})$ and $\hat{f}(\hat{\bs}) \defeq \frac{1}{2}\|\hat{\br}(\hat{\bs})\|^2$ respectively.
	We note that $\hat{\br}$ is continuously differentiable with Jacobian $\hat{J}(\hat{\bs}) = J(\bx_k+Q_k\hat{\bs}) Q_k$.
	Then since $\|Q_k\|\leq Q_{\max}$ from \assref{ass_sketching}\ref{item_sketch_bounded}, it is straightforward to show that both $\hat{\br}$ and $\hat{J}$ are uniformly bounded above and $\hat{J}$ is Lipschitz continuous (with a Lipschitz constant $Q_{\max}^2$ times larger than for $J(\bx)$).
	
	We can then consider $\hat{\bem_k}$ \eqref{eq_r_model} and $\hat{m}_k$ \eqref{eq_f_model} to be interpolation models for $\hat{r}$ and $\hat{f}$ in the low-dimensional ball $B(\b{0},\Delta_k)\subset\R^p$.
	From \cite[Lemma 3.3]{Cartis2019a}, we conclude that $\hat{m}_k$ is a fully linear model for $\hat{f}$ with constants $\kappa_{\rm ef},\kappa_{\rm eg}=\bigO(p^2 \Lambda^2)$.
	The $Q_k$-fully linear property follows immediately from this, noting that $\grad \hat{f}_k(\hat{\bs}) = Q_k^T \grad f(\bx_k+Q_k\hat{\bs})$.
\end{proof}

Given this result, the procedures in \cite[Chapter 6]{Conn2009} allow us to check and/or guarantee the $\Lambda$-poisedness of an interpolation set, and we have met all the requirements needed to fully specify RSDFO-GN.

Lastly, we note that underdetermined linear interpolation, where \eqref{eq_gn_interp_system} is underdetermined and solved in a minimal norm sense, has been recently shown to yield a property similar to $Q_k$-full linearity \cite[Theorem 3.6]{Hare2020b}.

\paragraph{Complete RSDFO-GN Algorithm}
A complete statement of RSDFO-GN is given in \algref{alg_dfbgn_theory}.
This exactly follows RSDFO (\algref{alg_rsdfo}), but where we ask that the interpolation set satsifies the conditions: $Y_k \subset B(\bx_k,\Delta_k)\cap\Y_k$ and $Y_k$ is $\Lambda$-poised in $B(\bx_k,\Delta_k)\cap\Y_k$.
From \lemref{lem_ls_interp_fully_linear}, this is sufficient to guarantee $Q_k$-full linearity of $\hat{m}_k$.

\begin{algorithm}[tb]
	\footnotesize{
	\begin{algorithmic}[1]
		\Require Starting point $\bx_0\in\R^n$, initial trust region radius $\Delta_0>0$, and subspace dimension $p\in\{1,\ldots,n\}$. 
		\vspace{0.2em}
		\Statex \underline{Parameters}: maximum trust-region radius $\Delta_{\rm max}\geq\Delta_0$, trust-region radius scalings $0<\gamma_{\rm dec}<1<\gamma_{\rm inc}\leq\overline{\gamma}_{\rm inc}$, criticality constants $\epsilon_C,\mu>0$ and trust-region scaling $0<\gamma_C<1$, safety step threshold $\beta_F>0$ and trust-region scaling $0<\gamma_F<1$, acceptance thresholds $0 < \eta_1 \leq \eta_2 < 1$, and poisedness constant $\Lambda>1$.
		\vspace{0.5em}
		\State Set flag \texttt{CHECK\_MODEL}=\texttt{FALSE}.
		\For{$k=0,1,2,\ldots$}
		    \If{\texttt{CHECK\_MODEL}=\texttt{TRUE}}
				\State Set $Q_k=Q_{k-1}$.
				\State Construct an interpolation set $Y_k \subset B(\bx_k,\Delta_k)\cap\Y_k$ which is $Y_k$ is $\Lambda$-poised in $B(\bx_k,\Delta_k)\cap\Y_k$.
		        \State Build the reduced model $\hat{m}_k:\R^p\to\R$ \eqref{eq_f_model} by solving \eqref{eq_gn_interp_system}. 
		    \Else
				\State Define a subspace by randomly sampling $Q_k\in\R^{n\times p}$.
		        \State Construct a reduced model $\hat{m}_k:\R^p\to\R$ \eqref{eq_f_model} by solving \eqref{eq_gn_interp_system}, where the interpolation points $Y_k\subset\Y_k$ need not be contained in $B(\bx_k,\Delta_k)$ or be $\Lambda$-poised in $B(\bx_k,\Delta_k)\cap\Y_k$.
		    \EndIf
		    \State Follow lines \ref{ln_main_start} to \ref{ln_main_end} of RSDFO (\algref{alg_rsdfo}), but replace every instance of checking $Q_k$-full linearity of $\hat{m}_k$ in $B(\bx_k,\Delta_k)$ with checking that $Y_k \subset B(\bx_k,\Delta_k)\cap\Y_k$ and $Y_k$ is $\Lambda$-poised in $B(\bx_k,\Delta_k)\cap\Y_k$.
		\EndFor
	\end{algorithmic}
	} 
	\caption{RSDFO-GN (Randomized Subspace Derivative-Free Optimization with Gauss-Newton) for solving \eqref{eq_ls_definition}.}
	\label{alg_dfbgn_theory}
\end{algorithm}

\subsection{Complexity Analysis for RSDFO-GN}
We are now in a position to specialize our complexity analysis for RSDFO to RSDFO-GN.
For this, we need to impose a smoothness assumption on $\br$.

\begin{assumption} \label{ass_ls_smoothness}
	The extended level set $\mathcal{L}\defeq \{\by\in\R^n : \|\by-\bx\| \leq\Delta_{\max} \: \text{for some} \: f(\bx) \leq f(\bx_0)\}$ is bounded, $\br$ is continuously differentiable, and the Jacobian $J$ is Lipschitz continuous on $\mathcal{L}$.
\end{assumption}

This smoothness requirement allows us to immediately apply the complexity analysis for RSDFO, yielding the following result.

\begin{corollary}
	Suppose Assumptions \ref{ass_ls_smoothness}, \ref{ass_bdd_hess}, \ref{ass_cauchy_decrease} and \ref{ass_sketching} hold, and we have $\beta_F\leq c_2$, $\delta_S < 1/(1+C_4)$ for $C_4$ defined in \lemref{lem_aligned_bound}, and $\gamma_{\rm inc} > \min(\gamma_C, \gamma_F, \gamma_{\rm dec}, \beta_F)^{-2}$.
	Then for the iterates generated by RSDFO-GN and $k$ sufficiently large, 
	\be \Prob{\min_{j\leq k} \|\grad f(\bx_j)\| \leq \frac{C}{\sqrt{k}}} \geq 1 - e^{-ck}, \label{eq_high_prob3_ls} \ee
	for some constants $c,C>0$.
	Alternatively, for $\epsilon\in(0,\epsilon_0)$ for some $\epsilon_0$, we have
	\be \Prob{K_{\epsilon} \leq \t{C}\epsilon^{-2}} \geq 1 - e^{-\t{c}\epsilon^{-2}}, \label{eq_high_prob4_ls} \ee
	for constants $\t{c},\t{C}>0$.
\end{corollary}
\begin{proof}
	\assref{ass_ls_smoothness} implies that both $\br$ and $J$ are uniformly bounded above on $\mathcal{L}$, which is sufficient for \lemref{lem_ls_interp_fully_linear} to hold.
	Hence, whenever we check/ensure that $Y_k \subset B(\bx_k,\Delta_k)\cap\Y_k$ and $Y_k$ is $\Lambda$-poised in $B(\bx_k,\Delta_k)\cap\Y_k$ we are checking/guaranteeing that $\hat{m}_k$ is $Q_k$-fully linear in $B(\bx_k,\Delta_k)$.
	In addition, from \cite[Lemma 3.2]{Cartis2019a} and taking $f_{\rm low}=0$, we have that \assref{ass_smoothness} is satisfied.
	Therefore the result follows directly from \corref{cor_high_prob_complexity}.
\end{proof}

\paragraph{Bounds on Objective Evaluations}
In each iteration of RSDFO, we require at most $p+1$ objective evaluations: at most $p$ to form the model $\hat{m}_k$---and this holds regardless of whether we need $\hat{m}_k$ to be $Q_k$-fully linear or not---and one evaluation for $\bx_k+\bs_k$. Hence the above bounds on $K_{\epsilon}$ also hold for the number of objective evaluations required to first achieve $\|\grad f(\bx_k)\|\leq\epsilon$, up to a constant factor of $p+1$.

\paragraph{Interpolation Models for RSDFO}
Using existing techniques for constructing $\Lambda$-poised interpolation sets for quadratic interpolation (as outlined in \cite{Conn2009}), the specific model construction ideas presented here can also be applied to general objective problems, and thus provide a concrete implementation of RSDFO.

\subsection{Linear Algebra Cost of RSDFO-GN} \label{sec_dfbgn_linalg_summary}
In RSDFO-GN, the interpolation linear system \eqref{eq_gn_interp_system} is solved in two steps, namely: factorize the interpolation matrix $\hat{W}_k$, then back-solve for each right-hand side.
Thus, the cost of the linear algebra is:
\begin{enumerate}
	\item Model construction costs $\bigO(p^3)$ to compute the factorization of $\hat{W}_k$, and $\bigO(mp^2)$ for the back-substitution solves with $m$ right-hand sides; and
	\item Lagrange polynomial construction costs $\bigO(p^3)$ in total, due to one backsolve for each of the $p+1$ polynomials (using the pre-existing factorization of $\hat{W}_k$).
\end{enumerate}
By updating the factorization or $\hat{W}_k^{-1}$ directly (e.g.~via the Sherman-Morrison formula), we can replace the $\bigO(p^3)$ factorization cost with a $\bigO(p^2)$ updating cost (c.f.~\cite{Powell2004a}).
However, the dominant $\bigO(mp^2)$ model construction cost remains, and in practice we have observed that the factorization needs to be recomputed from scratch to avoid the accumulation of rounding errors.

In the case of a full-space method where $p=n$ such as in \cite{Cartis2019a}, these costs becomes $\bigO(n^3)$ for the factorization (or $\bigO(n^2)$ if Sherman-Morrison is used) plus $\bigO(mn^2)$ for the back-solves.
When $n$ grows large, this linear algebra cost rapidly dominates the total runtime of these algorithms and limits the efficiency of full-space methods.
This issue is discussed in more detail, with numerical results, in \cite[Chapter 7.2]{Roberts2019}.

In light of this discussion, we now turn our attention to building an implementation of RSDFO-GN that has both strong performance (in terms of objective evaluations) and low linear algebra cost.

\section{DFBGN: An Efficient Implementation of RSDFO-GN} \label{sec_implementation}
An important tenet of DFO is that objective evaluations are often expensive, and so algorithms should be efficient in reusing information, hence limiting the total objective evaluations required to achieve a given decrease.
However because we require our model to sit within our active space $\Y_k$, we do not have a natural process by which to reuse evaluations between iterations, when the space changes.
We dedicate this section to outlining an implementation of RSDFO-GN, which we call DFBGN (Derivative-Free Block Gauss-Newton).
DFBGN is designed to be efficient in its objective queries while still only building low-dimensional models, and hence is also efficient in terms of linear algebra.
Specifically, we design DFBGN to achieve two aims: 
\begin{itemize}
	\item \textit{Low computational cost:} we want our implementation to have a per-iteration linear algebra cost which is linear in the ambient dimension;
	\item \textit{Efficient use of objective evaluations:} our implementation should follow the principles of other DFO methods and make progress with few objective evaluations. In particular, we hope that, when run with `full-space models' (i.e.~$p=n$), our implementation should have (close to) state-of-the-art performance.
\end{itemize}
We will assess the second point in \secref{sec_numerics} by comparison with DFO-LS \cite{Cartis2018} an open-source model-based DFO Gauss-Newton solver which explores the full space (i.e.~$p=n$).

\begin{remark} \label{rem_dfols_growing}
	As discussed in \cite{Cartis2018}, DFO-LS has a mechanism to build a model with fewer than $n+1$ interpolation points.
	However, in that context we modify the model so that it varies over the whole space $\R^n$, which enables the interpolation set to grow to the usual $n+1$ points and yield a full-dimensional model.
	There, the goal is to make progress with very few evaluations, but here our goal is scalability, so we keep our model low-dimensional throughout and instead change the subspace at each iteration.
\end{remark}

\subsection{Subspace Interpolation Models} \label{sec_reduced_subspace_interp}
Similar to \secref{sec_rsdfogn}, we assume that, at iteration $k$, our interpolation set has $p+1$ points $\{\bx_k,\by_1,\ldots,\by_p\}\subset\R^n$ with $1\leq p\leq n$.
However, we assume that these points are already given, and use them to determine the space $\Y_k$ (as defined by $Q_k$).
That is, given
\begin{align}
	W_k \defeq \begin{bmatrix}(\by_1-\bx_k)^T \\ \vdots \\ (\by_p-\bx_k)^T\end{bmatrix} \in \R^{p\times n},
\end{align}
we compute the QR factorization 
\be W_k^T = Q_k R_k, \label{eq_qr_decomp} \ee
where $Q_k\in\R^{n\times p}$ has orthonormal columns and $R_k\in\R^{p\times p}$ is upper triangular---and invertible provided $W_k^T$ is full rank, which we guarantee by judicious replacement of interpolation points.
This gives us the $Q_k$ that defines $\Y_k$ via \eqref{eq_subspace_definition}---in this case $Q_k$ is has orthonormal columns---and in this way all our interpolation points are in $\Y_k$.

Since each $\by_t\in\Y_k$, from \eqref{eq_qr_decomp} we have $\by_t = \bx_k + Q_k \hat{\bs}_t$, where $\hat{\bs}_t$ is the $t$-th column of $R_k$.
Hence we have $\hat{W}_k = R_k^T$ in \eqref{eq_gn_interp_system} and so $\hat{\bem}_k$ \eqref{eq_r_model} is given by solving
\begin{align}
	R_k^T \hat{J}_k^T = \begin{bmatrix} (\br(\by_1)-\br(\bx_k))^T \\ \vdots \\ (\br(\by_p)-\br(\bx_k))^T \end{bmatrix}, \label{eq_reduced_interp_system}
\end{align}
via forward substitution, since $R_k^T$ is lower triangular.
This ultimately gives us our local model $\hat{m}_k$ via \eqref{eq_f_model}.

We reiterate that compared to RSDFO-GN, we have used the interpolation set $Y_k$ to determine both $Q_k$ and $\hat{m}_k$, rather than first sampling $Q_k$, then finding interpolation points $Y_k\subset\Y_k$ with which to construct $\hat{m}_k$.
This difference is crucial in allowing the reuse of interpolation points between iterations, and hence lowering the objective evaluation requirements of model construction.

\begin{remark}
	As discussed in \cite[Chapter 7.3]{Roberts2019}, we can equivalently recover this construction by asking for a full-space model $\bem_k:\R^n\to\R^m$ given by $\bem_k(\bs)=\br(\bx_k)+J_k \bs$ such that the interpolation conditions $\bem_k(\by_t-\bx_k)=\br(\by_t)$ are satisfied and $J_k$ has minimal Frobenius norm.
\end{remark}

\subsection{Complete DFBGN Algorithm} \label{sec_dfbgn_full_algo}
A complete statement of DFBGN is given in \algref{alg_block_dfols}.
Compared to RSDFO-GN, we include specific steps to manage the interpolation set, which in turn dictates the choice of subspace $\Y_k$.
Specifically, one issue with our approach is that our new iterate $\bx_k+\bs_k$ is in $\Y_k$, so if we were to simply add $\bx_k+\bs_k$ into the interpolation set, $\Y_k$ would not change across iterations, and we will never explore the whole space.
On the other hand, unlike RSDFO and RSDFO-GN we do not want to completely resample $Q_k$ as this would require too many objective evaluations.
Instead, in DFBGN we delete a subset of points from the interpolation set and add new directions orthogonal to the existing directions, which ensures that $Q_{k+1}\neq Q_k$ in every iteration.\footnote{By contrast, the optional growing mechanism in DFO-LS (\remref{rem_dfols_growing}) is designed such that $\bx_k+\bs_k$ is not in $\Y_k$, and so the search space is automatically expanded at every iteration. However, this requires an expensive SVD of $J_k\in\R^{m\times n}$ at every iteration, and so is not suitable for our large-scale setting.}

We also note that DFBGN does not include some important algorithmic features present in RSDFO-GN, DFO-LS or other model-based DFO methods, and hence is quite simple to state.
These features are not necessary for a variety of reasons, which we now outline.

\begin{algorithm}[t]
	\footnotesize{
	\begin{algorithmic}[1]
		\Require Starting point $\bx_0\in\R^n$, initial trust region radius $\Delta_0>0$, subspace dimension $p\in\{1,\ldots,n\}$, and number of points to drop at each iteration $p_{\rm drop}\in\{1,\ldots,p\}$. 
		\vspace{0.2em}
		\Statex \underline{Parameters}: maximum and minimum trust-region radii $\Delta_{\rm max}\geq\Delta_0>\Delta_{\rm end}>0$, trust-region radius scalings $0<\gamma_{\rm dec}<1<\gamma_{\rm inc}\leq\overline{\gamma}_{\rm inc}$, and acceptance thresholds $0 < \eta_1 \leq \eta_2 < 1$.
		\vspace{0.5em}
		\State Select random orthonormal directions $\b{d}_1,\ldots,\b{d}_p\in\R^n$ using \algref{alg_dfbgn_add_points}, and build initial interpolation set $Y_0\defeq\{\bx_0,\bx_0+\Delta_0 \b{d}_1,\ldots,\bx_0+\Delta_0 \b{d}_p\}$. \label{ln_dfbgn_init}
		\For{$k=0,1,2,\ldots$}
			\State Given $\bx_k$ and $Y_k$, solve \eqref{eq_reduced_interp_system} to build subspace models $\hat{\bem}_k:\R^p\to\R^m$ \eqref{eq_r_model} and $\hat{m}_k:\R^p\to\R$ \eqref{eq_f_model}. 
			\State Approximately solve the subspace trust-region subproblem in $\R^{p}$ \eqref{eq_reduced_trs_generic} and calculate the step $\bs_k = Q_k \hat{\bs}_k \in\R^n$.
			\State Evaluate $\br(\bx_k+\bs_k)$ and calculate ratio $\rho_k$ \eqref{eq_ratio_generic}.
			\State Accept/reject step and update trust region radius: set $\bx_{k+1}$ and $\Delta_{k+1}$ as per \eqref{eq_dfbgn_tr_updating_generic}.
			\State \textbf{if} $\Delta_{k+1} \leq \Delta_{\rm end}$, \textbf{terminate}.
			\If{$p<n$}
				\State Set $Y_{k+1}^{\rm init}=Y_k\cup\{\bx_k+\bs_k\}$.
				\State Remove $\min(\max(p_{\rm drop},2),p)$ points from $Y_{k+1}^{\rm init}$ (without removing $\bx_{k+1}$) using \algref{alg_dfbgn_remove_points}. \label{ln_drop1}
			\Else
				\State Set $Y_{k+1}^{\rm init}=Y_k\cup\{\bx_k+\bs_k\}\setminus\{\by\}$ for some $\by\in Y_k\setminus\{\bx_{k+1}\}$. \label{ln_replace}
				\State Remove $\min(\max(p_{\rm drop},1),p)$ points from $Y_{k+1}^{\rm init}$ (without removing $\bx_{k+1}$) using \algref{alg_dfbgn_remove_points}. \label{ln_drop2}
			\EndIf
			\State Let $q\defeq p+1-|Y_{k+1}^{\rm init}|$, and generate random orthonormal vectors $\{\b{d}_1,\ldots,\b{d}_q\}$ that are also orthogonal to $\{\by-\bx_{k+1} : \by\in Y_{k+1}^{\rm init}\setminus\{\bx_{k+1}\}\}$, using \algref{alg_dfbgn_add_points}. \label{ln_new_dirns}
			\State Set $Y_{k+1} = Y_{k+1}^{\rm init}\cup\{\bx_{k+1}+\Delta_{k+1}\b{d}_1,\ldots,\bx_{k+1}+\Delta_{k+1}\b{d}_q\}$.
		\EndFor
	\end{algorithmic}
	} 
	\caption{DFBGN: Derivative-Free Block Gauss-Newton for solving \eqref{eq_ls_definition}.}
	\label{alg_block_dfols}
\end{algorithm}

\paragraph{No Criticality and Safety Steps}
Compared to RSDFO-GN, the implementation of DFBGN does not include criticality (which is also the case in DFO-LS) or safety steps.
These steps ultimately function to ensure we do not have $\hat{\bg}_k\ll\Delta_k$.
In DFBGN, we ensure $\Delta_k$ does not get too large compared to $\|\bs_k\|$ through \eqref{eq_dfbgn_tr_updating_generic}, while $\|\bs_k\|$ is linked to $\|\hat{\bg}_k\|$ through \lemref{lem_step_lower_bound}.
If $\|\bs_k\|$ is much smaller than $\Delta_k$ and our step produces a poor objective decrease, then we will set $\Delta_{k+1}\gets\|\bs_k\|$ for the next iteration.
Although \lemref{lem_step_lower_bound} allows $\|\bs_k\|$ to be large even if $\|\bg_k\|$ is small, in practice we do not observe $\Delta_k\gg\|\bg_k\|$ without DFBGN setting $\Delta_{k+1}\gets\|\bs_k\|$ after a small number of iterations.

\paragraph{No Model-Improving Steps}
An important feature of model-based DFO methods are model-improving procedures, which change the interpolation set to ensure $\Lambda$-poisedness (\defref{def_lambda_poised}), or equivalently ensure that the local model for $f$ is fully linear.
In RSDFO-GN for instance, model-improvement is performed when \texttt{CHECK\_MODEL}=\texttt{TRUE}, whereas in \cite[Algorithm 10.1]{Conn2009} there are dedicated model-improvement phases.

Instead, DFBGN ensures accurate interpolation models via a geometry-aware (i.e.~$\Lambda$-poisedness aware) process for deleting interpolation points at each iteration, where they are replaced by new points in directions (from $\bx_{k+1}$) which are orthogonal to $Q_k$ and selected at random.
The process for deleting interpolation points---and choosing a suitable number of points to remove, $p_{\rm drop}$---at each iteration are considered in Sections \ref{sec_dfbgn_interp_set} and \ref{sec_scalability_practicality} respectively.
The process for generating new interpolation points, \algref{alg_dfbgn_add_points}, is outlined in \secref{sec_dfbgn_interp_set}.

A downside of our approach is that the new orthogonal directions are not chosen by minimizing a model for the objective (i.e.~not attempting to reduce the objective), as we have no information about how the objective varies outside $\Y_k$.
This is the fundamental trade-off between a subspace approach and standard methods (such as DFO-LS); we can reduce the linear algebra cost, but must spend objective evaluations to change the search space between iterations.

\paragraph{Linear Algebra Cost of DFBGN}
As in \secref{sec_dfbgn_linalg_summary}, our approach in DFBGN yields substantial reductions in the required linear algebra costs compared to DFO-LS:
\begin{itemize}
	\item Model construction costs $\bigO(np^2)$ for the factorization \eqref{eq_qr_decomp} and $\bigO(mp^2)$ for back-substitution solves \eqref{eq_reduced_interp_system}, rather than $\bigO(n^3)$ and $\bigO(mn^2)$ respectively for DFO-LS; and
	\item Lagrange polynomial construction costs $\bigO(p^3)$ rather than $\bigO(n^3)$.\footnote{Here, we use the existing factorization \eqref{eq_qr_decomp} and solve with $p+1$ right-hand sides, as in \secref{sec_dfbgn_linalg_summary}.}
\end{itemize}
As well as these reductions, we also get a smaller trust-region subproblem \eqref{eq_reduced_trs_generic}---in $\R^p$ rather than $\R^n$---and smaller memory requirements for storing the model Jacobian: we only store $\hat{J}_k$ and $Q_k$, requiring $\bigO((m+n)p)$ memory rather than $\bigO(mn)$ for storing the full $m\times n$ Jacobian.
However, in \eqref{eq_reduced_trs_generic}, we do have the extra cost of projecting $\hat{\bs}_k\in\R^p$ into the full space $\R^n$, which requires a multiplication by $Q_k$, costing $\bigO(np)$.
In addition to the reduced linear algebra costs, the smaller interpolation set means we have a lower evaluation cost to construct the initial model of $p+1$ evaluations (rather than $n+1$).

No particular choice of $p$ is needed for this method, and anything from $p=1$ (i.e.~coordinate search) to $p=n$ (i.e.~full space search) is allowed.
However, unsurprisingly, we shall see that larger values of $p$ give better performance in terms of evaluations, except for the very low-budget phase, where smaller values of $p$ benefit from a lower initialization cost.
Hence, we expect that our approach with small $p$ is useful when the $\bigO(mn^2+n^3)$ per-iteration linear algebra cost of DFO-LS is too great, and reducing the linear algebra cost is worth (possibly) needing more objective evaluations to achieve a given accuracy.
As a result, $p$ should in general be set as large as possible, given the linear algebra costs the user is willing to bear.

In \tabref{tab_linalg_comparison}, we compare the linear algebra costs of DFO-LS and DFBGN.
The overall per-iteration cost of DFO-LS is $\bigO(mn^2+n^3)$ and the cost of DFBGN is $\bigO(mp^2+np^2+p^3)$, depending on the choice of $p\in\{1,\ldots,n\}$.
The key benefit is that our dependency on the underlying problem dimension $n$ decreases from cubic in DFO-LS to linear in DFBGN (provided $p\ll n$).
We also note that both methods have linear cost in the number of residuals $m$, but with a factor that is significantly smaller in DFBGN than in DFO-LS---$\bigO(p^2)$ compared to $\bigO(n^2)$.

\begin{table}[tb]
	\centering
	\footnotesize{
	\begin{tabular}{lccl}
		\hline\noalign{\smallskip}
		Algorithm phase & DFO-LS & DFBGN & Comment \\ \noalign{\smallskip}\hline\noalign{\smallskip}
		Form $\hat{\bem}_k$ \eqref{eq_r_model} & $\bigO(n^3+mn^2)$ & $\bigO(np^2+mp^2)$ & Factorization plus linear solves \\
		Form $\hat{m}_k$ \eqref{eq_f_model} & $\bigO(mn^2)$ & $\bigO(mp^2)$ & Form $\hat{J}_k^T\hat{J}_k$ \\
		Trust-region subproblem & $\bigO(n^2)$--$\bigO(n^3)$ & $\bigO(p^2)$--$\bigO(p^3)$ & Depending on \# CG iterations* \\
		Calculate $\bs_k=Q_k\hat{\bs}_k$ & --- & $\bigO(np)$ & \\
		Form new step $\bx_k+\bs_k$ & $\bigO(n)$ & $\bigO(n)$ & \\ 
		Choose point to replace & $\bigO(n^3)$ & --- & Compute Lagrange polynomials \\ 
		Model improvement & $\bigO(n^3)$ & --- & Recompute Lagrange polynomials \\ 
		Choose points to remove & --- & $\bigO(p^3+np)$ & See \algref{alg_dfbgn_remove_points} \\
		Generate new directions & --- & $\bigO(np^2)$ & See \algref{alg_dfbgn_add_points} \\ \noalign{\smallskip}\hline\noalign{\smallskip}
		\textit{Total} & $\bigO(mn^2+n^3)$ & $\bigO(mp^2+np^2+p^3)$ & \\
		\noalign{\smallskip}\hline\noalign{\smallskip}
	\end{tabular}}
	\caption{Comparison of per-iteration linear algebra costs of DFO-LS and DFBGN (\algref{alg_block_dfols}) with subspace dimension $p\in\{1,\ldots,n\}$. *Note that the trust-region subproblem is solved using a truncated CG method \cite[Chapter 7.5.1]{Conn2000} originally from \cite{Powell2009} in both DFO-LS and DFBGN.}
	\label{tab_linalg_comparison}
\end{table}

\begin{remark}
	In every iteration we must compute the QR factorization \eqref{eq_qr_decomp}.
	However, we note, similar to \cite[Section 4.2]{Cartis2019a}, that adding, removing and changing interpolation points all induce simple changes to $\hat{W}_k^T$ (adding or removing columns, and low-rank updates).
	This means that \eqref{eq_qr_decomp} can be computed with cost $\bigO(np)$ per iteration using the updating methods in \cite[Section 12.5]{Golub1996}.
	In our implementation, however, we do not do this, as we find that these updates introduce errors\footnote{Leading to $\hat{W}_k^T\neq Q_k R_k$, not relating to $Q_k$ orthogonal or $R_k$ upper-triangular.} that accumulate at every iteration and reduce the accuracy of the resulting interpolation models.
	To maintain the numerical performance of our method, we need to recompute \eqref{eq_qr_decomp} from scratch regularly (e.g.~every 10 iterations), and so would not see the $\bigO(np)$ per-iteration cost, on average.
\end{remark}

\begin{remark}
	The default parameter choices for DFBGN are the same as DFO-LS, namely: $\Delta_{\rm max}=10^{10}$, $\Delta_{\rm end}=10^{-8}$, $\gamma_{\rm dec}=0.5$, $\gamma_{\rm inc}=2$, $\overline{\gamma}_{\rm inc}=4$, $\eta_1=0.1$, and $\eta_2=0.7$.
	DFBGN also uses the same default choice $\Delta_0=0.1\max(\|\bx_0\|_{\infty},1)$.
	The default choice of $p_{\rm drop}$ is discussed in \secref{sec_scalability_practicality}.
\end{remark}

\paragraph{Adaptive Choice of $\bm{p}$}
One approach that we have considered is to allow $p$ to vary between iterations of DFBGN, rather than being constant throughout.
Instead of adding $p_{\rm drop}$ new points at the end of each iteration (line \ref{ln_new_dirns}), we implement a variable $p$ by adding at least one new point to the interpolation set, continuing until some criterion is met.
This criterion is designed to allow $p$ small when such a $p$ allows us to make reasonable progress, but to grow $p$ up to $p\approx n$ when necessary.

We have tested several possible criteria---comparing some combination of model gradient and Hessian, trust-region radius, trust-region step length, and predicted decrease from the trust-region step---and found the most effective to be comparing the model gradient and Hessian with the trust-region radius.
Specifically, we continue adding new directions until (c.f.~\lemref{lem_epsilon_g} and \cite[Lemma 3.22]{Cartis2019a})
\be \frac{\|\bg_k\|}{\max(\|H_k\|,1)} \geq \alpha \Delta_k, \ee
for some $\alpha>0$ (we use $\alpha=0.2(n-p)/n$ for an interpolation set with $p+1$ points).
However, our numerical testing has shown that DFBGN with $p$ fixed outperforms this approach for all budget and accuracy levels, on both medium- and large-scale problems, and so we do not consider it further here.
We delegate further study of this approach to future work, to see if alternative adaptive choices for $p$ can be beneficial.

\subsection{Interpolation Set Management} \label{sec_dfbgn_interp_set}
We now provide more details about how we manage the interpolation set in DFBGN. 
Specifically, we discuss how points are chosen for removal from $Y_k$, and how new interpolations points are calculated.

\subsubsection{Geometry Management} \label{sec_dfbgn_geom_management}
In the description of DFBGN, there are no explicit mechanisms to ensure that the interpolation set is well-poised.
DFBGN ensures that the interpolation set has good geometry through two mechanisms:
\begin{itemize}
	\item We use a geometry-aware mechanism for removing points, based on \cite{Powell2009,Cartis2019a}, which requires the computation of Lagrange polynomials.
	This mechanism is given in \algref{alg_dfbgn_remove_points}, and is called in lines \ref{ln_drop1} and \ref{ln_drop2} of DFBGN, as well as to select a point to replace in line \ref{ln_replace}; and
	\item Adding new directions that are orthogonal to existing directions, and of length $\Delta_k$, means adding these new points never causes the interpolation set to have poor poisedness.
\end{itemize}
Together, these two mechanisms mean that any points causing poor poisedness are quickly removed, and replaced by high-quality interpolation points (orthogonal to existing directions, and within distance $\Delta_k$ of the current iterate).

\begin{algorithm}[t]
	\footnotesize{
	\begin{algorithmic}[1]
		\Require Interpolation set $\{\bx_{k+1},\by_1,\ldots,\by_p\}$ with current iterate $\bx_{k+1}$, trust-region radius $\Delta_{k+1}>0$ and number of points to remove $p_{\rm drop}\in\{1,\ldots,p\}$. 
		\vspace{0.5em}
		\State Compute the (linear) Lagrange polynomials for $\{\bx_{k+1},\by_1,\ldots,\by_p\}$ in the same way as \eqref{eq_reduced_interp_system}.
		\State For $t=1,\ldots,p$ (i.e.~all interpolation points except $\bx_{k+1}$), compute
		\be \theta_t \defeq \max_{\bx\in B(\bx_{k+1},\Delta_{k+1})} |\ell_t(\bx)|\cdot\max\left(\frac{\|\by_t-\bx_{k+1}\|^4}{\Delta_{k+1}^4}, 1\right). \label{eq_drop_thresh} \ee
		\State Remove the $p_{\rm drop}$ interpolation points with the largest values of $\theta_t$.
	\end{algorithmic}
	} 
	\caption{Mechanism for removing points from the interpolation set in DFBGN.}
	\label{alg_dfbgn_remove_points}
\end{algorithm}

The linear algebra cost of \algref{alg_dfbgn_remove_points} is $\bigO(p^3)$ to compute $p$ Lagrange polynomials with cost $\bigO(p^2)$ each (since we already have a factorization of $\hat{W}_k^T$).
Then for each $t$ we must evaluate $\theta_t$ \eqref{eq_drop_thresh}, with cost $\bigO(p)$ to maximize $\ell_t(\bx)$ (since $\ell_t$ is linear and varies only in directions $\Y_k$), and $\bigO(n)$ to calculate $\|\by_t-\bx_{k+1}\|$.
This gives a total cost of $\bigO(p^3+np)$.\footnote{We could instead compute $\|\by_t-\bx_{k+1}\|$ by taking the norm of the $t$-th column of $R_k$, provided we have the factorization \eqref{eq_qr_decomp}, for cost $\bigO(p)$ for each $t$. This does not affect the overall conclusion of \tabref{tab_linalg_comparison}.}

\begin{figure}[t]
	\centering
	\begin{subfigure}[b]{0.48\textwidth}
		\includegraphics[width=\textwidth]{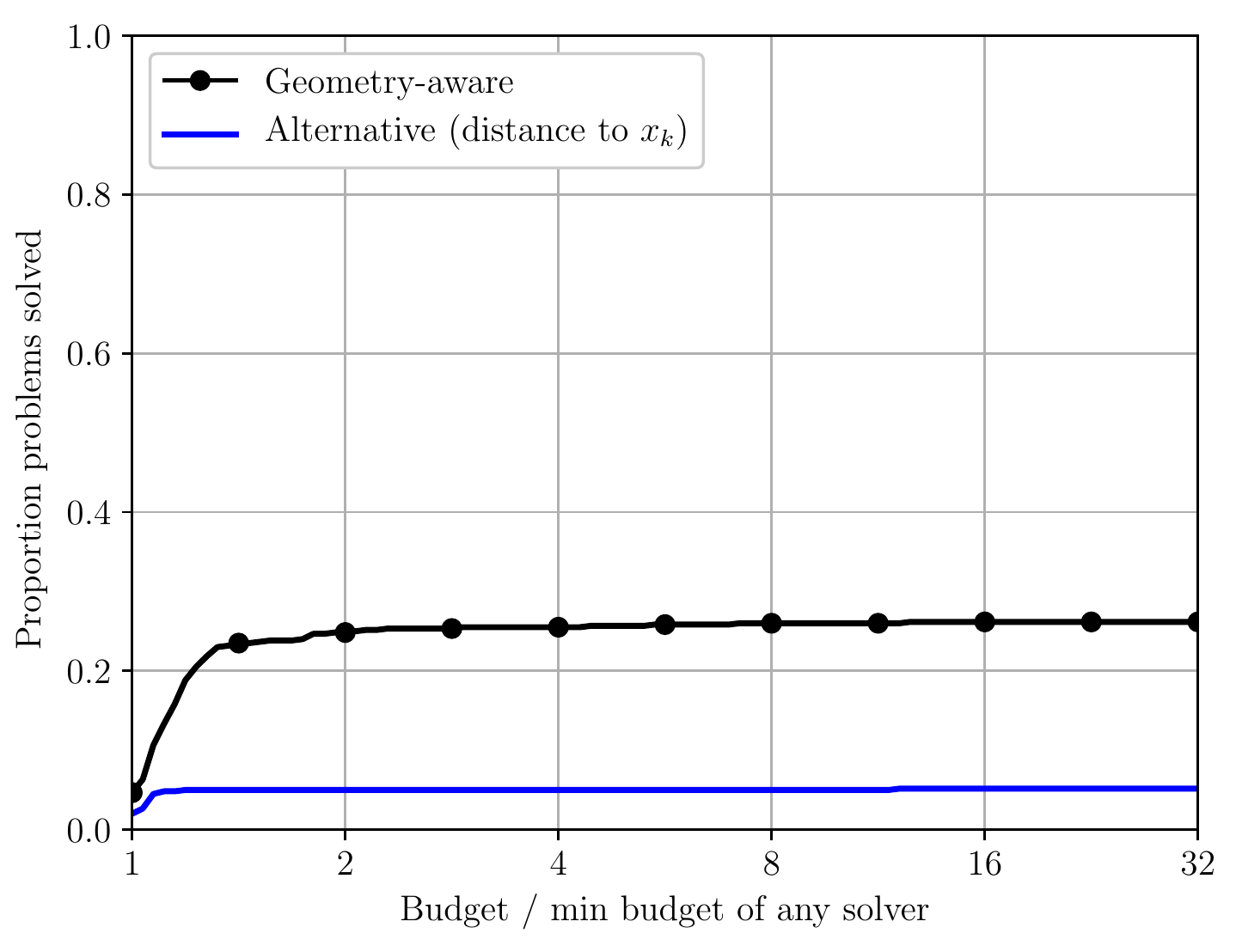}
		\caption{DFBGN with $p=n/10$}
	\end{subfigure}
	~
	\begin{subfigure}[b]{0.48\textwidth}
		\includegraphics[width=\textwidth]{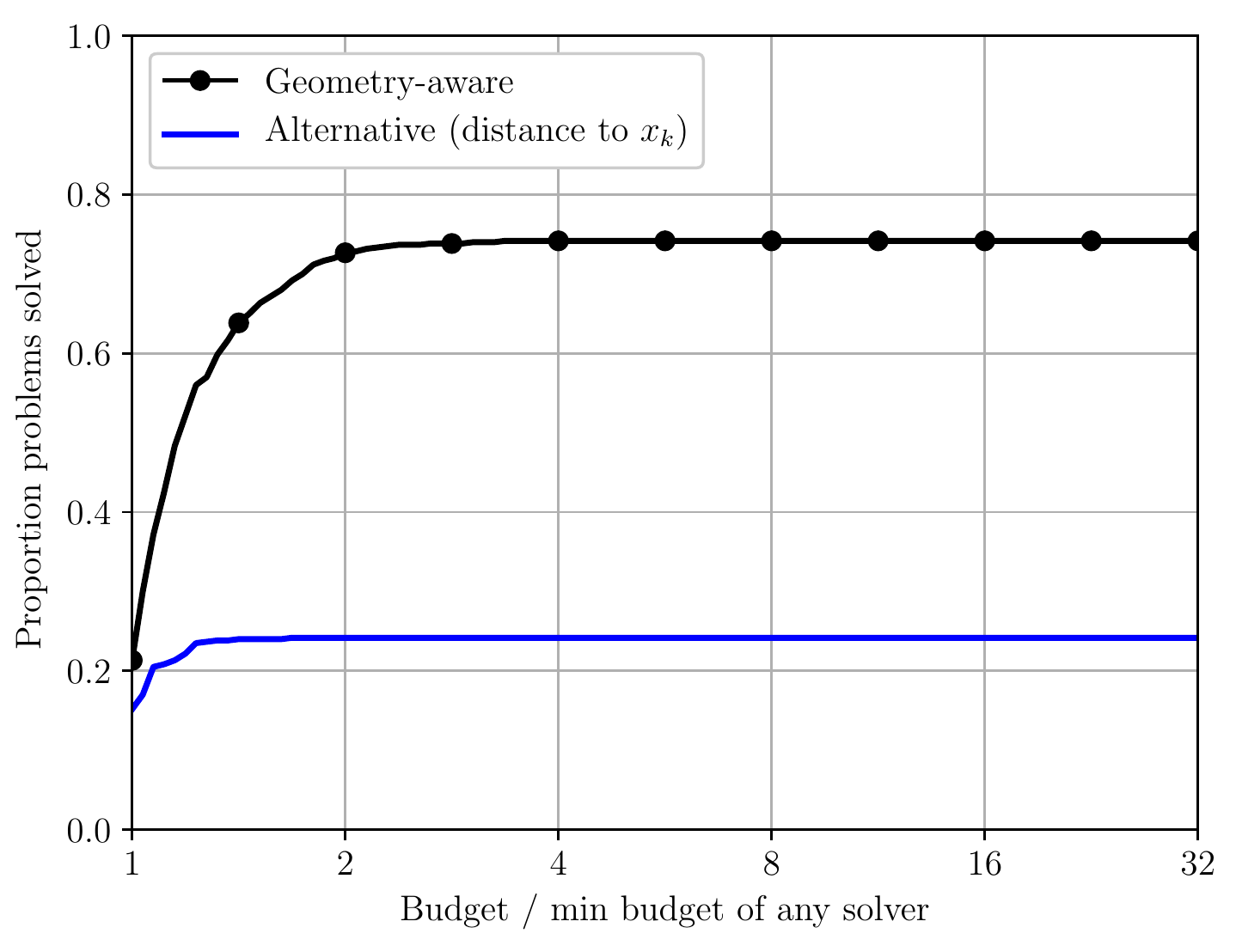}
		\caption{DFBGN with $p=n$}
	\end{subfigure}
	\caption{Performance profiles (in evaluations) for DFBGN when $p=n/10$ and $p=n$, comparing removing points with the geometry-aware \algref{alg_dfbgn_remove_points} and without Lagrange polynomials (by distance to the current iterate). We use accuracy level $\tau=10^{-5}$, and results are an average of 10 runs, each with a budget of $100(n+1)$ evaluations. The problem collection is (CR). See \secref{sec_numerics_framework} for details on the testing framework.}
	\label{fig_geom_motivation}
\end{figure}

\paragraph{Alternative Point Removal Mechanism}
Instead of \algref{alg_dfbgn_remove_points}, we could have used a simpler mechanism for removing points, such as removing the points furthest from the current iterate (with total cost\footnote{As above, if we have \eqref{eq_qr_decomp}, we could calculate all distances to $\bx_{k+1}$ using columns of $R_k$, with total cost $\bigO(p^2)$.} $\bigO(np)$).
However, this leads to a substantial performance penalty.
In \figref{fig_geom_motivation}, we compare these two approaches for selecting points to be removed, namely \algref{alg_dfbgn_remove_points} and distance to $\bx_{k+1}$, on the (CR) test set with $p=n/10$ and $p=n$ (using the default value of $p_{\rm drop}$, as detailed in \secref{sec_scalability_practicality}).
For more details on the numerical testing framework, see \secref{sec_numerics_framework} below.
We see that the geometry-aware criterion \eqref{eq_drop_thresh} gives substantially better performance than the cheaper criterion.

\subsubsection{Generation of New Directions} \label{sec_dfbgn_new_directions}
We now detail how new directions $\b{d}_1,\ldots,\b{d}_q$ are created in line \ref{ln_new_dirns} of DFBGN (\algref{alg_block_dfols}).
The same approach is suitable for generating the initial directions $\bs_1,\ldots,\bs_p$ in line \ref{ln_dfbgn_init} of DFBGN, using $\t{A}=A$ below (i.e.~no $Q$ required).

\begin{algorithm}[t]
	\footnotesize{
	\begin{algorithmic}[1]
		\Require Orthonormal basis for current subspace $Q\in\R^{n\times p_1}$ (optional), number of new directions $q\leq n-p_1$. 
		\vspace{0.5em}
		\State Generate $A\in\R^{n\times q}$ with i.i.d.~standard normal entries. 
		\State If $Q$ is specified, calculate $\t{A}=A-QQ^T A$, otherwise set $\t{A}=A$. 
		\State Perform the QR factorization $\t{Q}\t{R}=\t{A}$ and return $\b{d}_1,\ldots,\b{d}_q$ as the columns of $\t{Q}$.
	\end{algorithmic}
	} 
	\caption{Mechanism for generating new directions in DFBGN.}
	\label{alg_dfbgn_add_points}
\end{algorithm}

Suppose our current subspace is defined by the orthonormal columns of $Q\in\R^{n\times p_1}$, and we wish to generate $q$ new orthonormal vectors that are also orthogonal to the columns of $Q$ (with $p_1+q\leq n$). 
When called in line \ref{ln_new_dirns} of DFBGN, we will have $p_1=p-p_{\rm drop}$ and $q=p_{\rm drop}$.
We use the approach in \algref{alg_dfbgn_add_points}.
From the QR factorization, the columns of $\t{Q}$ are orthonormal, and if $\t{A}$ is full rank (which occurs with probability 1; see \lemref{lem_full_rank_prob1} below) then we also have $\col(\t{Q}) = \col(\t{A})$.
So, to confirm the columns of $\t{Q}$ are orthogonal to $Q$, we only need to check that the columns of $\t{A}$ are orthogonal to $Q$.
Let $\b{\t{a}}_i$ be the $i$-th column of $\t{A}$ and $\b{q}_j$ be the $j$-th column of $Q$.
Then, if $\b{a}_i$ is the $i$-th column of $A$, we have
\be \b{\t{a}}_i^T\b{q}_j = \b{a}_i^T(I-QQ^T)\b{q}_j = \b{a}_i^T(\b{q}_j-Q\bee_j) = 0, \ee
as required.

The cost of \algref{alg_dfbgn_add_points} is $\bigO(nq)$ to generate $A$, $\bigO(np_1 q)$ to form $\t{A}$ and $\bigO(nq^2)$ for the QR factorization.
Since $p_1,q\leq p$ (since $p_1$ is the number of directions remaining in the interpolation set and $q$ is the number of new directions to be added), the whole process has cost at most $\bigO(np^2)$.
This bound is tight, up to constant factors, as we could take $p_1=q=p/2$, for instance.

\begin{lemma} \label{lem_full_rank_prob1}
	The matrix $\t{A}$ has full column rank with probability 1.
\end{lemma}
\begin{proof}
	Let $\b{a}_i$ and $\b{\t{a}}_i$ be the $i$-th columns of $A$ and $\t{A}$ respectively.
	From \cite[Proposition 7.1]{Eaton2007}, $A$ has full column rank with probability 1, and each $\b{a}_i\notin\col(Q)$ with probability 1.
	Now suppose we have constants $c_1,\ldots,c_q$ so that $\sum_{i=1}^{q}c_i\b{\t{a}}_i=\b{0}$.
	Then since $\t{\b{a}}_i=\b{a}_i-QQ^T\b{a}_i$, we have
	\be \sum_{i=1}^{q}c_i\b{a}_i = \sum_{i=1}^{q}c_i QQ^T\b{a}_i. \ee
	The right-hand side is in $\col(Q)$, so since $\b{a}_i\notin\col(Q)$, we must have $\sum_{i=1}^{q}c_i\b{a}_i=\b{0}$.
	Thus $c_1=\cdots=c_q=0$ since $A$ has full column rank, and so $\t{A}$ has full column rank.
\end{proof}

\subsection{Selecting an Appropriate Value of $\bm{p_{\rm drop}}$} \label{sec_scalability_practicality}
An important component of DFBGN that we have not yet specified is how many points to remove from the interpolation set at each iteration, $p_{\rm drop}\in\{1,\ldots,p\}$.

On one hand, a large $p_{\rm drop}$ enables us to change the subspace by a large amount between iterations, ensuring we explore the whole of $\R^n$ quickly, rather than searching in unproductive subspaces for many iterations.
However, a small $p_{\rm drop}$ means we require few objective evaluations per iteration, and so are more likely to use our evaluation budget efficiently.
We consider two choices of $p_{\rm drop}$, aimed at each of these possible benefits: $p_{\rm drop}=p/10$ to change subspaces quickly, and $p_{\rm drop}=1$ (the minimum possible value) to use few objective evaluations.

Another approach that we consider is a compromise between these two choices.
We note that having $p_{\rm drop}=1$ is useful to make progress with few evaluations, so we use this value while we are making progress---we consider this to occur when we have a successful iteration (i.e.~$\rho_k\geq\eta_1$).
When we are not progressing (i.e.~unsuccessful steps with $\rho_k<\eta_1$), we use the larger value $p_{\rm drop}=p/10$.

\begin{figure}[t]
	\centering
	\begin{subfigure}[b]{0.48\textwidth}
		\includegraphics[width=\textwidth]{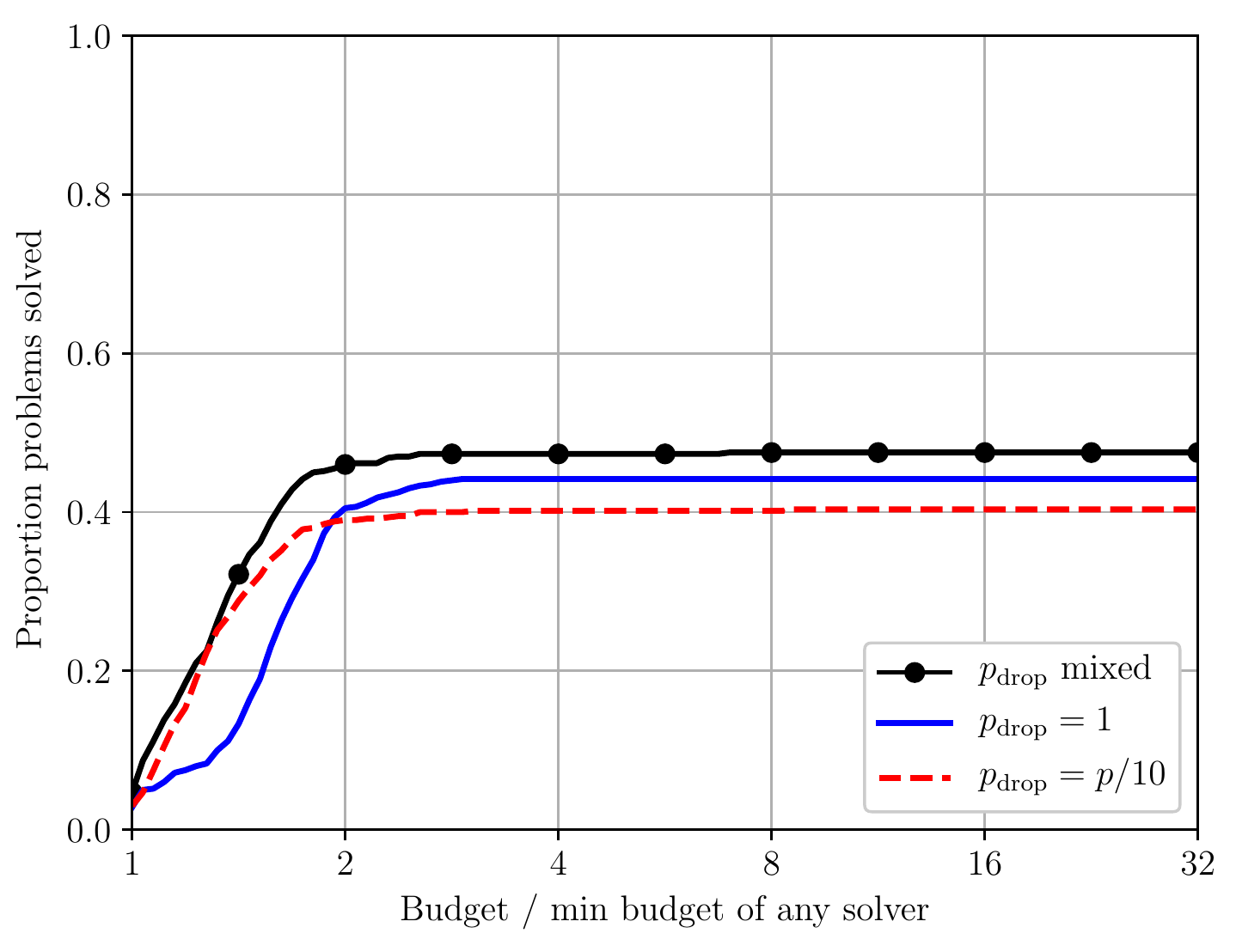}
		\caption{DFBGN with $p=n/2$}
	\end{subfigure}
	~
	\begin{subfigure}[b]{0.48\textwidth}
		\includegraphics[width=\textwidth]{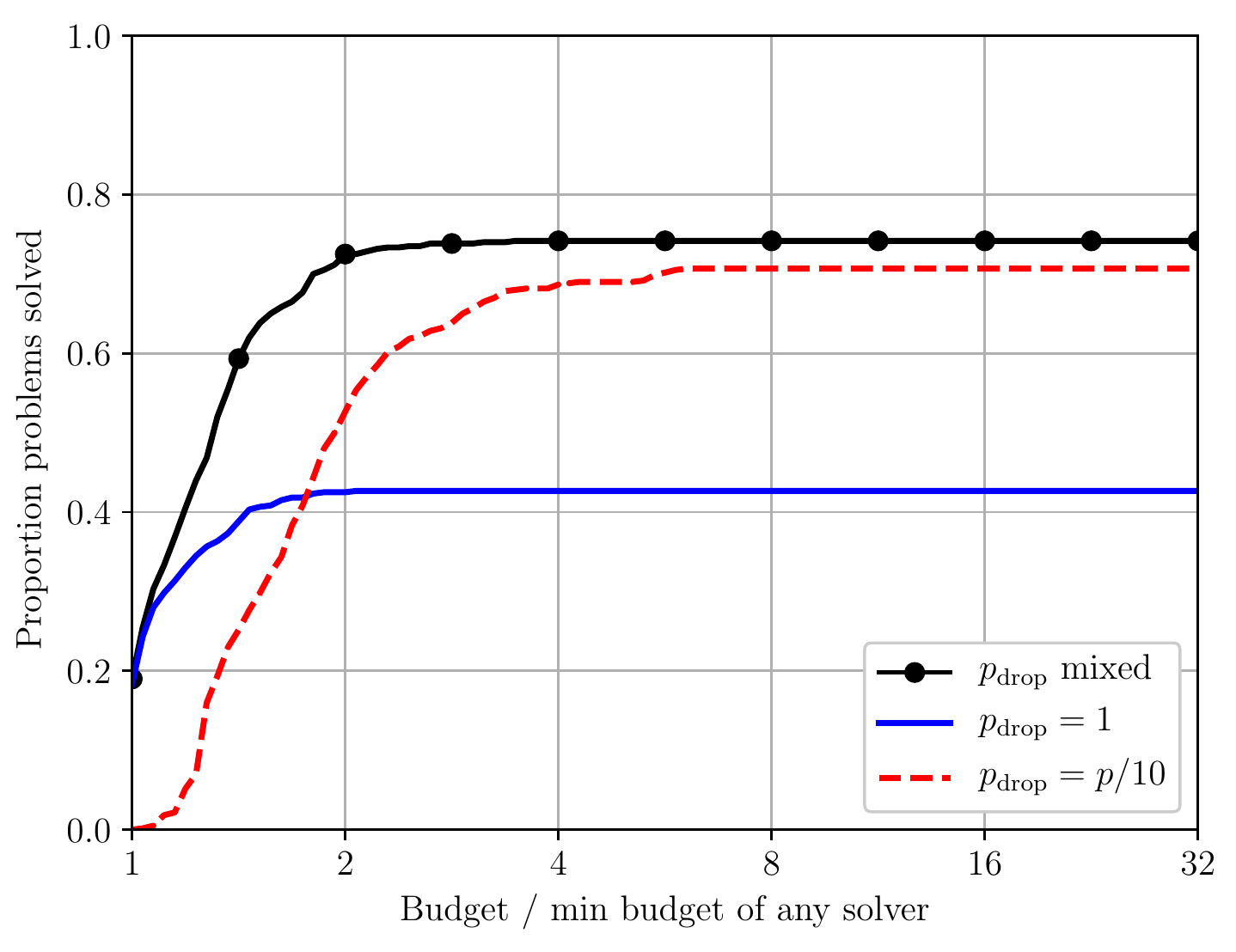}
		\caption{DFBGN with $p=n$}
	\end{subfigure}
	\caption{Performance profiles (in evaluations) comparing different choices of $p_{\rm drop}$, for DFBGN with $p=n/2$ and $p=n$, with accuracy level $\tau=10^{-5}$. The choice `$p_{\rm drop}$ mixed' uses $p_{\rm drop}=1$ for successful iterations and $p_{\rm drop}=p/10$ for unsuccessful iterations. Results an average of 10 runs, each with a budget of $100(n+1)$ evaluations. The problem collection is (CR). See \secref{sec_numerics_framework} for details on the testing framework.}
	\label{fig_drop_mixed_motivation}
\end{figure}

We compare these three approaches on the (CR) problem collection with a budget of $100(n+1)$ evaluations in \figref{fig_drop_mixed_motivation}.
Since these different choices of $p_{\rm drop}$ are similar when $p$ is small, we show results for subspace dimensions $p=n/2$ and $p=n$.
We first see that, even with $p=n/2$, the three approaches all perform similarly. 
However, for $p=n$ the compromise choice $p_{\rm drop}\in\{1,p/10\}$ performs better than the two constant-$p$ approaches.
In addition, $p_{\rm drop}=1$ outperforms $p_{\rm drop}=p/10$ for small performance ratios, but is less robust and solves fewer problems overall.

Given these results, in DFBGN we use the compromise choice as the default mechanism: $p_{\rm drop}=1$ on successful iterations and $p_{\rm drop}=p/10$ on unsuccessful iterations.

\paragraph{Relationship to model-improvement phases}
The \texttt{CHECK\_MODEL} flag in RSDFO-GN is important for ensuring we do not reduce $\Delta_k$ too quickly without first ensuring the quality of the interpolation model.\footnote{This is related to ensuring $\Delta_k$ does not get too small compared to $\|\hat{\bg}_k\|$ via \lemref{lem_eventually_successful}.}
For a similar purpose, DFO-LS incorporates a second trust-region radius which also is involved with ensuring $\Delta_k$ does not decrease too quickly \cite{Cartis2018}.
In DFBGN, as described in \secref{sec_dfbgn_geom_management}, we maintain the geometry of the interpolation set by replacing poorly-located points with orthogonal directions around the current iterate; in practice this ensures the quality of the interpolation set.
However, the choice of $p_{\rm drop}$ has a large impact on causing $\Delta_k$ to shrink too quickly.

In many cases, DFBGN may reach a point where its model is not accurate and we start to have unsuccessful iterations.
To fix this (and continue making progress), we need to introduce several new interpolation points to produce a high-quality model.
If $p_{\rm drop}$ is small, this may take many unsuccessful iterations, causing $\Delta_k$ to decrease quickly.

\begin{figure}[t]
	\centering
	\begin{subfigure}[b]{0.48\textwidth}
		\includegraphics[width=\textwidth]{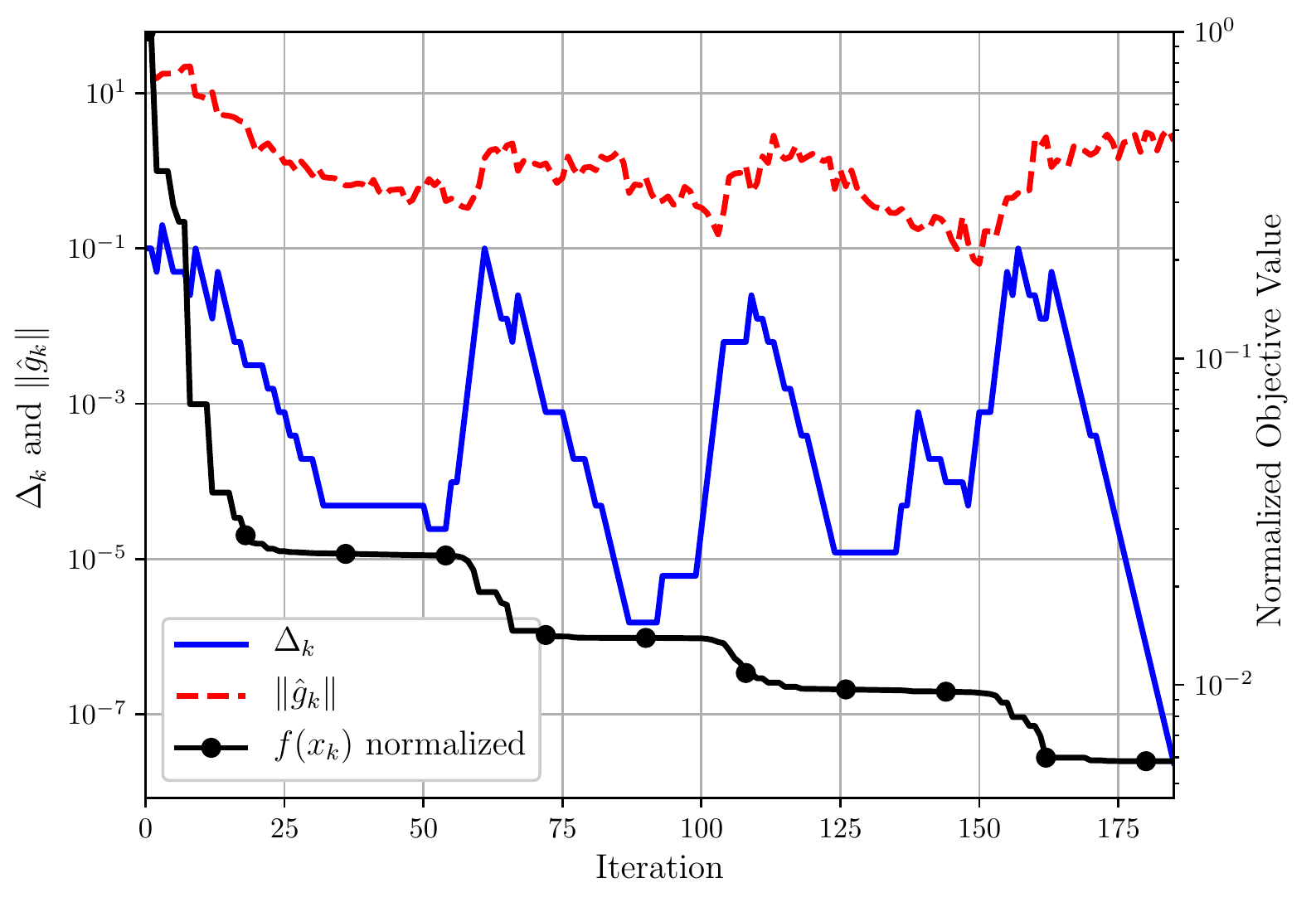}
		\caption{{\sc drcavty1}, $p_{\rm drop}=1$}
	\end{subfigure}
	~
	\begin{subfigure}[b]{0.48\textwidth}
		\includegraphics[width=\textwidth]{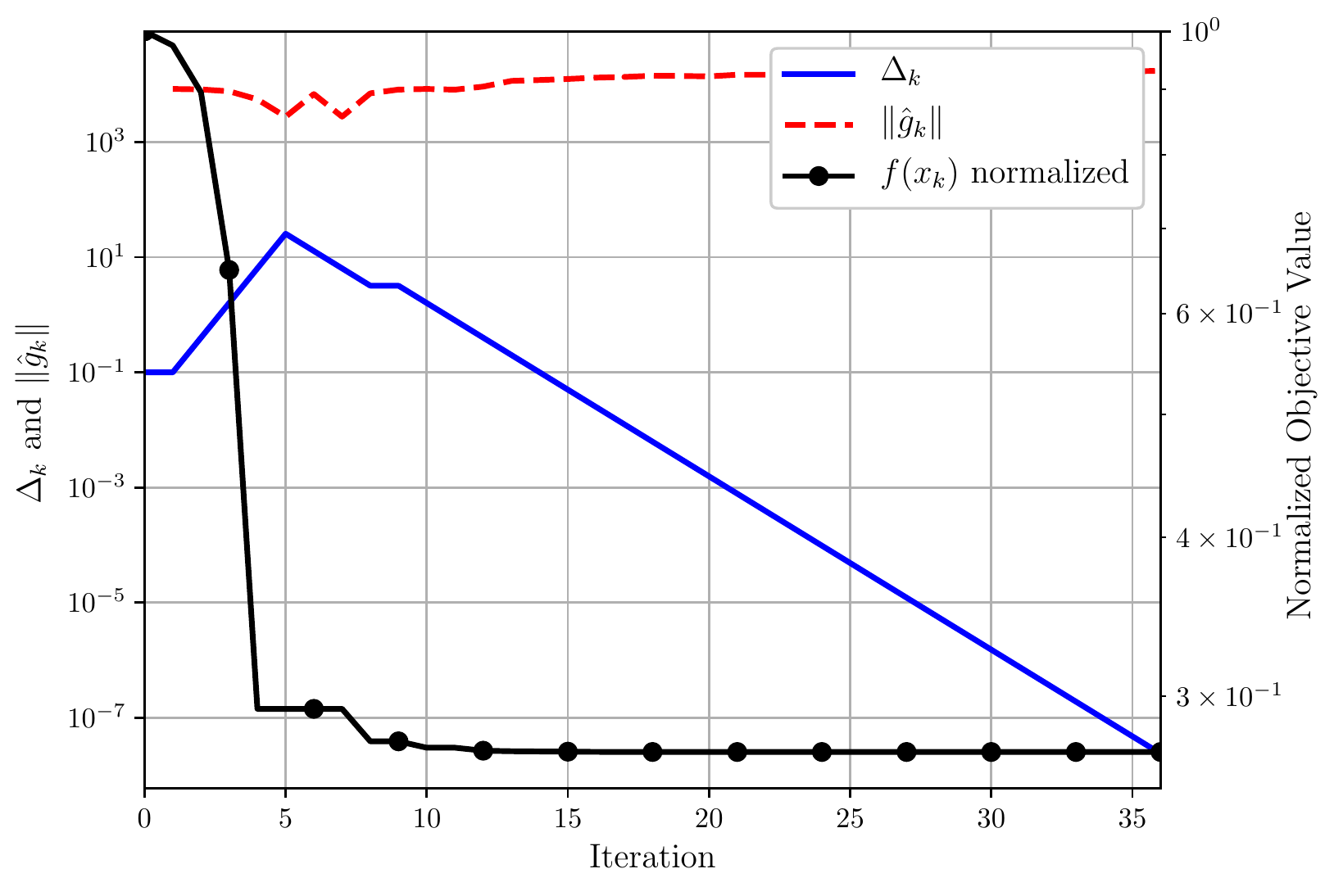}
		\caption{{\sc luksan13}, $p_{\rm drop}=1$}
	\end{subfigure}
	\caption{Comparison of $\Delta_k$, $\|\hat{\bg}_k\|$ (left $y$-axis) and normalized objective value (right $y$-axis) for DFBGN with $p=n$ and $p_{\rm drop}=1$. The two problems are taken from the (CR) collection. See \secref{sec_numerics_framework} for details on the testing framework.}
	\label{fig_drop_mixed_motivation1}
\end{figure}

The result of having $p_{\rm drop}$ small is seen in \figref{fig_drop_mixed_motivation1}.
Here, we show $\Delta_k$, $\|\hat{\bg}_k\|$ and $f(\bx_k)$ for DFBGN with $p=n$ and $p_{\rm drop}=1$ for two problems from the (CR) collection.
Both problems show that $\Delta_k$ can quickly shrink to be much smaller than $\|\hat{\bg}_k\|$ before reaching optimality.
In the case of {\sc drcavty1}, we see multiple instances where, after several unsuccessful iterations, we recover a high-quality model and continue making progress (causing $\Delta_k$ to increase again); this manifests itself as large oscillations in $\Delta_k$ with comparatively little change in $\|\hat{\bg}_k\|$.
Ultimately, as we terminate on $\Delta_k\leq\Delta_{\rm end}=10^{-8}$, DFBGN quits without solving the problem (reaching accuracy $\tau\approx 6\times10^{-3}$).
A more extreme version of this behaviour is seen for problem {\sc luksan13}, where we terminate on small $\Delta_k$ in the first sequence of unsuccessful iterations---DFBGN does not allow enough time to recover a high-quality model and terminates after achieving accuracy $\tau\approx 0.3$.

\begin{figure}[t]
	\centering
	\begin{subfigure}[b]{0.48\textwidth}
		\includegraphics[width=\textwidth]{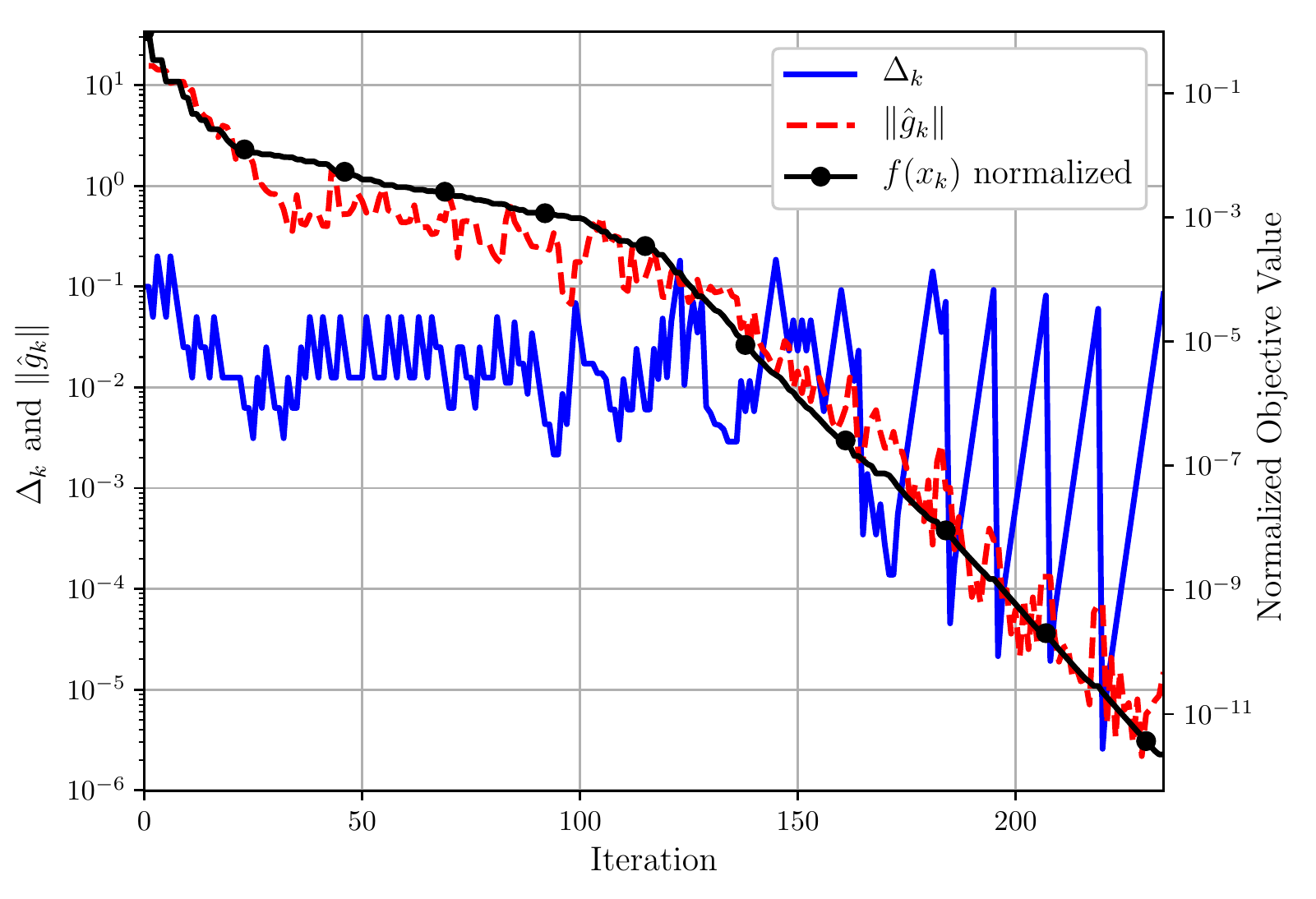}
		\caption{{\sc drcavty1}, $p_{\rm drop}$ mixed}
	\end{subfigure}
	~
	\begin{subfigure}[b]{0.48\textwidth}
		\includegraphics[width=\textwidth]{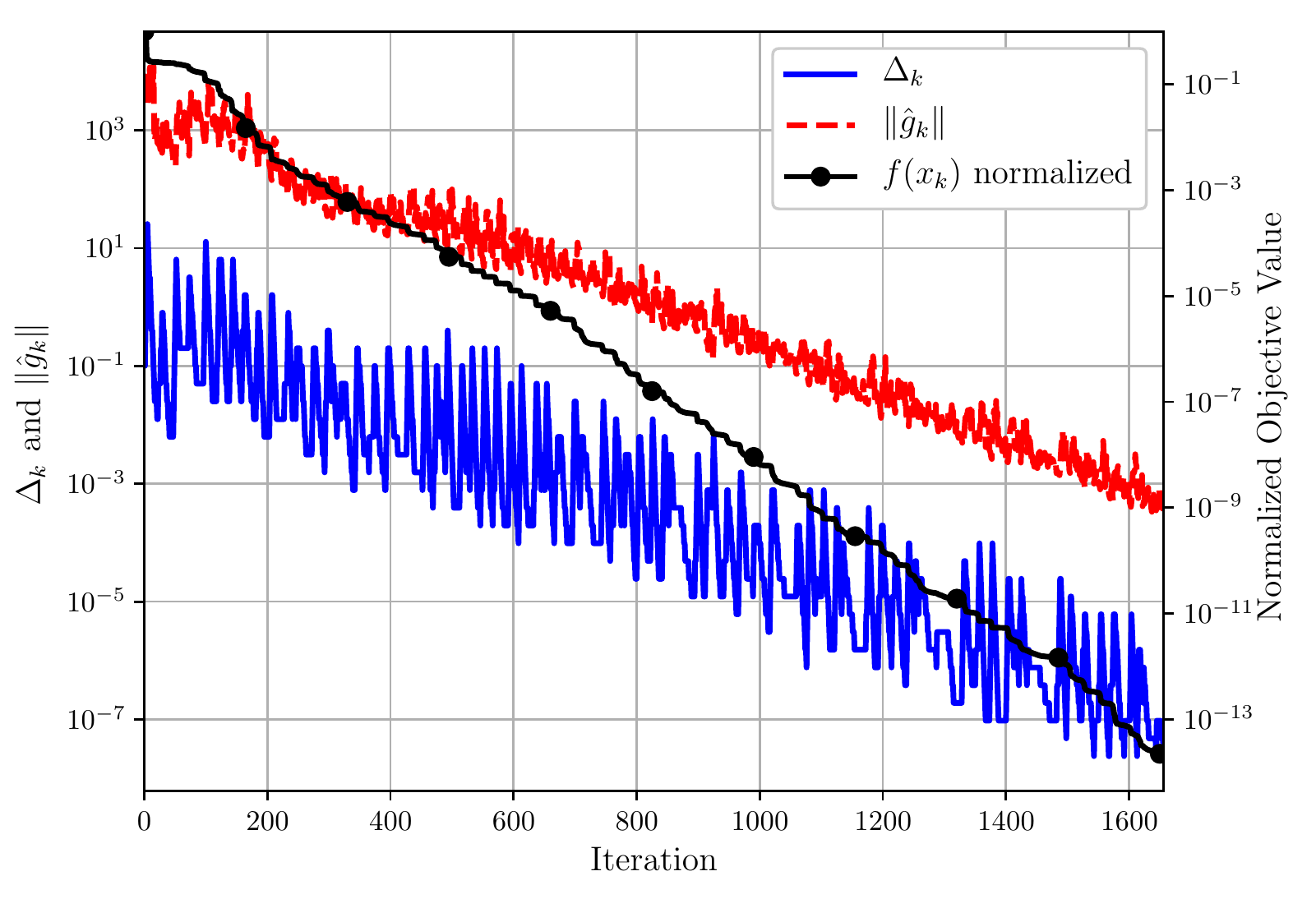}
		\caption{{\sc luksan13}, $p_{\rm drop}$ mixed}
	\end{subfigure}
	\caption{Comparison of $\Delta_k$, $\|\hat{\bg}_k\|$ (left $y$-axis) and normalized objective value (right $y$-axis) for DFBGN with $p=n$ and the default $p_{\rm drop}\in\{1,p/10\}$. The two problems are taken from the (CR) collection. See \secref{sec_numerics_framework} for details on the testing framework.}
	\label{fig_drop_mixed_motivation2}
\end{figure}

This effect is mitigated by our default choice of $p_{\rm drop}\in\{1,p/10\}$.
By using a larger $p_{\rm drop}$ on unsuccessful iterations, when our model is performing poorly, our interpolation set is changed quickly.
This results in DFBGN recovering a high-quality model after a smaller decrease in $\Delta_k$.
To demonstrate this, in \figref{fig_drop_mixed_motivation2} we show the results of DFBGN with this $p_{\rm drop}$ for the same problems as \figref{fig_drop_mixed_motivation1} above.
In both cases, we still see oscillations in $\Delta_k$, but their magnitude is substantially reduced---it takes fewer iterations to get successful steps, and $\Delta_k$ stays well above $\Delta_{\rm end}$.
This leads to both problems being solved to high accuracy.

In \figref{fig_drop_mixed_motivation2}, we also see that, as we approach the solution, $\|\hat{\bg}_k\|$ and $\Delta_k$ decrease at the same rate, as we would hope.
For {\sc drcavty1} after iteration 150, we also see the phenomenon described above, where $\Delta_k$ can become much larger than $\|\hat{\bg}_k\|$ due to many successful iterations, before an unsuccessful iteration with $\|\bs_k\|$ small means that $\Delta_k$ returns to the level of $\|\hat{\bg}_k\|$ regularly.

\paragraph{Alternative Mechanism for Recovering High-Quality Models}
An alternative approach for avoiding unnecessary decreases in $\Delta_k$ while the interpolation model quality is improved is to simply decrease $\Delta_k$ more slowly on unsuccessful iterations.
This corresponds to setting $\gamma_{\rm dec}$ to be closer to 1, which is the default choice of DFO-LS for noisy problems (see \cite[Section 3.1]{Cartis2018}), and aligns with our theoretical requirements on the trust-region parameters (\thmref{thm_high_prob_complexity}).

In \figref{fig_noisy_comparison}, we compare the DFBGN default choices, of $p_{\rm drop}\in\{1,p/10\}$ and $\gamma_{\rm dec}=0.5$, with $p_{\rm drop}=1$ and $\gamma_{\rm dec}\in\{0.5,0.98\}$ on the (CR) problem collection.
For small values of $p$ (where the different choices of $p_{\rm drop}$ are essentially identical), the choice of $\gamma_{\rm dec}$ has almost no impact on the performance of DFBGN.
For larger values of $p$, using $\gamma_{\rm dec}=0.98$ with $p_{\rm drop}=1$ performs comparably well to the DFBGN default ($\gamma_{\rm dec}=0.5$ with $p_{\rm drop}\in\{1,p/10\}$).
However, we opt for keeping $\gamma_{\rm dec}=0.5$, to allow us to use the larger value for noisy problems (just as in DFO-LS), and to reduce the risk of overfitting our trust-region parameters to a particular problem collection.

\begin{figure}[t]
	\centering
	\begin{subfigure}[b]{0.48\textwidth}
		\includegraphics[width=\textwidth]{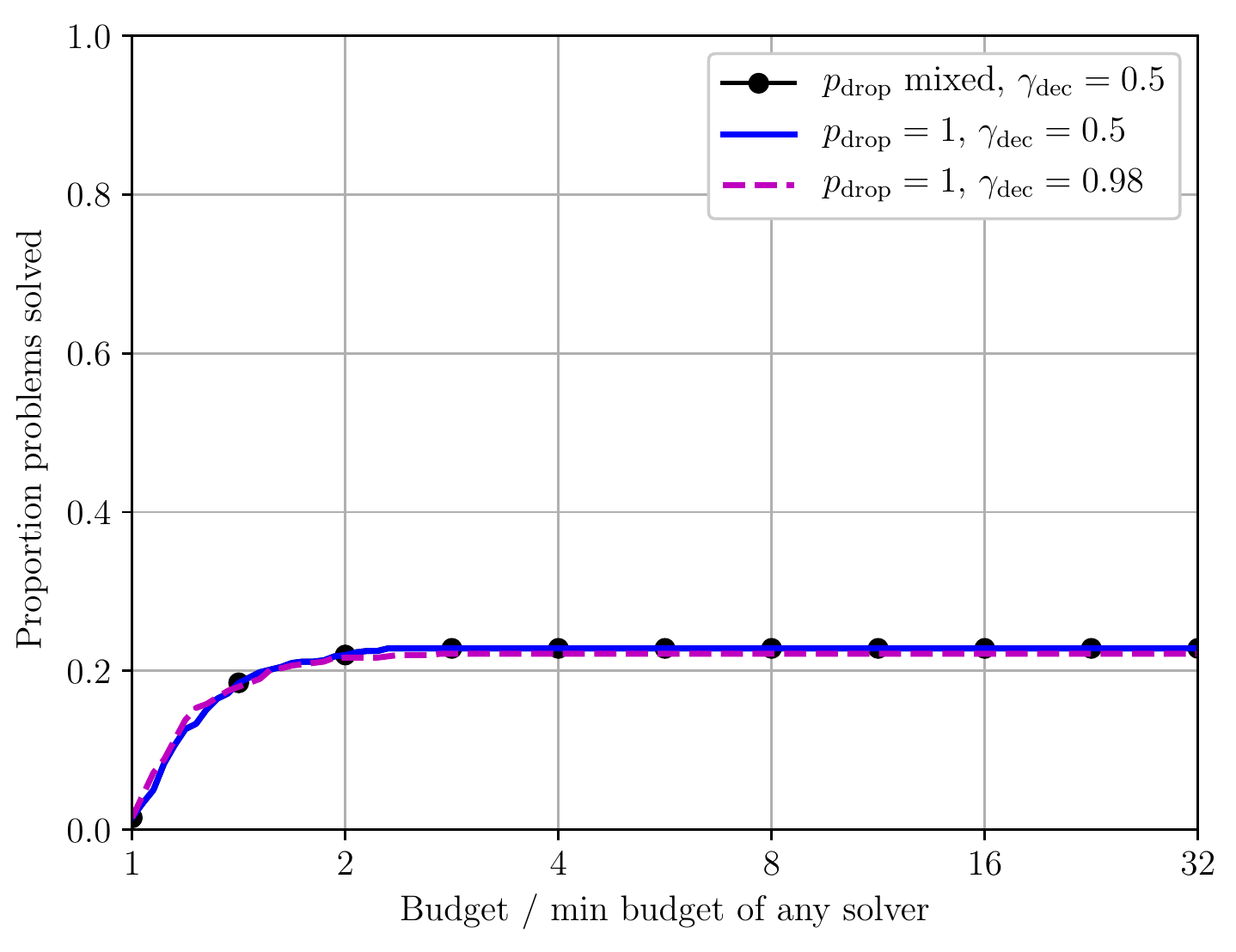}
		\caption{DFBGN with $p=n/100$}
	\end{subfigure}
	~
	\begin{subfigure}[b]{0.48\textwidth}
		\includegraphics[width=\textwidth]{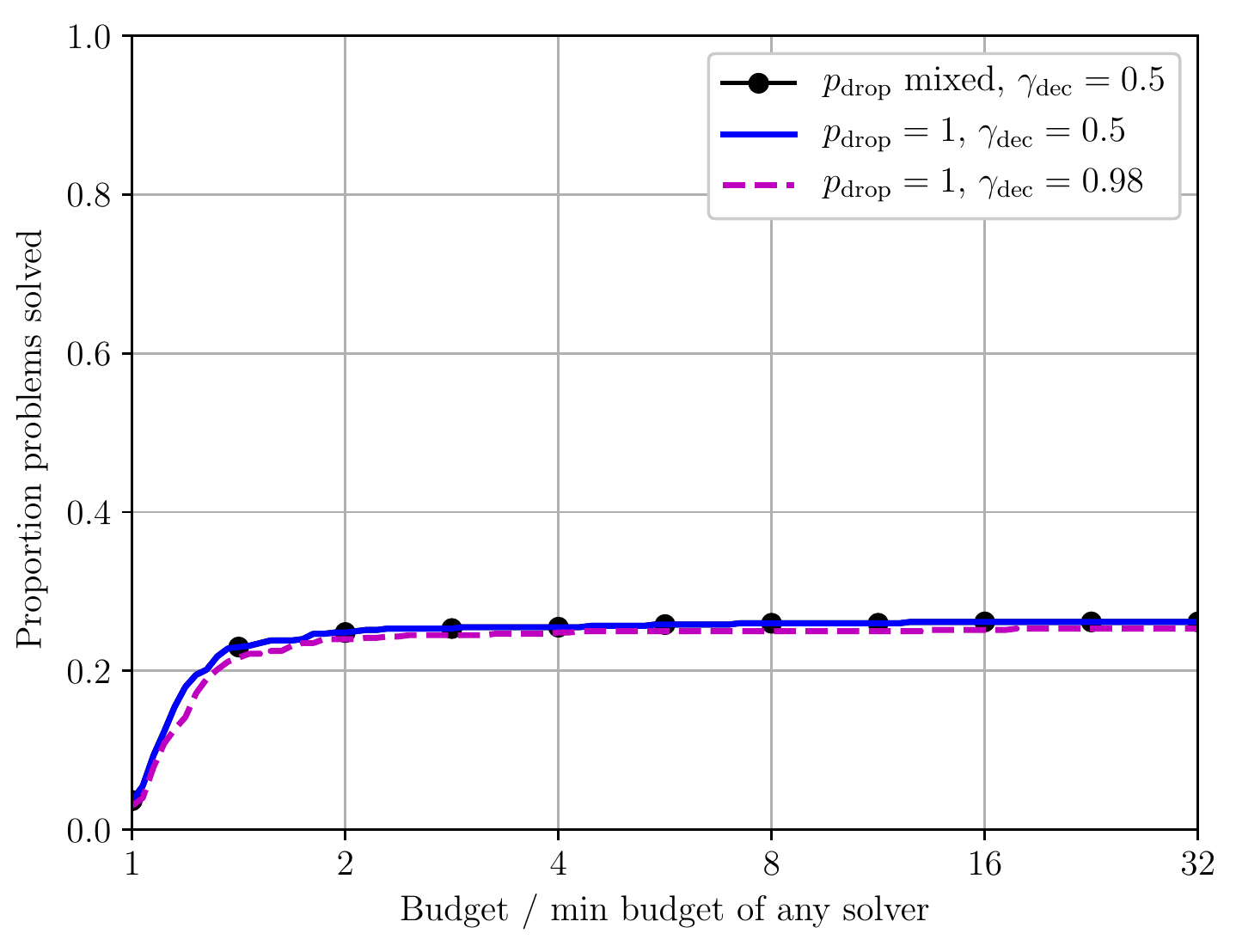}
		\caption{DFBGN with $p=n/10$}
	\end{subfigure}
	\\
	\begin{subfigure}[b]{0.48\textwidth}
		\includegraphics[width=\textwidth]{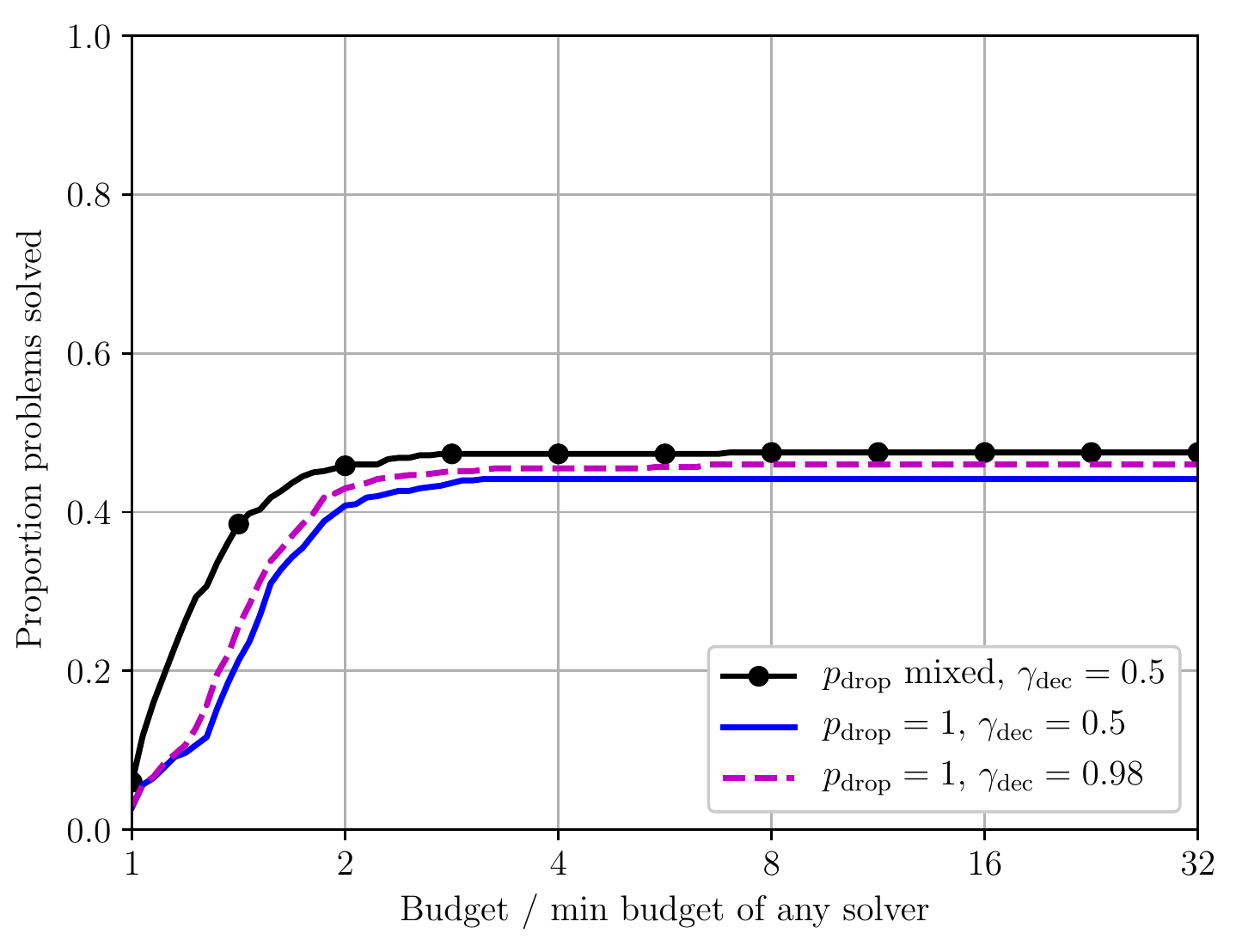}
		\caption{DFBGN with $p=n/2$}
	\end{subfigure}
	~
	\begin{subfigure}[b]{0.48\textwidth}
		\includegraphics[width=\textwidth]{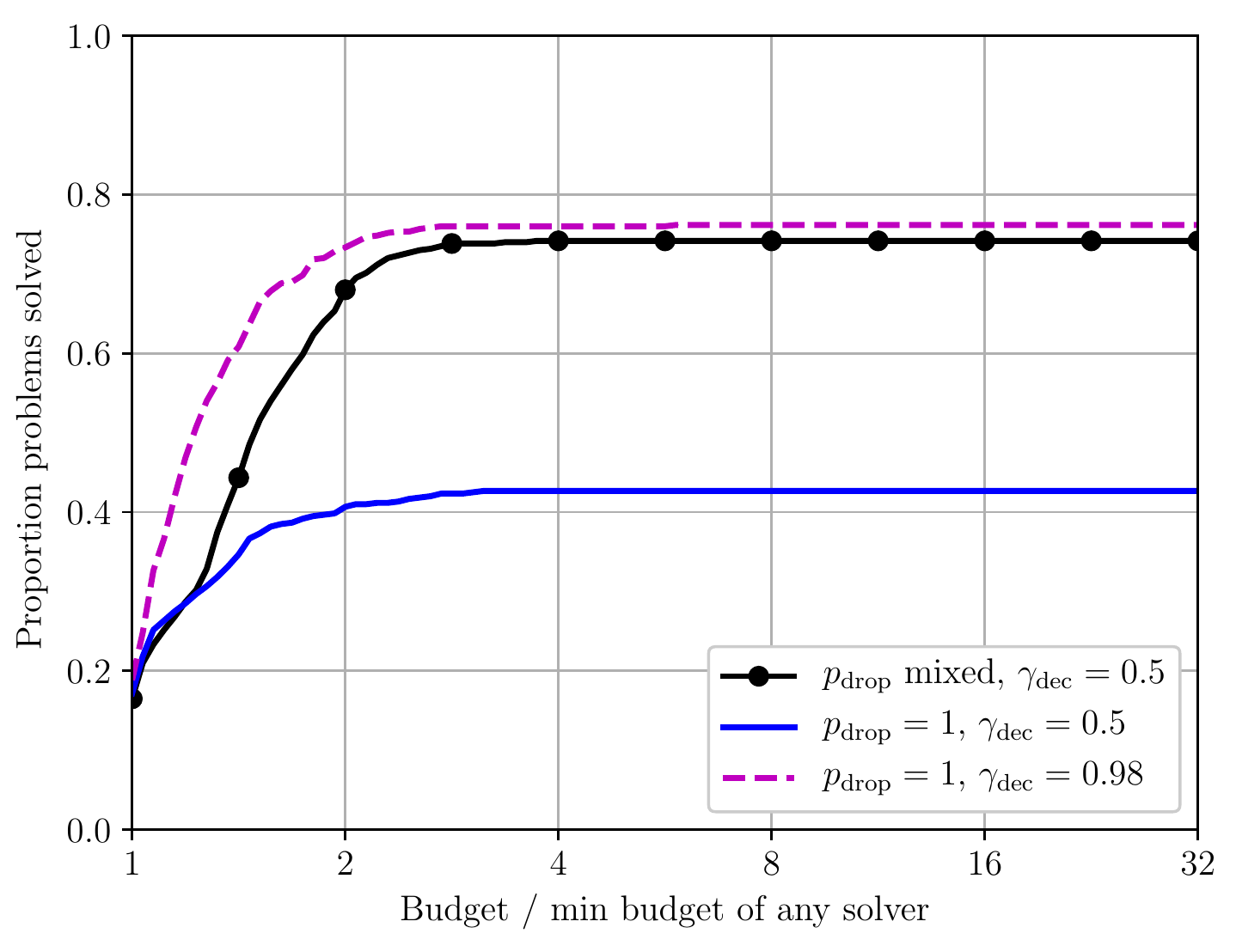}
		\caption{DFBGN with $p=n$}
	\end{subfigure}
	\caption{Performance profiles (in evaluations) comparing different choices of $p_{\rm drop}$ and $\gamma_{\rm dec}$ for DFBGN, at accuracy level $\tau=10^{-5}$. The choice `$p_{\rm drop}$ mixed' uses $p_{\rm drop}=1$ for successful iterations and $p_{\rm drop}=p/10$ for unsuccessful iterations. Results an average of 10 runs, each with a budget of $100(n+1)$ evaluations. The problem collection is (CR).  See \secref{sec_numerics_framework} for details on the testing framework.}
	\label{fig_noisy_comparison}
\end{figure}

\section{Numerical Results} \label{sec_numerics}
In this section we compare the performance of DFBGN (\algref{alg_block_dfols}) to that of DFO-LS.
We note that that DFO-LS has been shown to have state-of-the-art performance compared to other solvers in \cite{Cartis2018}.
As described in \secref{sec_dfbgn_full_algo}, the implementation of DFBGN is based on the decision to reduce the linear algebra cost of the algorithm at the expense of more objective evaluations per iteration.
However, we still maintain the goal of DFBGN achieving (close to) state-of-the-art performance when it is run as a `full space' method (i.e.~$p=n$).
Here, we will investigate this tradeoff in practice.

\subsection{Testing Framework} \label{sec_numerics_framework}
In our testing, we will compare a Python implementation of DFBGN (\algref{alg_block_dfols}) against DFO-LS version 1.0.2 (also implemented in Python).
The implementation of DFBGN is available on Github.\footnote{See \url{https://github.com/numericalalgorithmsgroup/dfbgn}. Results here use version 0.1.}
We will consider both the standard version of DFO-LS, and one where we use a reduced initialization cost of $n/100$ evaluations (c.f.~\remref{rem_dfols_growing}).
This will allow us to compare both the overall performance of DFBGN and its performance with small budgets (since DFBGN also has a reduced initialization cost of $p+1$ evaluations).
We compare these against DFBGN with the choices $p\in\{n/100, n/10, n/2, n\}$ and the adaptive choice of $p_{\rm drop}\in\{1,p/10\}$ (\secref{sec_scalability_practicality}).
All default settings are used for both solvers, and since both are randomized (DFO-LS uses random initial directions only, and DFBGN is randomized through \algref{alg_dfbgn_add_points}), we run 10 instances of each problem under all solver configurations.

\paragraph{Test Problems}
We will consider two collections of nonlinear least-squares test problems, both taken from the CUTEst collection \cite{Gould2015}.
The first, denoted (CR), is a collection of 60 medium-scale problems (with $25\leq n\leq 110$ and $n\leq m \leq 400$).
Full details of the (CR) collection may be found in \cite[Table 3]{Cartis2019a}.
The second, denoted (CR-large), is a collection of 28 large-scale problems (with $1000 \leq n \leq 5000$ and $n\leq m \leq 9998$).
This collection is a subset of problems from (CR), with their dimension increased substantially.
Full details of the (CR-large) collection are given in \appref{app_cr_large_problems}.
Note that the 12 hour runtime limit was only relevant for (CR-large) in all cases.

\paragraph{Measuring Solver Performance}
For every problem, we allow all solvers a budget of at most $100(n+1)$ objective evaluations (i.e.~evaluations of the full vector $\br(\bx)$).
This dimension-dependent choice may be understood as equivalent to 100 evaluations of $\br(\bx)$ and the Jacobian $J(\bx)$ via finite differencing.
However, given the importance of linear algebra cost to our comparisons, we allow each solver a maximum runtime of 12 hours for each instance of each problem.\footnote{Since all problems are implemented in Fortran via CUTEst, the cost of objective evaluations for this testing is minimal.}
For each solver $S$, each problem instance $P$, and accuracy level $\tau\in(0,1)$, we calculate
\begin{align}
	N(S,P,\tau) &\defeq \text{\# evaluations of $\br(\bx)$ required to find a point $\bx$ with} \nonumber \\
	&\qquad\qquad f(\bx) \leq f(\bx^*) + \tau (f(\bx_0)-f(\bx^*)), \label{eq_thresh_measure}
\end{align}
where $f(\bx^*)$ is an estimate of the minimum of $f$ as listed in \cite[Table 3]{Cartis2019a} for (CR) and \appref{app_cr_large_problems} for (CR-large).
If this objective decrease is not achieved by a solver before its budget or runtime limit is hit, we set $N(S,P,\tau)=\infty$.
We then compare solver performances on a problem collection $\mathcal{P}$ by plotting either data profiles \cite{More2009}
\be d_{S,\tau} (\alpha) \defeq \frac{1}{|\mathcal{P}|}\left|\{P\in\mathcal{P} : N(S,P,\tau) \leq \alpha (n_P+1)\}\right|, \ee
where $n_P$ is the dimension of problem instance $P$ and $\alpha\in[0,100]$ is an evaluation budget (in ``gradients'', or multiples of $n+1$), or performance profiles \cite{Dolan2002}
\be \pi_{S,\tau} (\alpha) \defeq \frac{1}{|\mathcal{P}|}\left|\{P\in\mathcal{P} : N(S,P,\tau) \leq \alpha N_{\min}(P,\tau)\}\right|, \ee
where $N_{\min}(P,\tau)$ is the minimum value of $N(S,P,\tau)$ for any solver $S$, and $\alpha\geq 1$ is a performance ratio.
In some instances, we will plot profiles based on runtime rather than objective evaluations.
For this, we simply replace ``number of evaluations of $\br(\bx)$'' with ``runtime'' in \eqref{eq_thresh_measure}.

When we plot the objective reduction achieved by a given solver, we normalize the objective value to be in $[0,1]$ by plotting
\be \frac{f(\bx)-f(\bx^*)}{f(\bx_0)-f(\bx^*)}, \ee
which corresponds to the best $\tau$ achieved in \eqref{eq_thresh_measure} after a given number of evaluations (again measured in ``gradients'') or runtime.

\subsection{Results Based on Evaluations} \label{sec_dfbgn_results_evals}
We begin our comparisons by considering the performance of DFO-LS and DFBGN when measured in terms of evaluations.

\begin{figure}[t]
	\centering
	\begin{subfigure}[b]{0.48\textwidth}
		\includegraphics[width=\textwidth]{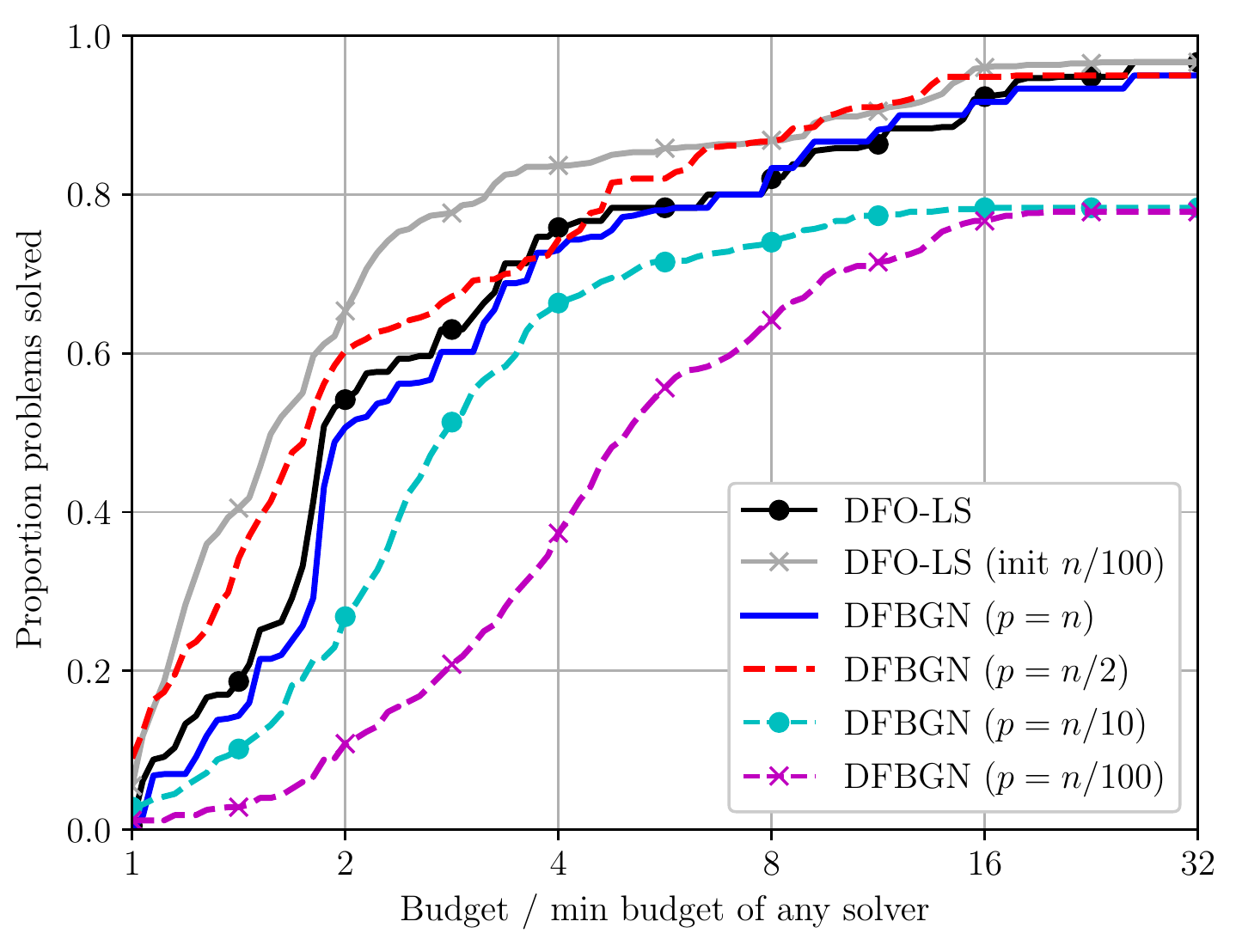}
		\caption{$\tau=0.5$}
	\end{subfigure}
	~
	\begin{subfigure}[b]{0.48\textwidth}
		\includegraphics[width=\textwidth]{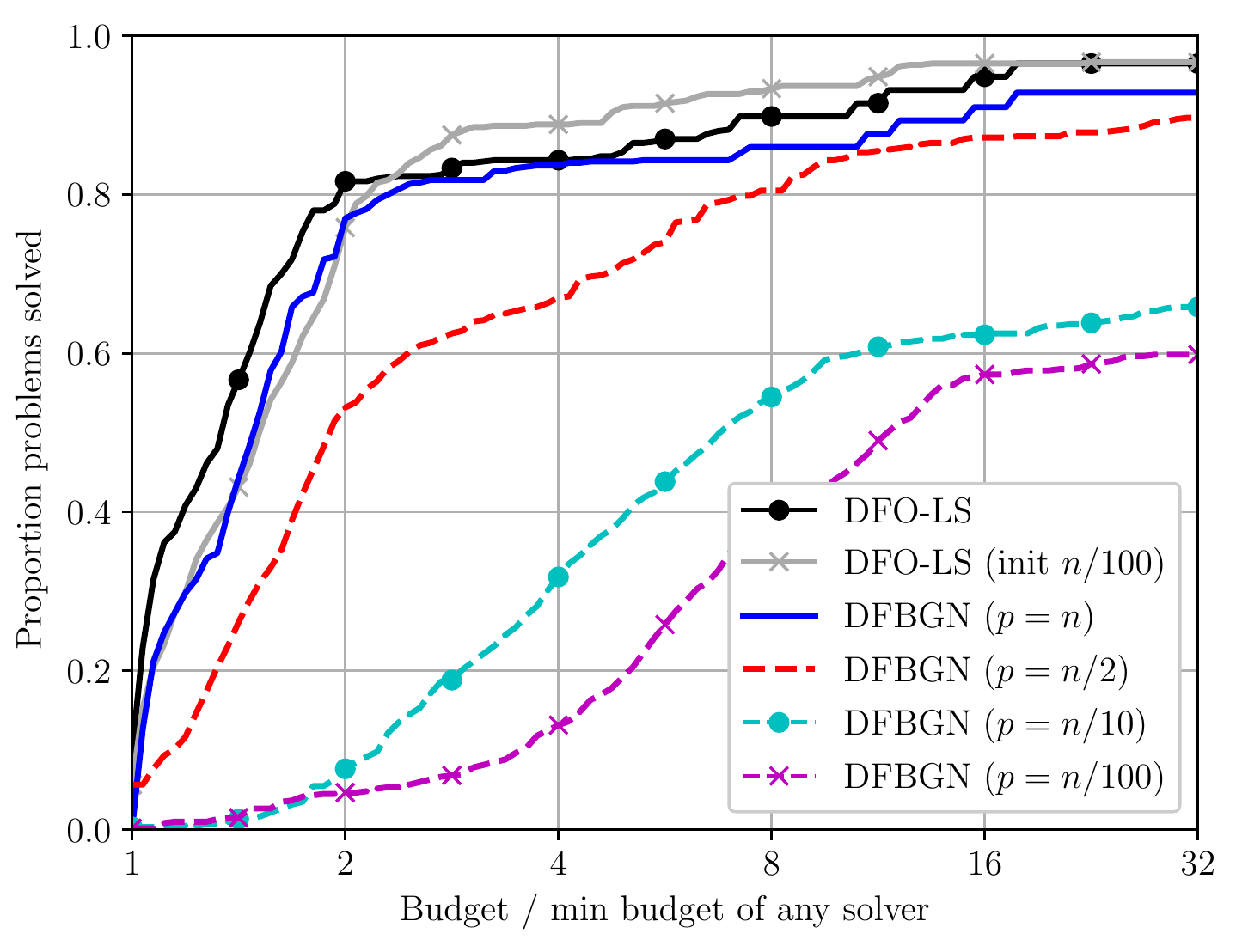}
		\caption{$\tau=10^{-1}$}
	\end{subfigure}
	\\
	\begin{subfigure}[b]{0.48\textwidth}
		\includegraphics[width=\textwidth]{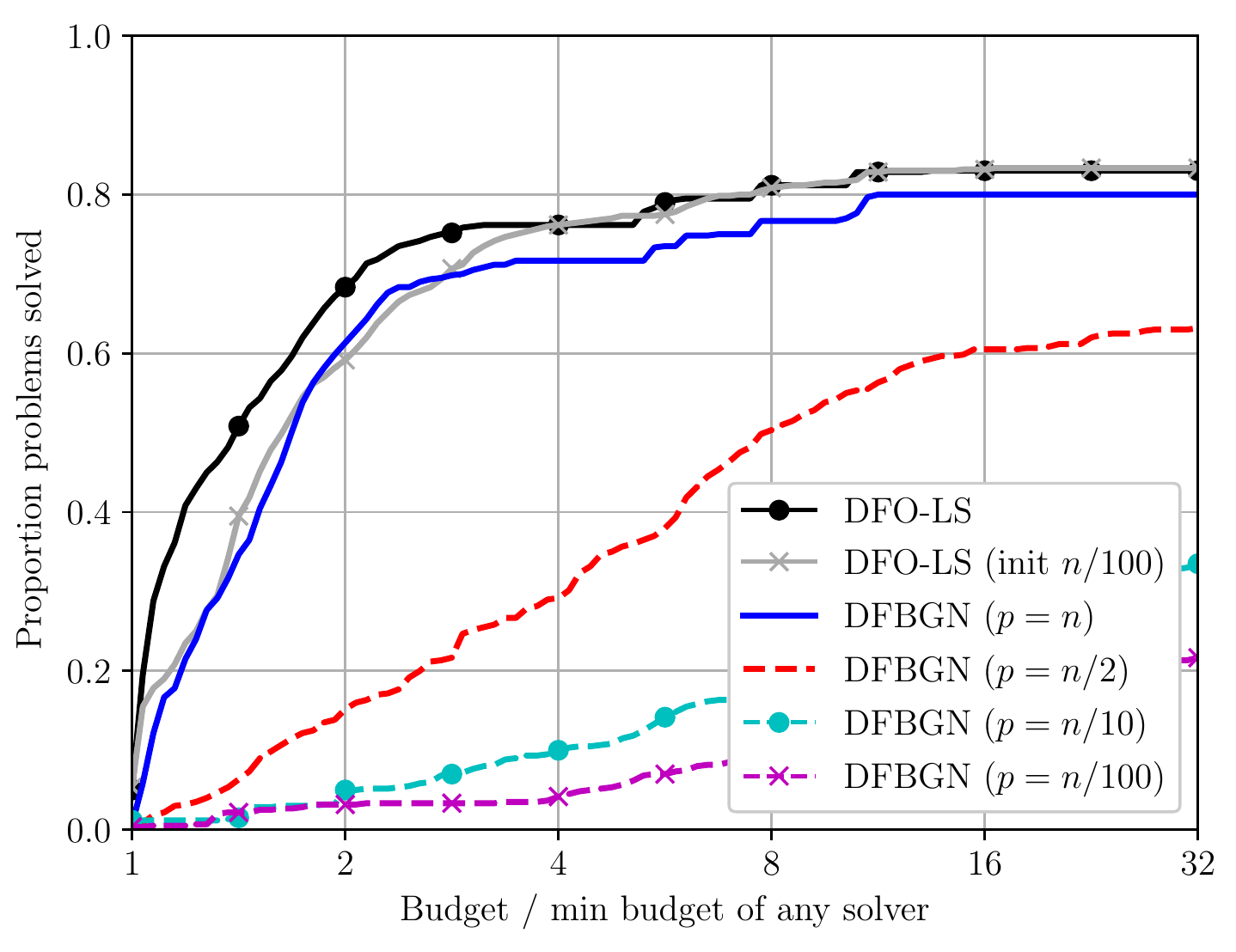}
		\caption{$\tau=10^{-3}$}
	\end{subfigure}
	~
	\begin{subfigure}[b]{0.48\textwidth}
		\includegraphics[width=\textwidth]{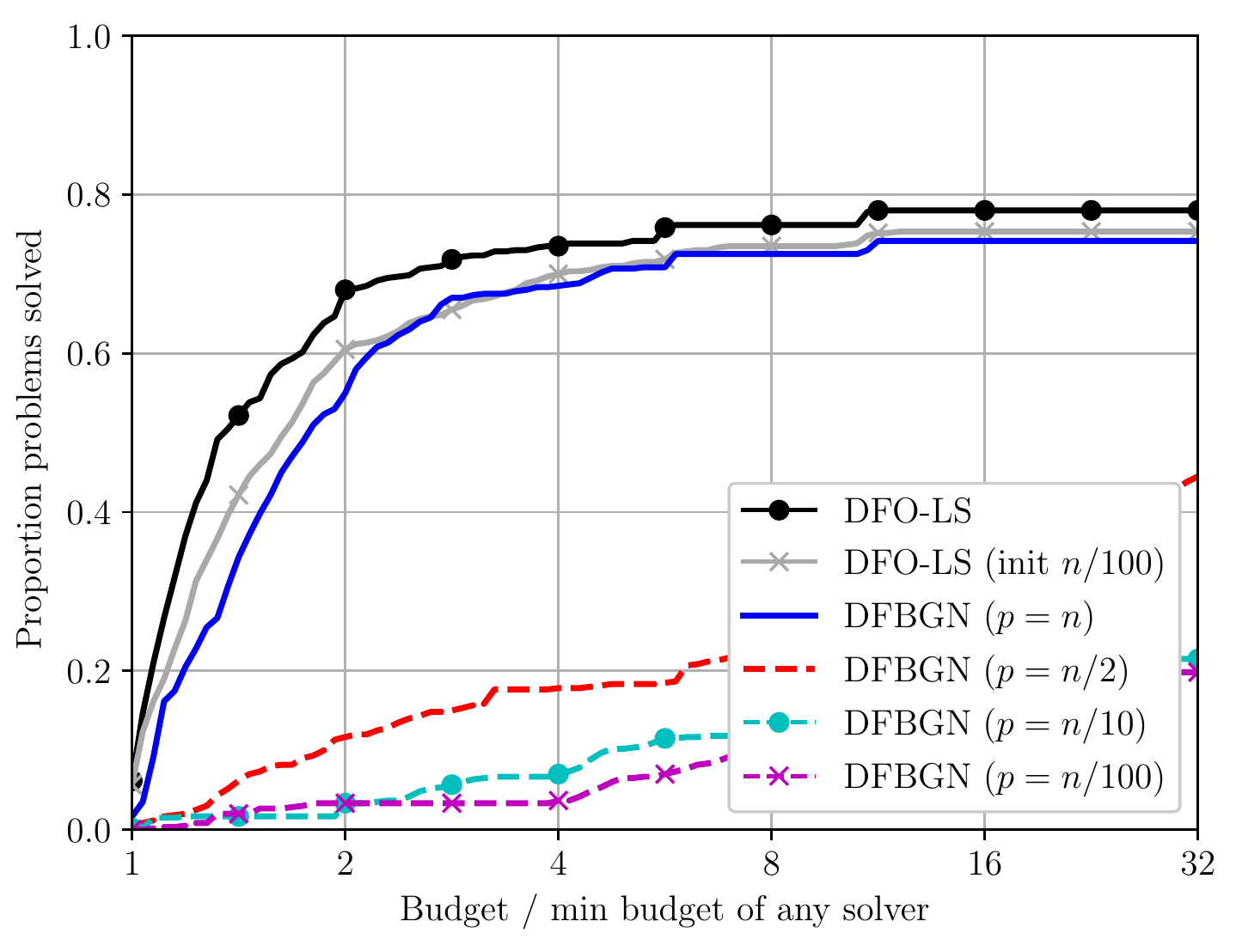}
		\caption{$\tau=10^{-5}$}
	\end{subfigure}
	\caption{Performance profiles (in evaluations) comparing DFO-LS (with and without reduced initialization cost) with DFBGN (various $p$ choices) for different accuracy levels. Results are an average of 10 runs for each problem, with a budget of $100(n+1)$ evaluations and a 12 hour runtime limit per instance. The problem collection is (CR).}
	\label{fig_dfbgn_cutest100}
\end{figure}

\paragraph{Medium-Scale Problems (CR)}
First, in \figref{fig_dfbgn_cutest100}, we show the results for different accuracy levels for the (CR) problem collection (with $n\approx 100$).
For the lowest accuracy level $\tau=0.5$, DFO-LS with reduced initialization cost is the best-performing solver, followed by DFBGN with $p=n/2$.
These correspond to methods with lower initialization costs than DFO-LS and DFBGN with $p=n$, so this is likely a large driver behind their performance. 
DFBGN with full space size $p=n$ performs similarly to DFO-LS, and DFBGN with $p=n/10$ and $p=n/100$ perform worst (as they are optimizing in a very small subspace at each iteration).

However, as we look at higher accuracy levels, we see that DFO-LS (with and without reduced initialization cost) performs best, and the DFBGN methods perform worse.
The performance gap is more noticeable for small values of $p$.
As expected, this means that DFBGN requires more evaluations to achieve these levels of accuracy, and benefits from being allowed to use a larger $p$.
Notably, DFBGN with $p=n$ has only a slight performance loss compared to DFO-LS, even though it uses $p/10$ evaluations on unsuccessful iterations (rather than 1--2 for DFO-LS); this indicates that our choice of $p_{\rm drop}$ provides a suitable compromise between solver robustness and evaluation efficiency.

\begin{figure}[t]
	\centering
	\begin{subfigure}[b]{0.48\textwidth}
		\includegraphics[width=\textwidth]{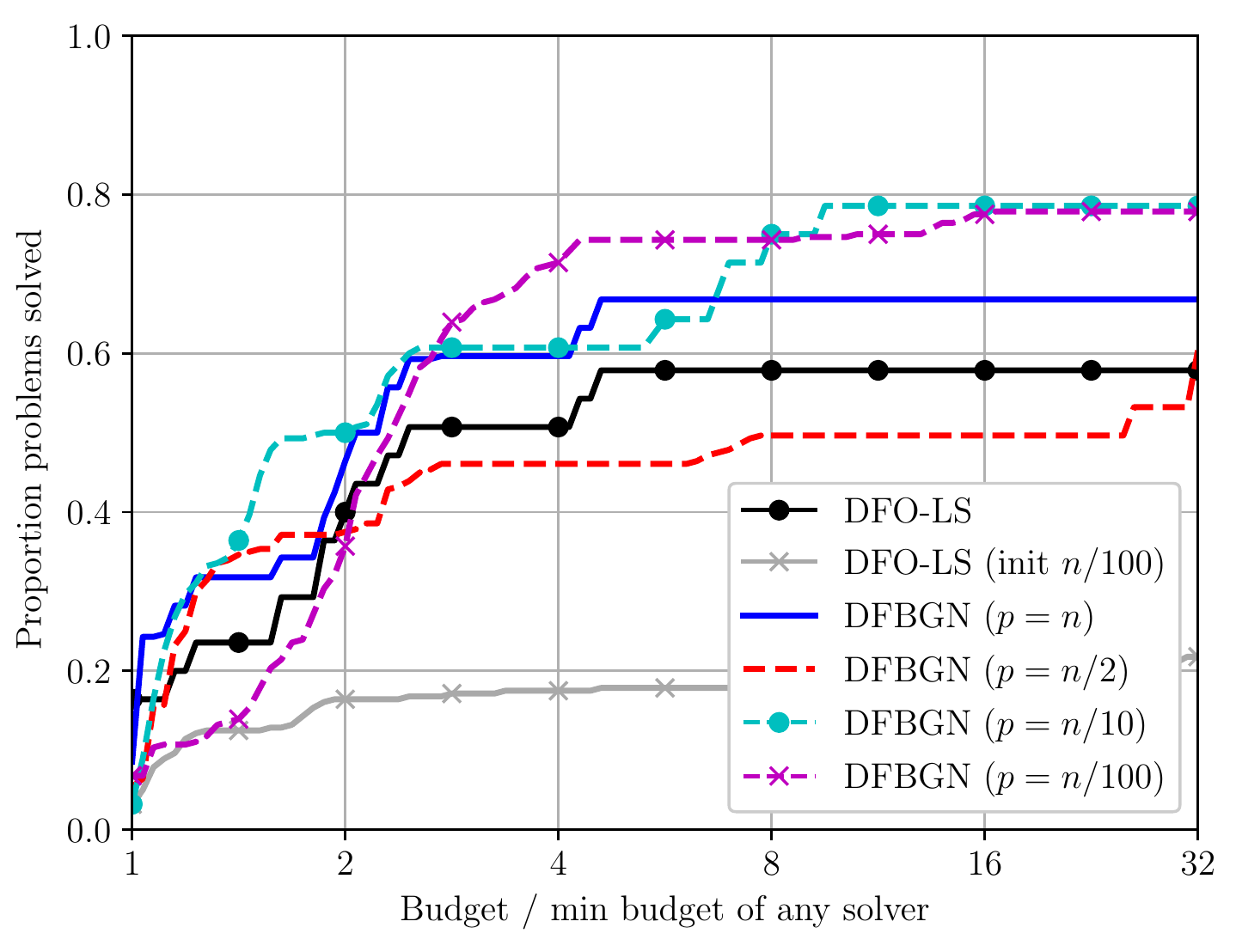}
		\caption{$\tau=0.5$}
	\end{subfigure}
	~
	\begin{subfigure}[b]{0.48\textwidth}
		\includegraphics[width=\textwidth]{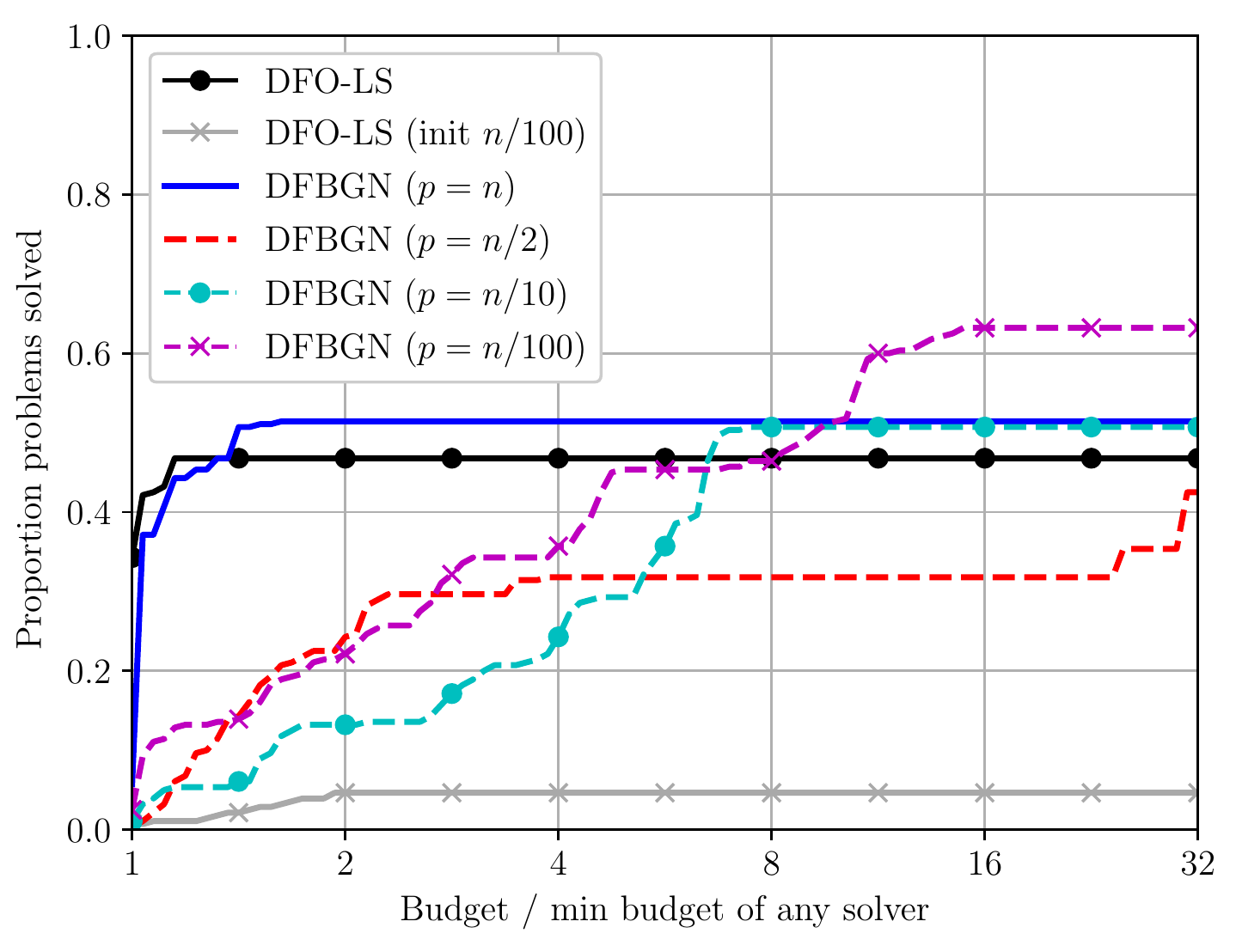}
		\caption{$\tau=10^{-1}$}
	\end{subfigure}
	\\
	\begin{subfigure}[b]{0.48\textwidth}
		\includegraphics[width=\textwidth]{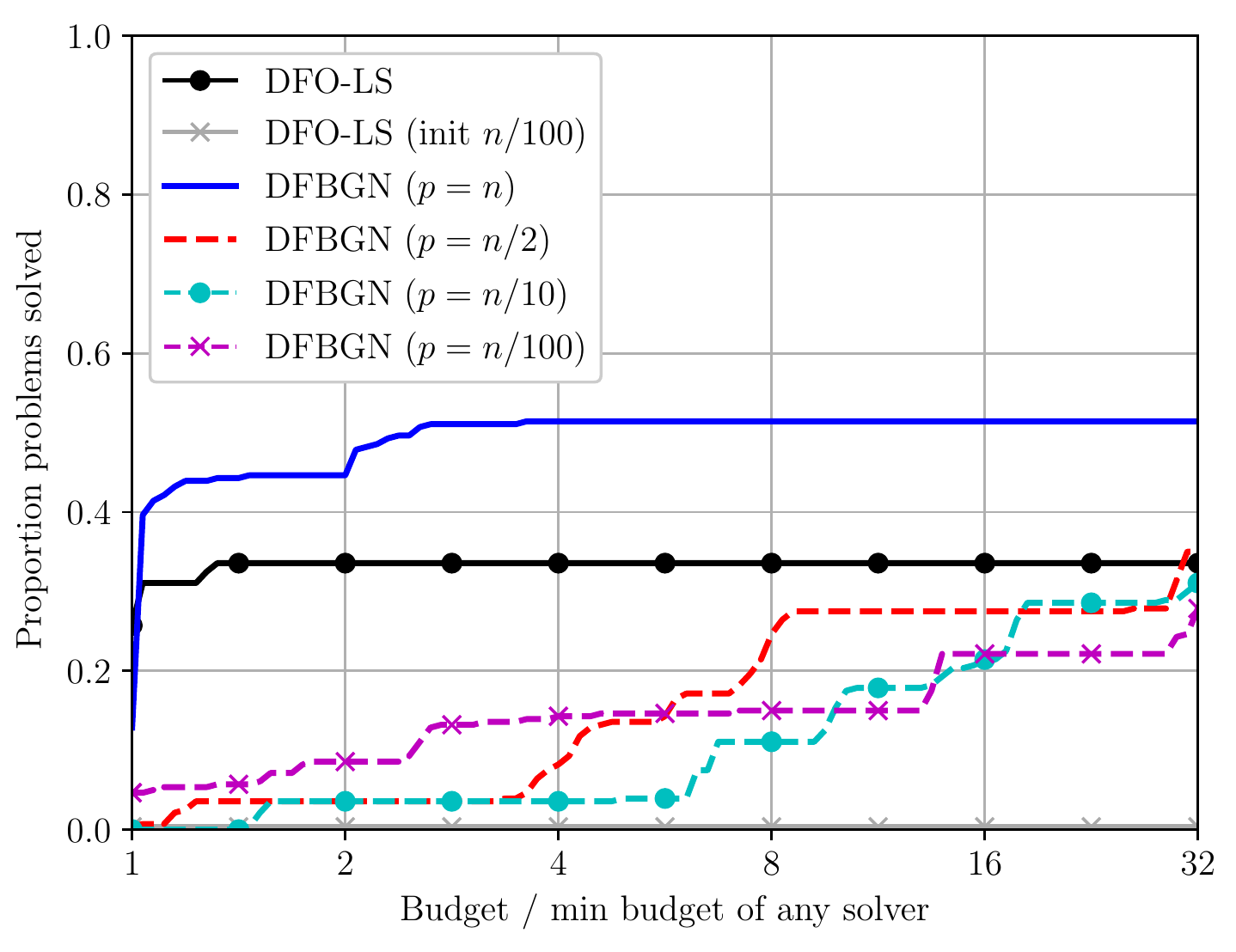}
		\caption{$\tau=10^{-3}$}
	\end{subfigure}
	~
	\begin{subfigure}[b]{0.48\textwidth}
		\includegraphics[width=\textwidth]{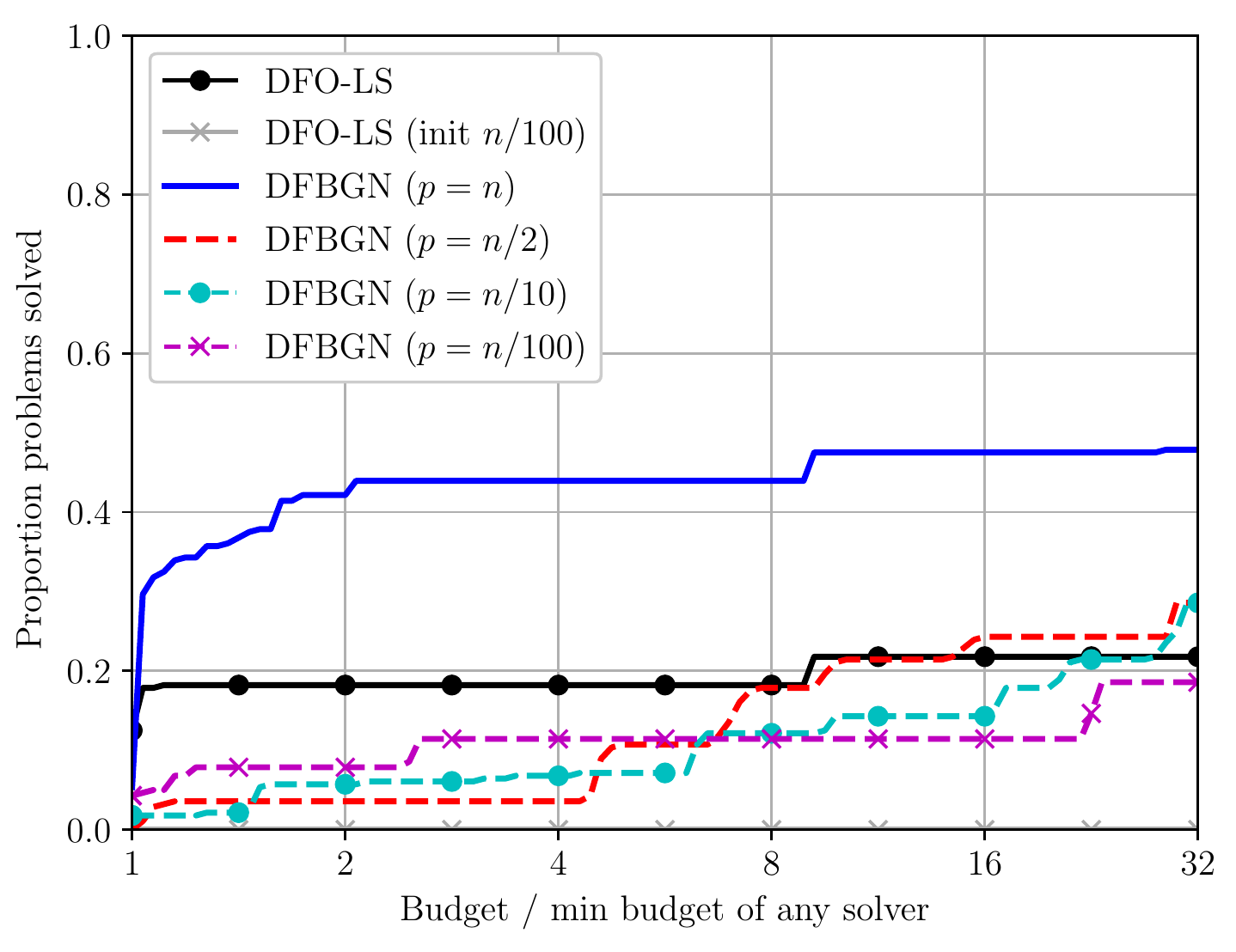}
		\caption{$\tau=10^{-5}$}
	\end{subfigure}
	\caption{Performance profiles (in evaluations) comparing DFO-LS (with and without reduced initialization cost) with DFBGN (various $p$ choices) for different accuracy levels. Results are an average of 10 runs for each problem, with a budget of $100(n+1)$ evaluations and a 12 hour runtime limit per instance. The problem collection is (CR-large).}
	\label{fig_dfbgn_cutest1000}
\end{figure}

\paragraph{Large-Scale Problems (CR-large)}
Next, in \figref{fig_dfbgn_cutest1000}, we show the same plots but for the (CR-large) problem collection, with $n\approx 1000$.
Compared to \figref{fig_dfbgn_cutest100}, the situation is quite different.

At the lowest accuracy level, $\tau=0.5$, DFBGN with small subspaces ($p=n/10$ and $p=n/100$) gives the best-performing solvers, followed by the full-space solvers (DFO-LS and DFBGN with $p=n$).
For higher accuracy levels, the performance of DFBGN with small $p$ deteriorates compared with the full-space methods.
DFBGN with $p=n/2$ is the worst-performing DFBGN variant at low accuracy levels, and performs similar to DFBGN with small $p$ at high accuracy levels.
DFO-LS with reduced initialization cost is the worst-performing solver for this dataset.

Unlike the medium-scale results above, we no longer have a clear trend in the performance of DFBGN as we vary $p$.
Instead, we have a combination of two factors coming into play, which have opposite impacts on the performance of DFBGN as we vary $p$.
On one hand, we have the number of evaluations required for DFBGN (with a given $p$) to reach the desired accuracy level.
On the other hand, we have the number of iterations that DFBGN can perform before reaching the 12 hour runtime limit.

DFBGN with small $p$ requires more evaluations to reach a given level of accuracy (as seen with the medium-scale results), but can perform many evaluations before timing out due to its low per-iteration linear algebra cost.
This is reflected in it solving many problems to low accuracy, but few problems to high accuracy.
By contrast, DFBGN with $p=n$ is allowed to perform fewer iterations before timing out (and hence see fewer evaluations), but requires many fewer evaluations to solve problems, particularly for high accuracy.
This manifests in its good performance for low and high accuracy levels.
The middle ground, DFBGN with $p=n/2$, has its performance negatively impacted by both issues: it requires many fewer evaluations to solve problems than $p=n$ (especially for high accuracy), but also has a relatively high per-iteration linear algebra cost and times out compared to small $p$.

Both variants of DFO-LS show worse performance here than for the medium-scale problems.
This is because, as suggested by the analysis in \tabref{tab_linalg_comparison}, they are both affected by the runtime limit.
DFO-LS with reduced initialization cost is particularly affected, because of the high cost of the SVD (of the full $m\times n$ Jacobian) at each iteration for these problems.
We note that this cost is only noticeable for these large-scale problems, and this variant of DFO-LS is still useful for small- and medium-scale problems, as discussed in \cite{Cartis2018}.

\begin{table}[t]
	\centering
	\footnotesize{
	\begin{tabular}{lc}
		\hline\noalign{\smallskip}
		Solver & \% timeout \\ \noalign{\smallskip}\hline\noalign{\smallskip}
		DFO-LS & 92.5\% \\
		DFO-LS (init $n/100$) & 97.9\% \\ \noalign{\smallskip}\hline\noalign{\smallskip}
		DFBGN ($p=n/100$) & 34.6\% \\
		DFBGN ($p=n/10$) & 73.9\% \\
		DFBGN ($p=n/2$) & 81.8\% \\
		DFBGN ($p=n$) & 66.4\% \\
		\noalign{\smallskip}\hline\noalign{\smallskip}
	\end{tabular}}
	\caption{Proportion of problem instances from (CR-large) for which each solver terminated on the maximum 12 hour runtime.}
	\label{tab_timeout_comparison}
\end{table}

We can verify the impact of the timeout on DFO-LS and DFBGN by considering the proportion of problem instances for (CR-large) for which the solver terminated because of the timeout.
These results are presented in \tabref{tab_timeout_comparison}.
DFO-LS reaches the 12 hour maximum much more frequently than DFBGN: over 90\% rather than 35\% for DFBGN with $p=n/100$ or 66\% for DFBGN with $p=n$ (see \remref{rem_dfbgn_n_v_dfols} below). 
For DFBGN with different values of $p$, we see the same behaviour as in \figref{fig_dfbgn_cutest1000}.
That is, DFBGN with small $p$ times out the least frequently, as its low per-iteration runtime means it performs enough iterations to terminate naturally.
For DFBGN with $p=n$, we time out more frequently (due to the high per-iteration runtime), but not as often as with $p=n/2$, as the its superior budget performance for high accuracy levels means it fully solves more problems, even with comparatively fewer iterations.
We note that \tabref{tab_timeout_comparison} does not measure what accuracy level was achieved before the timeout, which is better captured in the performance profiles \figref{fig_dfbgn_cutest1000}.

\begin{remark} \label{rem_dfbgn_n_v_dfols}
	DFBGN with $p=n$ has a similar per-iteration linear algebra cost to DFO-LS.
	Hence it can perform a similar number of iterations before reaching the runtime limit.
	However, DFBGN performs more objective evaluations per iteration, because of the choice of $p_{\rm drop}$.
	Since DFBGN with $p=n$ has a similar performance to DFO-LS when measured on budget (as seen in \figref{fig_dfbgn_cutest100}), this means that it has a superior performance when measured by runtime.
	Additionally, if multiple objective evaluations can be run in parallel, then DFBGN would also be able to benefit from this, unlike DFO-LS.
\end{remark}

\begin{remark}
	For completeness, in \appref{app_dfbgn_extra_numerics} we compare DFBGN with DFO-LS on the low-dimensional collection of test problems from Mor\'e and Wild \cite{More2009}.
	We do not include this discussion here as these problems are low-dimensional, which is not the main use case for DFBGN.
\end{remark}

\begin{figure}[t]
	\centering
	\begin{subfigure}[b]{0.48\textwidth}
		\includegraphics[width=\textwidth]{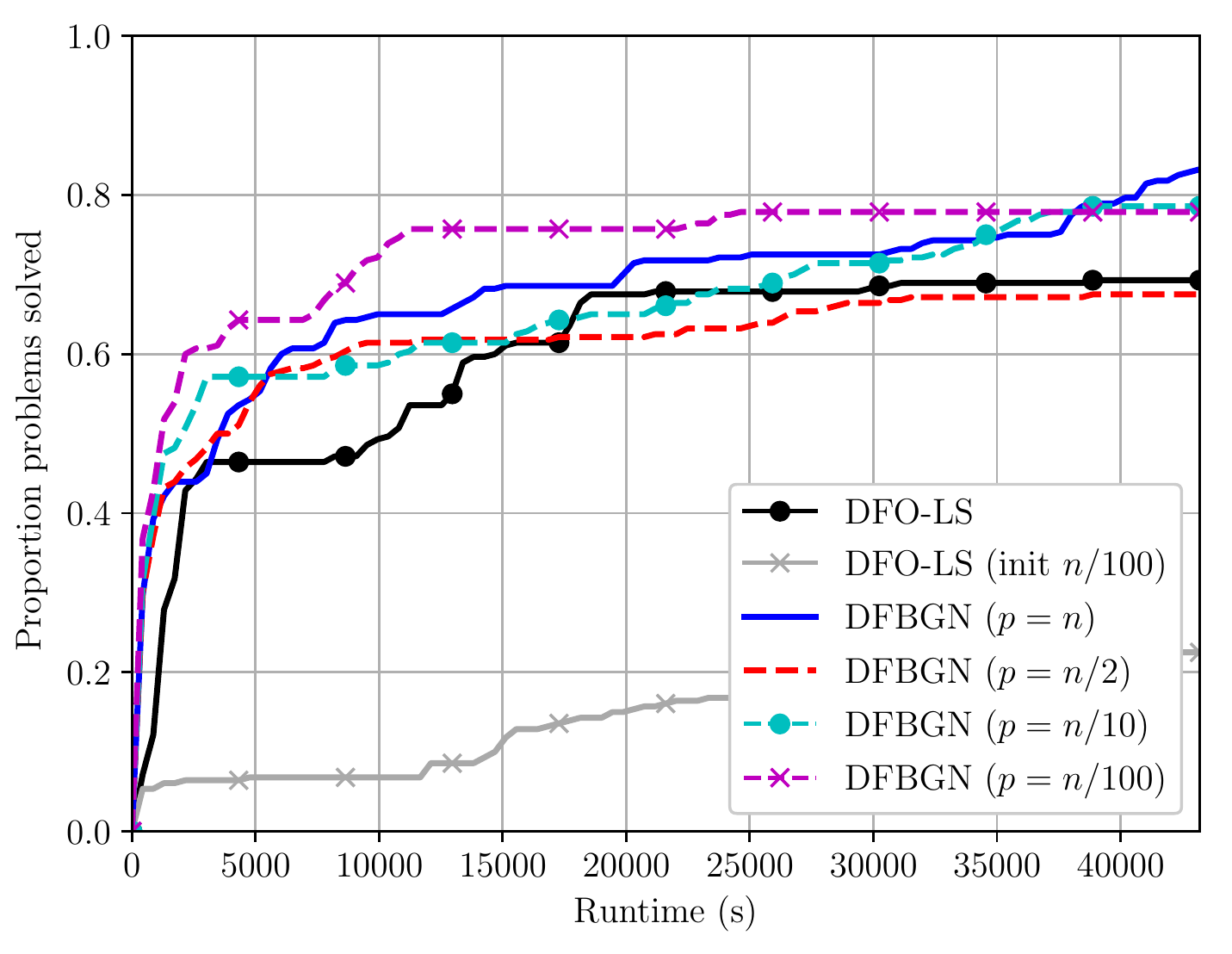}
		\caption{$\tau=0.5$}
	\end{subfigure}
	~
	\begin{subfigure}[b]{0.48\textwidth}
		\includegraphics[width=\textwidth]{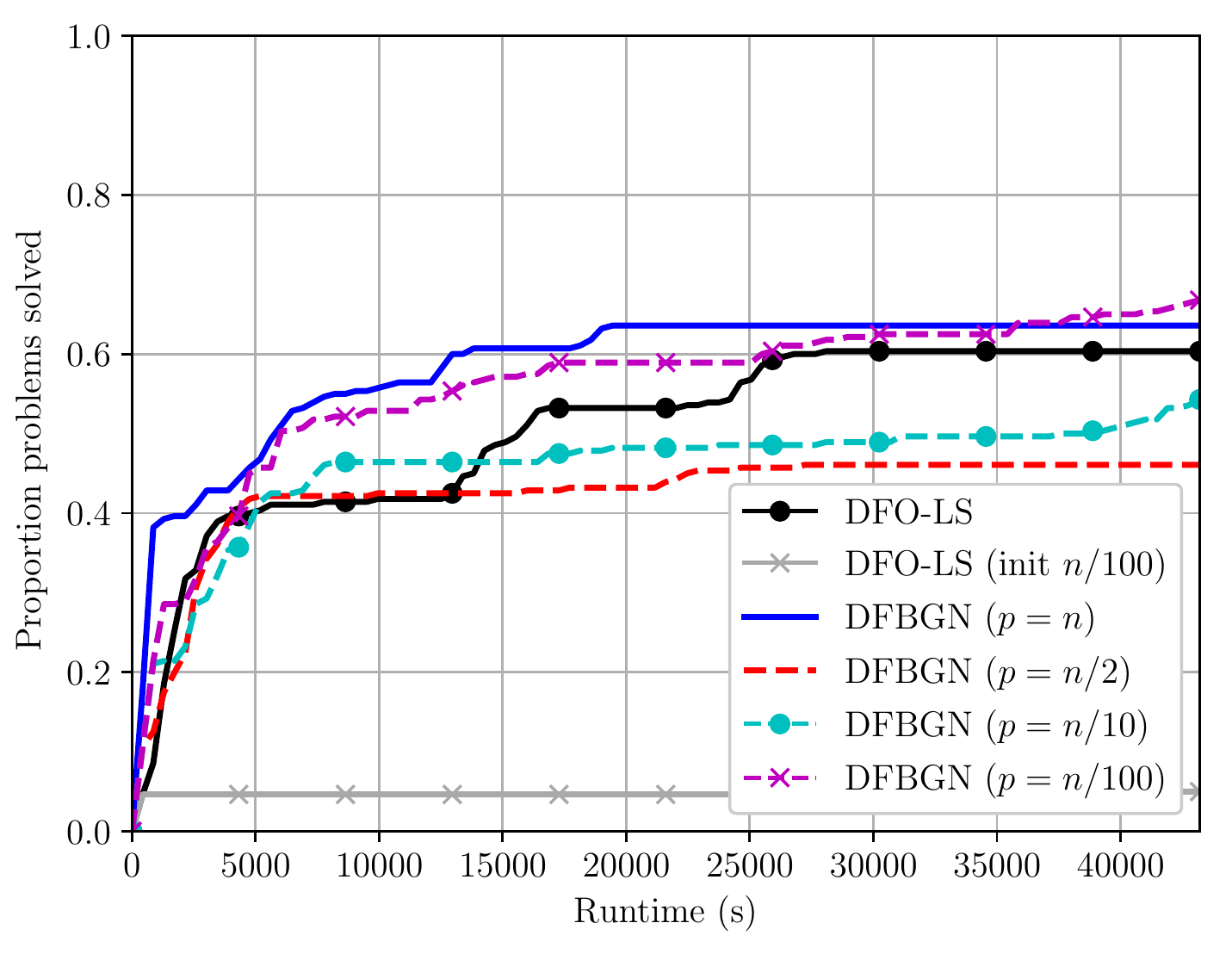}
		\caption{$\tau=10^{-1}$}
	\end{subfigure}
	\\
	\begin{subfigure}[b]{0.48\textwidth}
		\includegraphics[width=\textwidth]{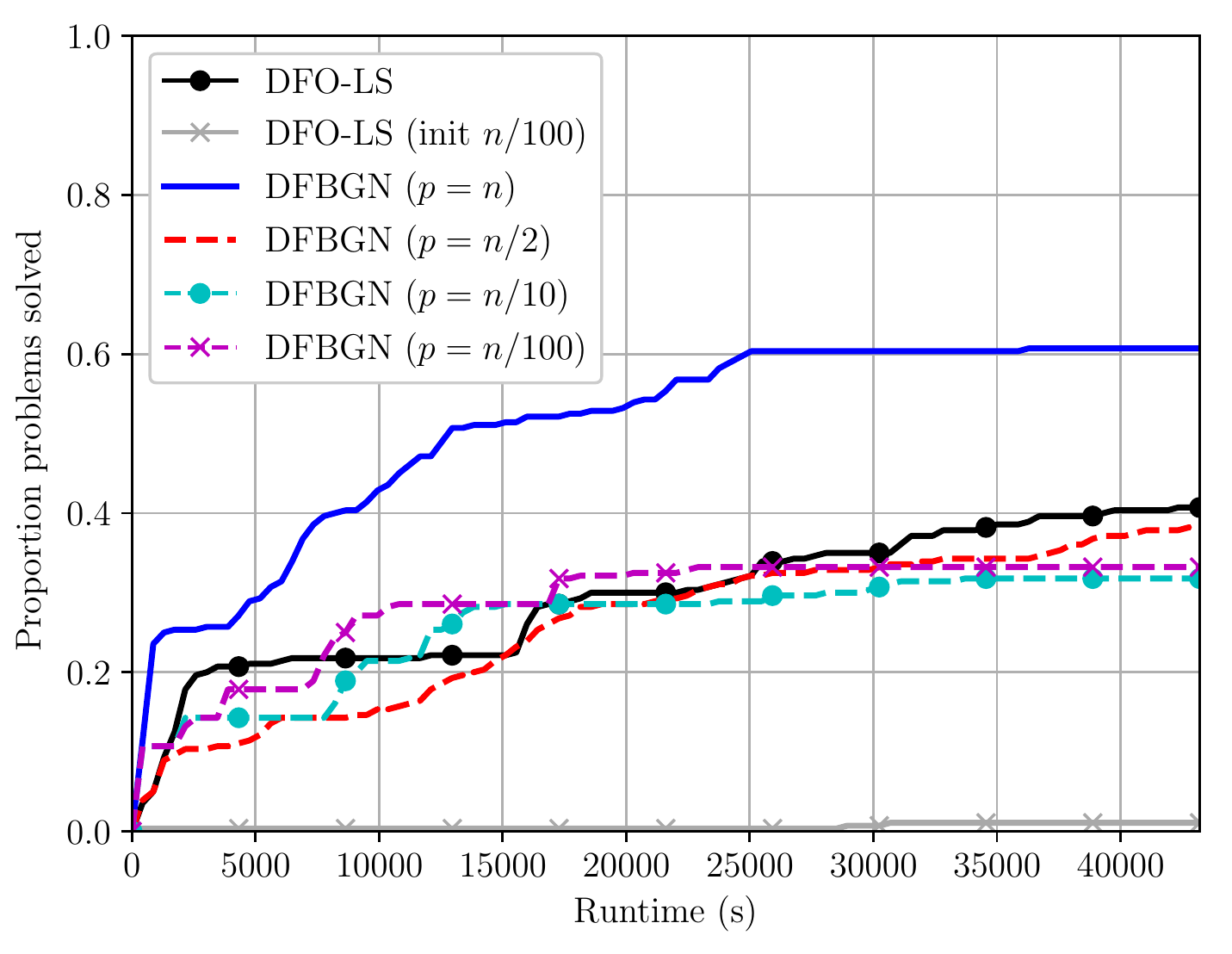}
		\caption{$\tau=10^{-3}$}
	\end{subfigure}
	~
	\begin{subfigure}[b]{0.48\textwidth}
		\includegraphics[width=\textwidth]{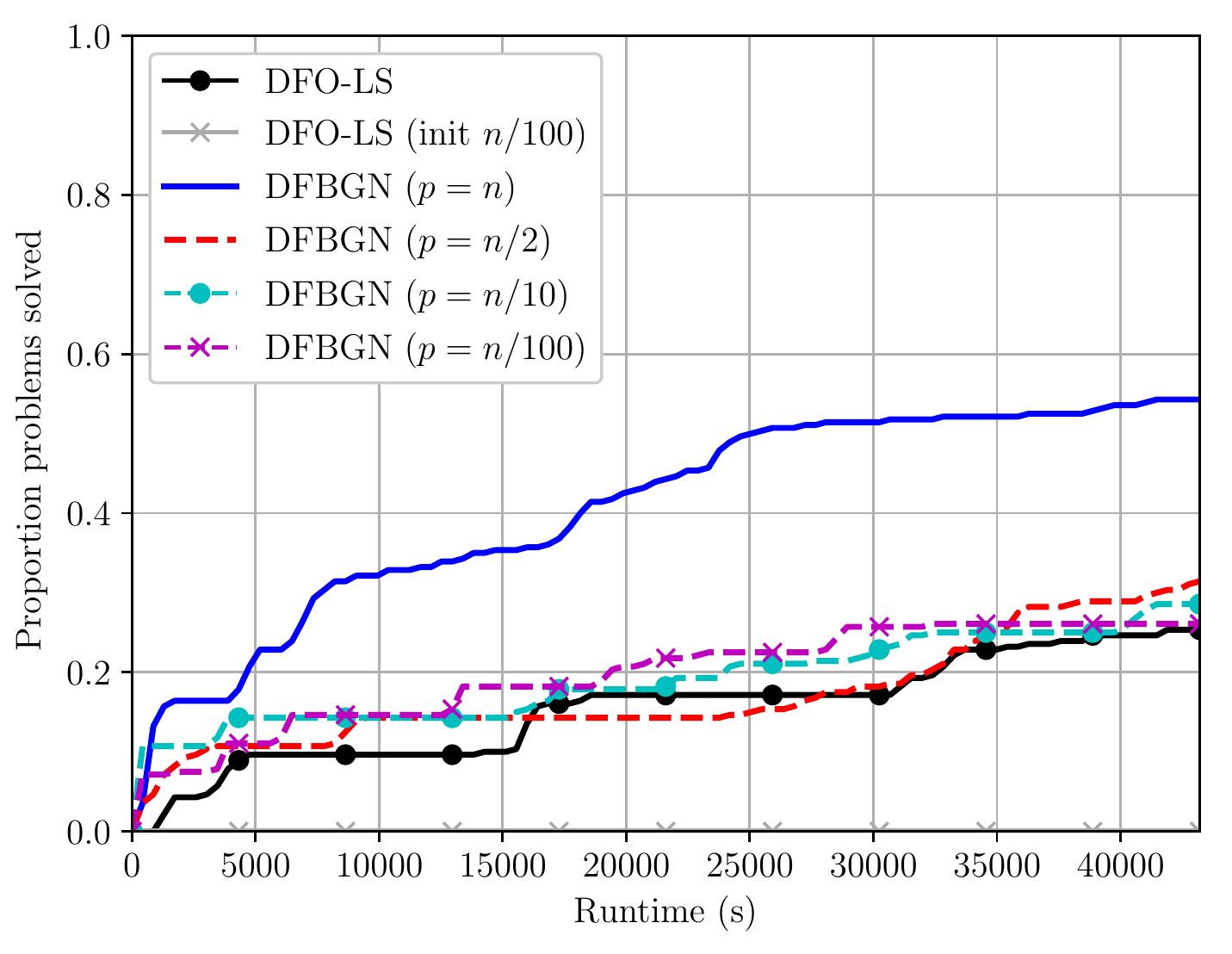}
		\caption{$\tau=10^{-5}$}
	\end{subfigure}
	\caption{Data profiles comparing the runtime of DFO-LS (with and without reduced initialization cost) with DFBGN (various $p$ choices) for different accuracy levels. Results are an average of 10 runs for each problem, with a budget of $100(n+1)$ evaluations and a 12 hour runtime limit per instance. The problem collection is (CR-large).}
	\label{fig_dfbgn_cutest1000_runtime}
\end{figure}

\subsection{Results Based on Runtime} \label{sec_dfbgn_results_runtime}
\enlargethispage{1em}
We have seen above that DFBGN performs well compared to DFO-LS on the (CR-large) problem collection, as the 12 hour timeout causes DFO-LS to terminate after relatively few objective evaluations.
In \figref{fig_dfbgn_cutest1000_runtime}, we show the same comparison for (CR-large) as in \figref{fig_dfbgn_cutest1000}, but showing data profiles of problems solved versus runtime (rather than evaluations).
Here, all DFBGN variants perform similar to or better than DFO-LS for low accuracy levels, since DFBGN has a lower per-iteration runtime than DFO-LS, and this is the regime where DFBGN performs best (on budget).
For high accuracy levels, DFBGN with $p=n$ is the best-performing solver, as it uses large enough subspaces to solve many problems to high accuracy.
By contrast, both DFBGN with small $p$ and DFO-LS perform similarly at high accuracy levels---the impact of the timeout on DFO-LS roughly matches the reduced robustness of DFBGN with small $p$ at these accuracy levels.
Again, as we observed above, DFO-LS with reduced initialization cost is the worst-performing solver, due to the high cost of the SVD at each iteration. 

\begin{figure}[t]
	\centering
	\begin{subfigure}[b]{0.48\textwidth}
		\includegraphics[width=\textwidth]{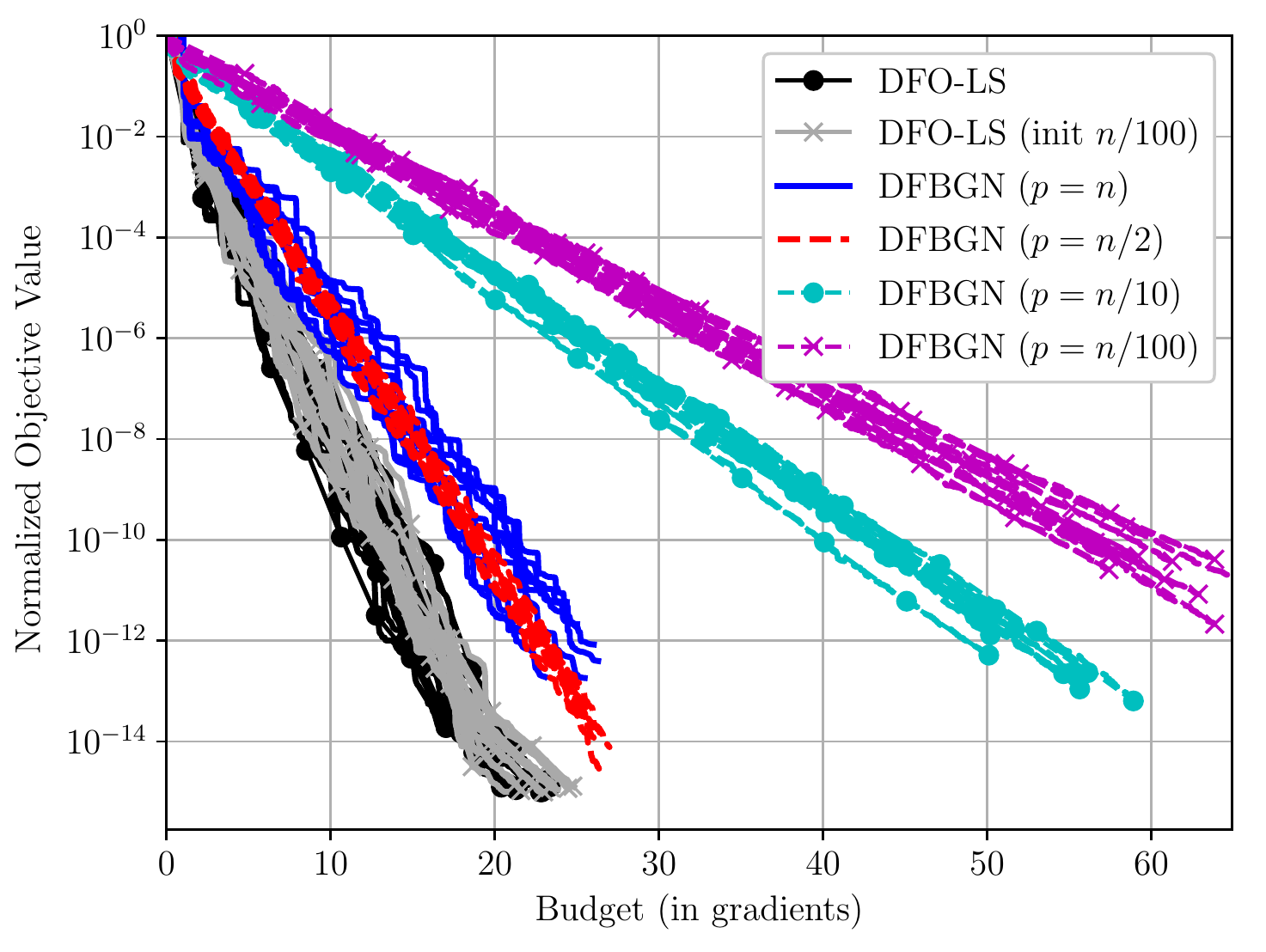}
		\caption{$n=100$, objective vs.~budget}
	\end{subfigure}
	~
	\begin{subfigure}[b]{0.48\textwidth}
		\includegraphics[width=\textwidth]{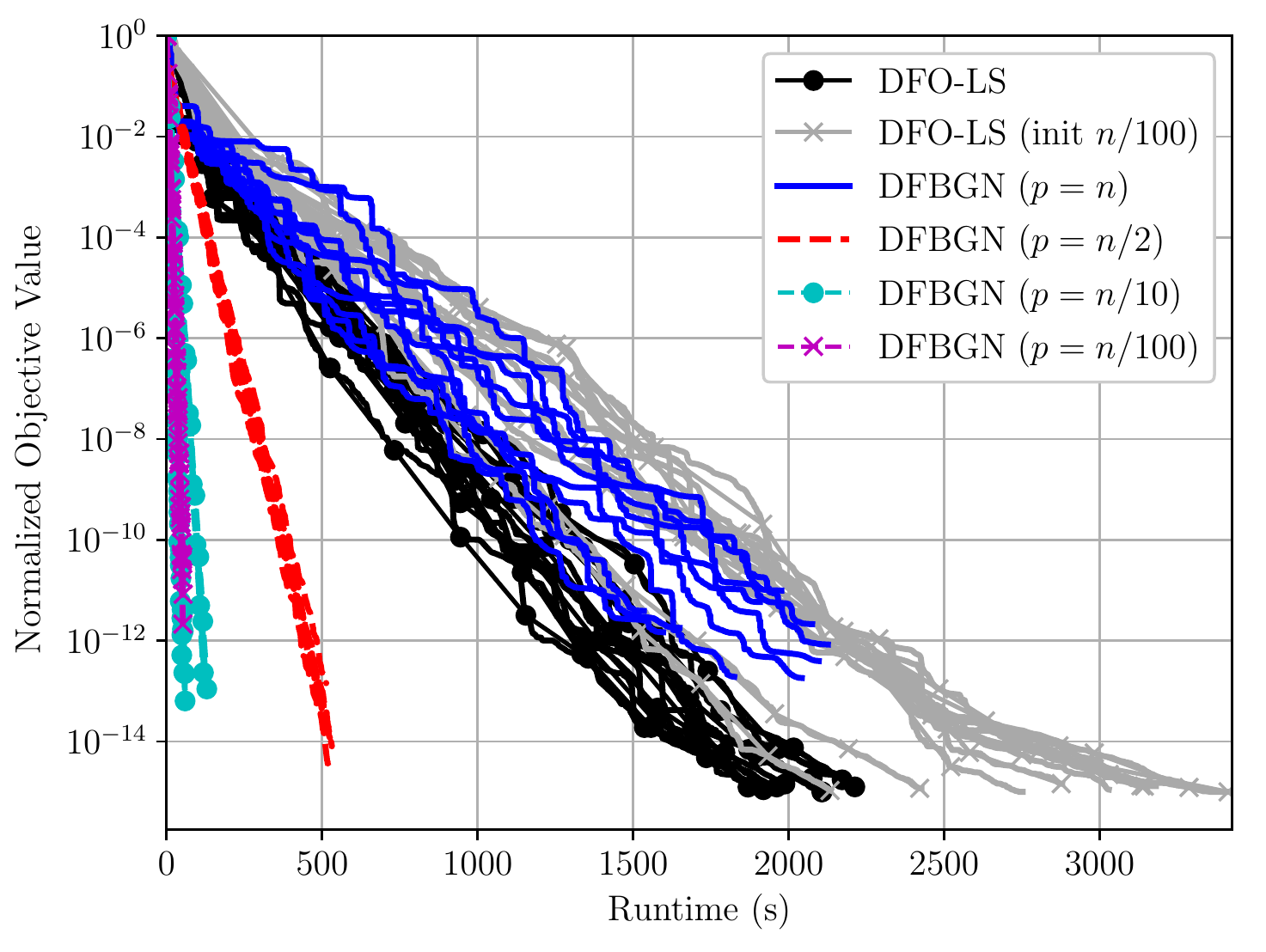}
		\caption{$n=100$, objective vs.~runtime}
	\end{subfigure}
	\\
	\begin{subfigure}[b]{0.48\textwidth}
		\includegraphics[width=\textwidth]{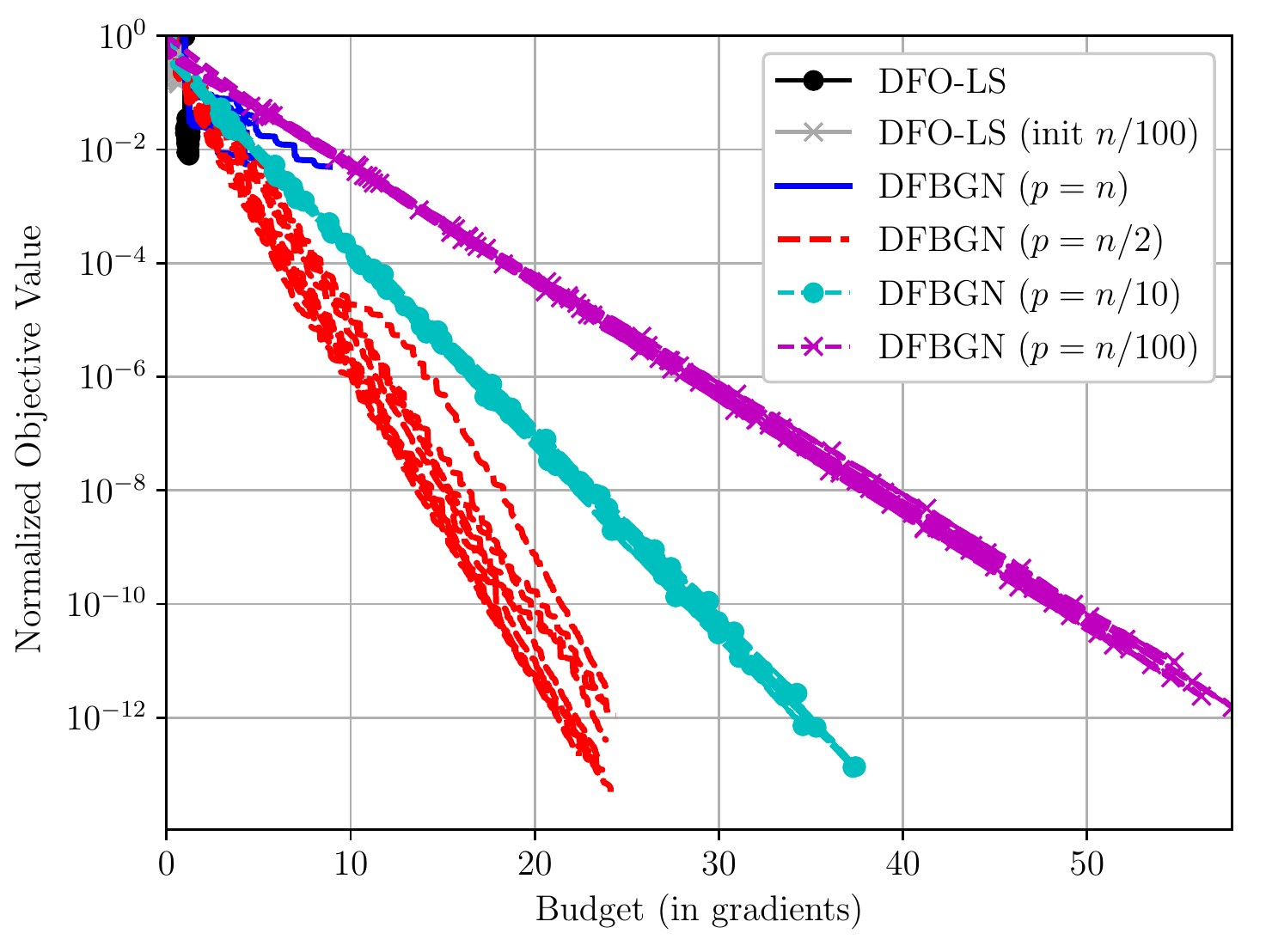}
		\caption{$n=1000$, objective vs.~budget}
	\end{subfigure}
	~
	\begin{subfigure}[b]{0.48\textwidth}
		\includegraphics[width=\textwidth]{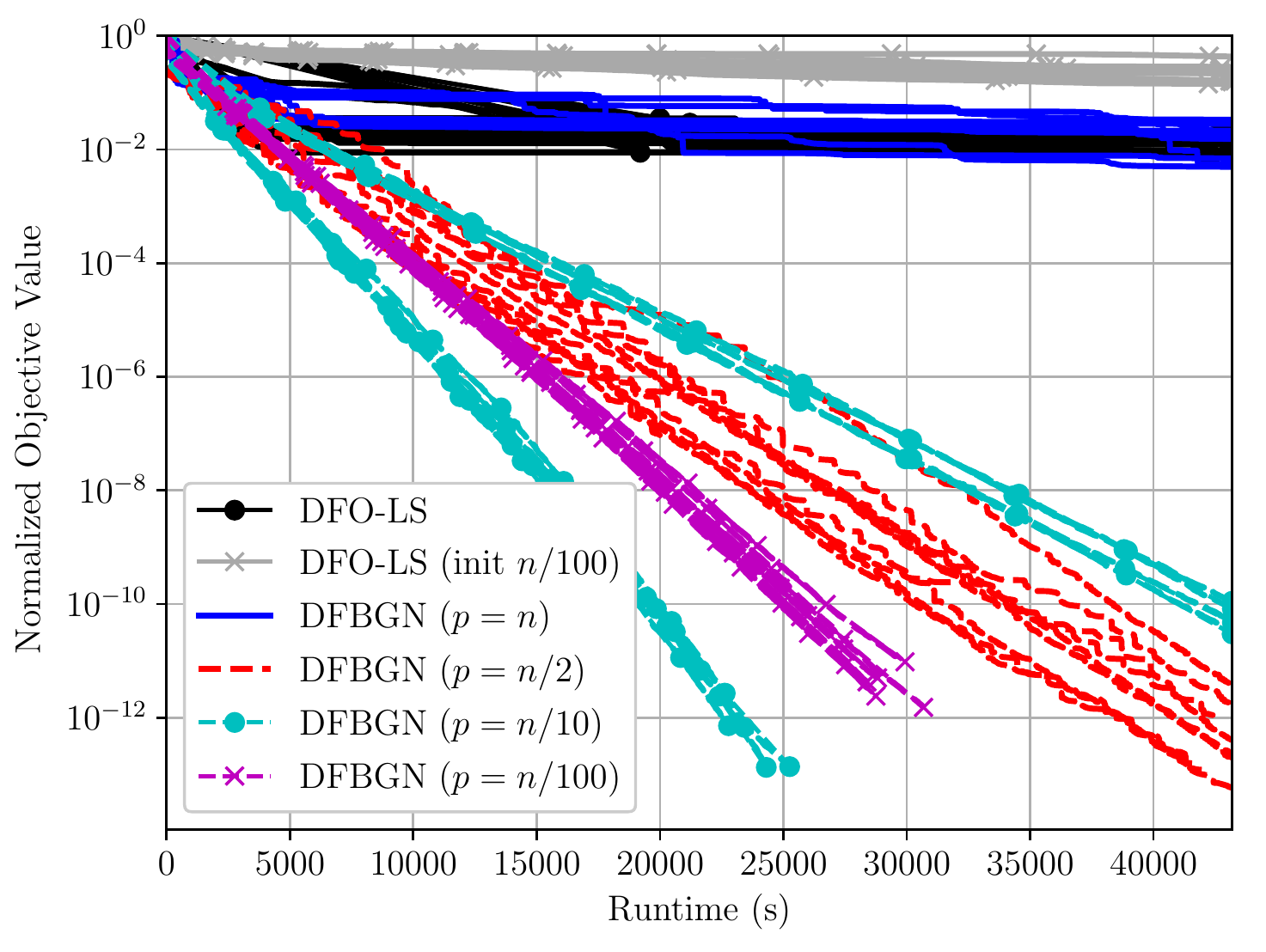}
		\caption{$n=1000$, objective vs.~runtime}
	\end{subfigure}
	\\
	\begin{subfigure}[b]{0.48\textwidth}
		\includegraphics[width=\textwidth]{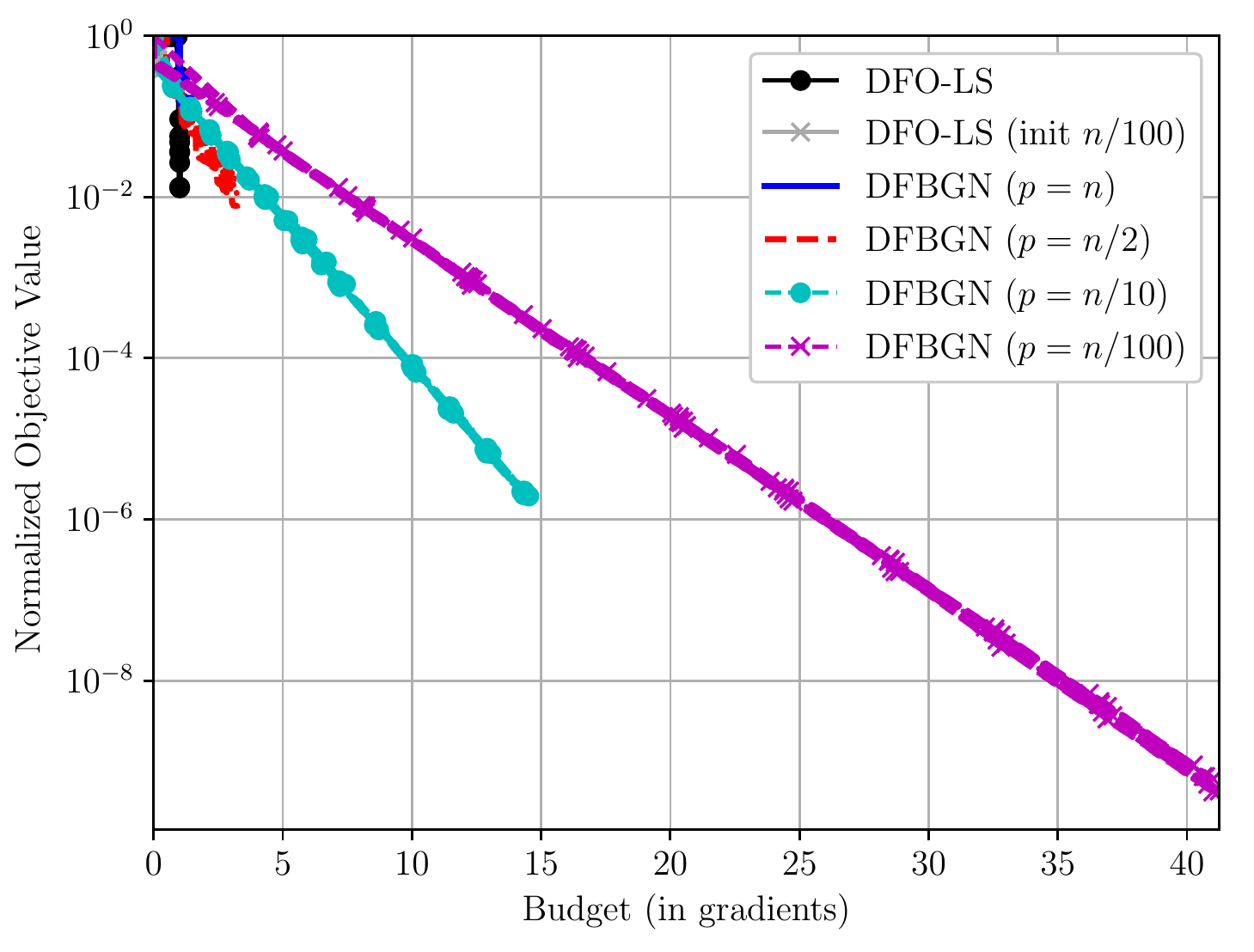}
		\caption{$n=2000$, objective vs.~budget}
	\end{subfigure}
	~
	\begin{subfigure}[b]{0.48\textwidth}
		\includegraphics[width=\textwidth]{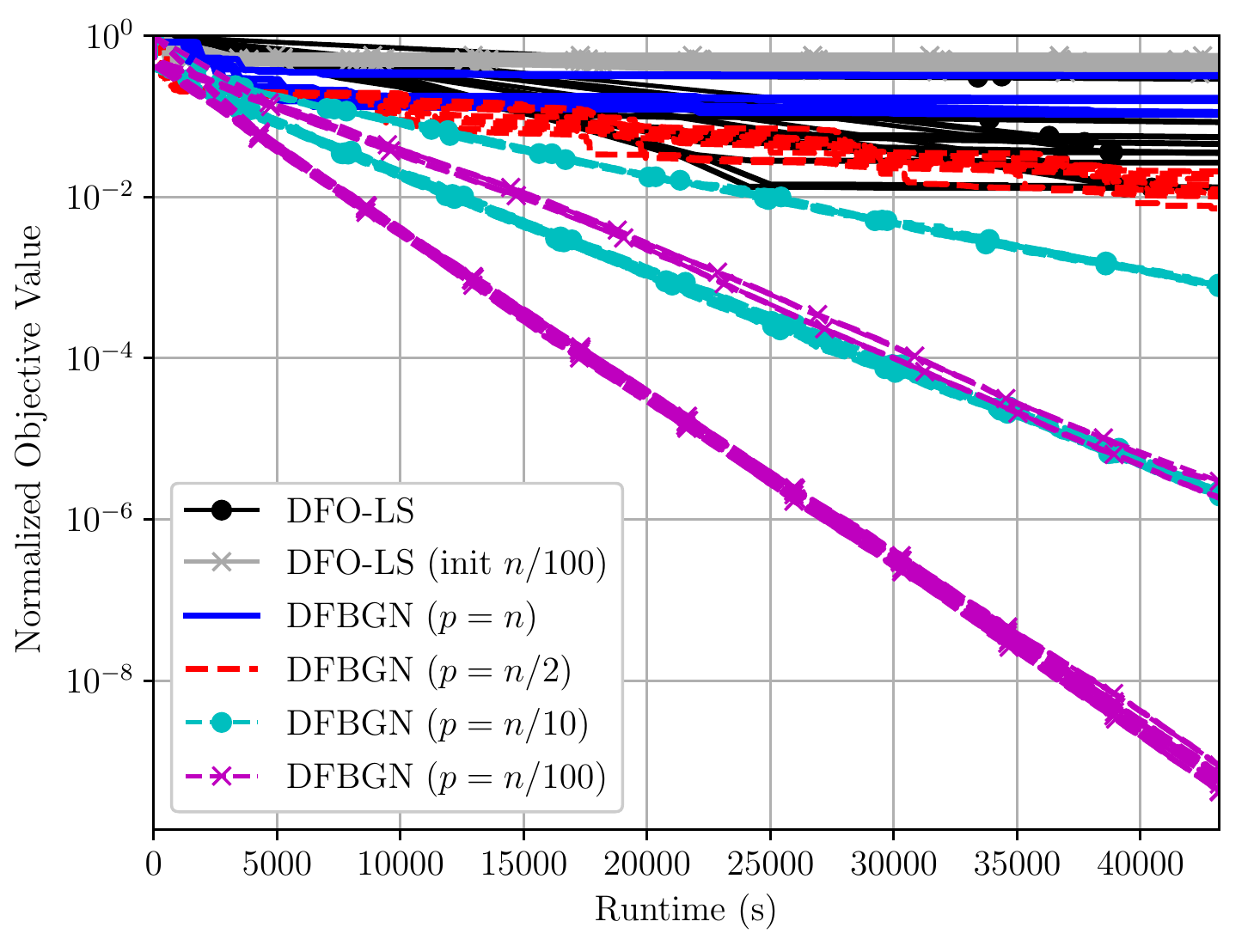}
		\caption{$n=2000$, objective vs.~runtime}
	\end{subfigure}
	\caption{Normalized objective value (versus evaluations and runtime) for 10 runs of DFO-LS and DFBGN on CUTEst problem {\sc arwhdne}. These results use a budget of $100(n+1)$ evaluations and a 12 hour runtime limit per instance.}
	\label{fig_dfbgn_arwhdne}
\end{figure}

\clearpage

To further see the impact of this issue, we now consider how the solvers perform for a variable-dimension test problem, as we increase the underlying dimension.
We run each solver, with the same settings as above, on the CUTEst problem {\sc arwhdne} for different choices of problem dimension $n$.\footnote{This problem appears in the collections (CR) and (CR-large), with $n=100$ and $n=1000$ respectively.}
In \figref{fig_dfbgn_arwhdne} we plot the objective reduction for each solver against budget and runtime for DFO-LS and DFBGN, showing $n=100$, $n=1000$ and $n=2000$.

We see that, when measured on evaluations, both DFO-LS variants achieve the fastest objective reduction, and that DFBGN with small $p$ achieves the slowest objective reduction.
This is in line with our results from \secref{sec_dfbgn_results_evals}.
However, when we instead consider objective decrease against runtime, we see that DFBGN with small $p$ gives the fastest decrease---the larger number of iterations needed by these variants (as seen by the larger number of evaluations) is offset by the substantially reduced per-iteration linear algebra cost.
When viewed against runtime, both DFO-LS variants can only achieve a small objective decrease in the allowed 12 hours, even though they are showing fast decrease against budget, and would achieve higher accuracy than DFBGN if the linear algebra cost were negligible.

\begin{figure}[t]
	\centering
	\begin{subfigure}[b]{0.48\textwidth}
		\includegraphics[width=\textwidth]{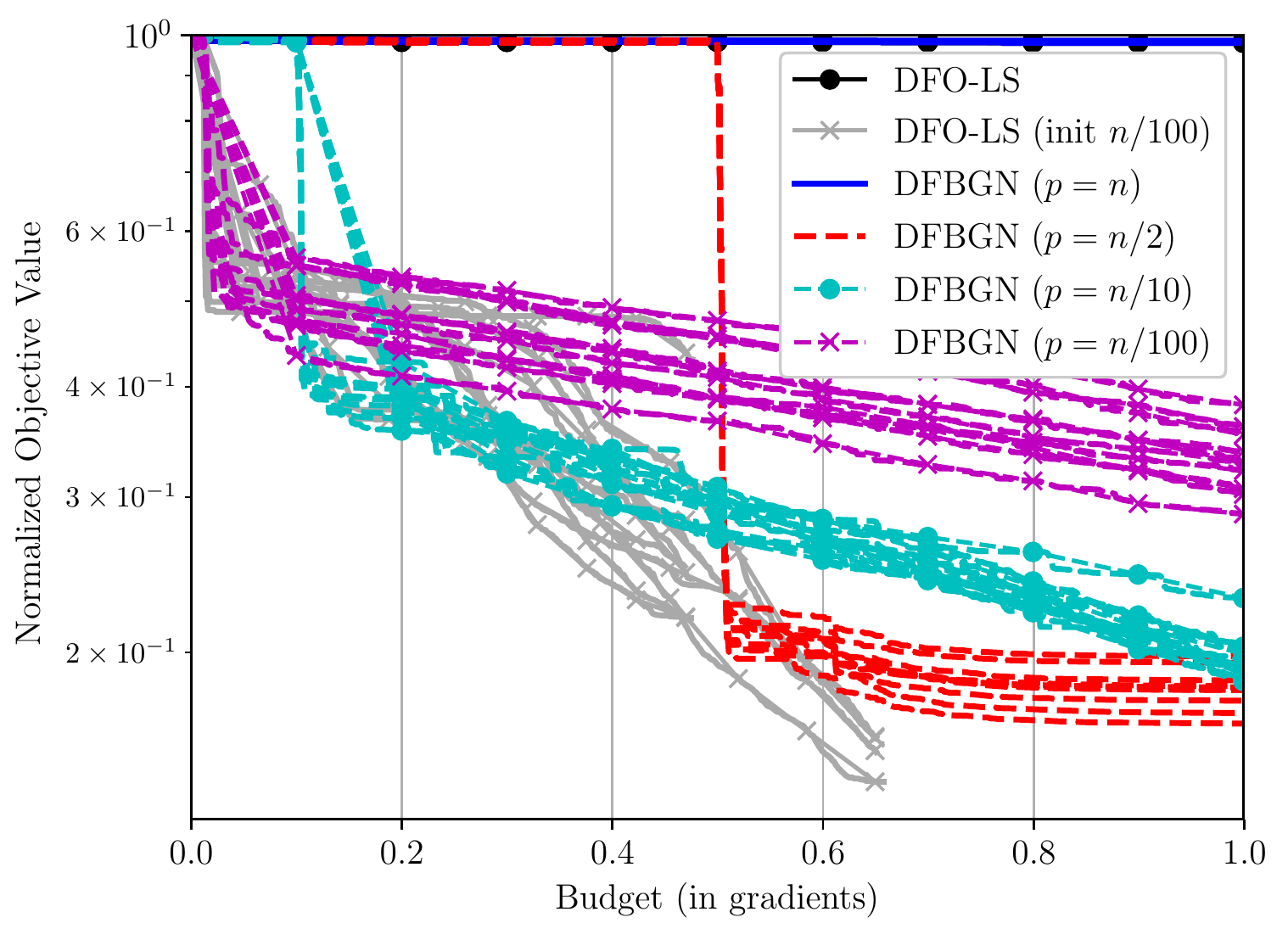}
		\caption{{\sc arwhdne}, $n=1000$}
	\end{subfigure}
	~
	\begin{subfigure}[b]{0.48\textwidth}
		\includegraphics[width=\textwidth]{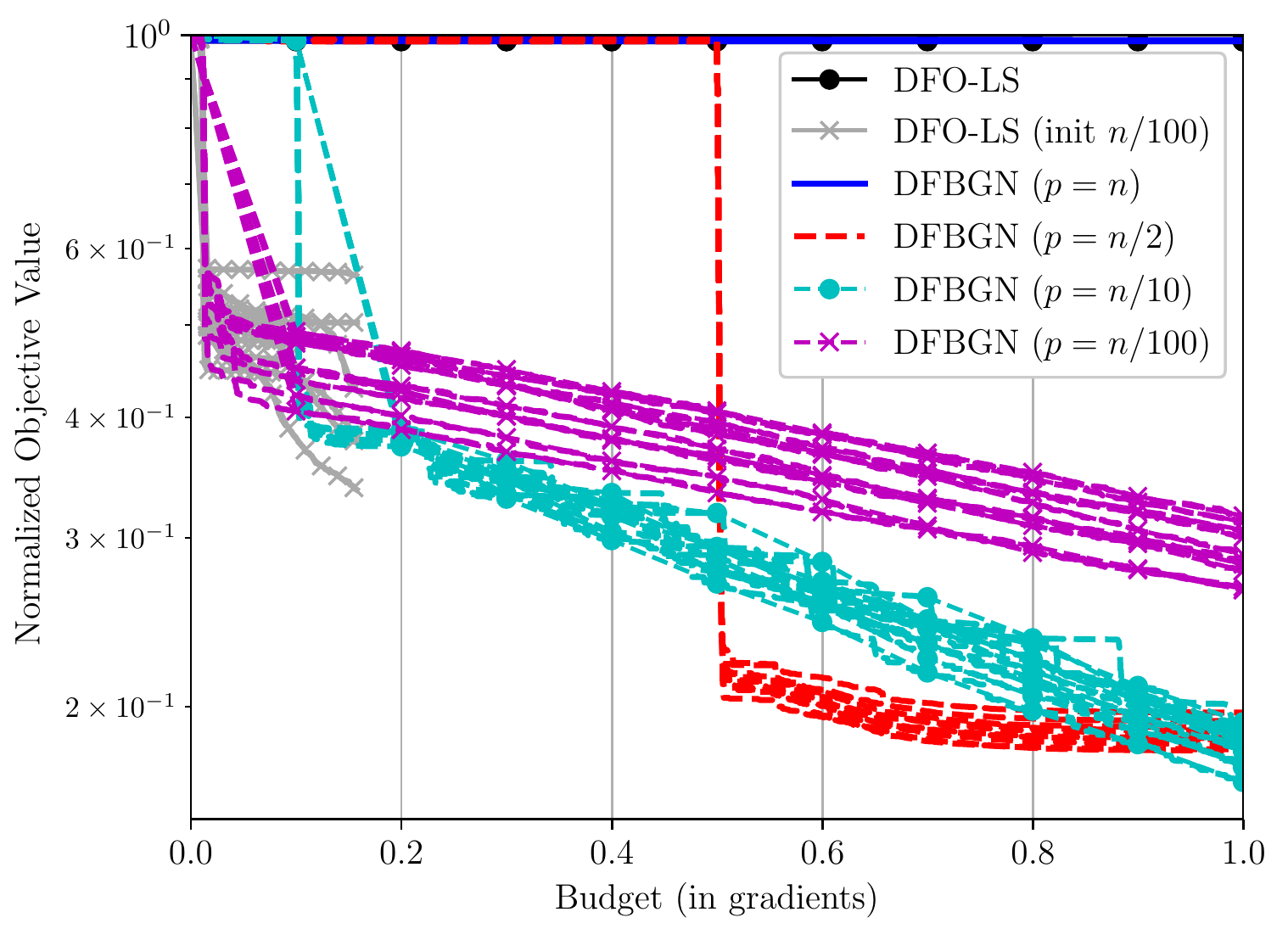}
		\caption{{\sc arwhdne}, $n=2000$}
	\end{subfigure}
	\\
	\begin{subfigure}[b]{0.48\textwidth}
		\includegraphics[width=\textwidth]{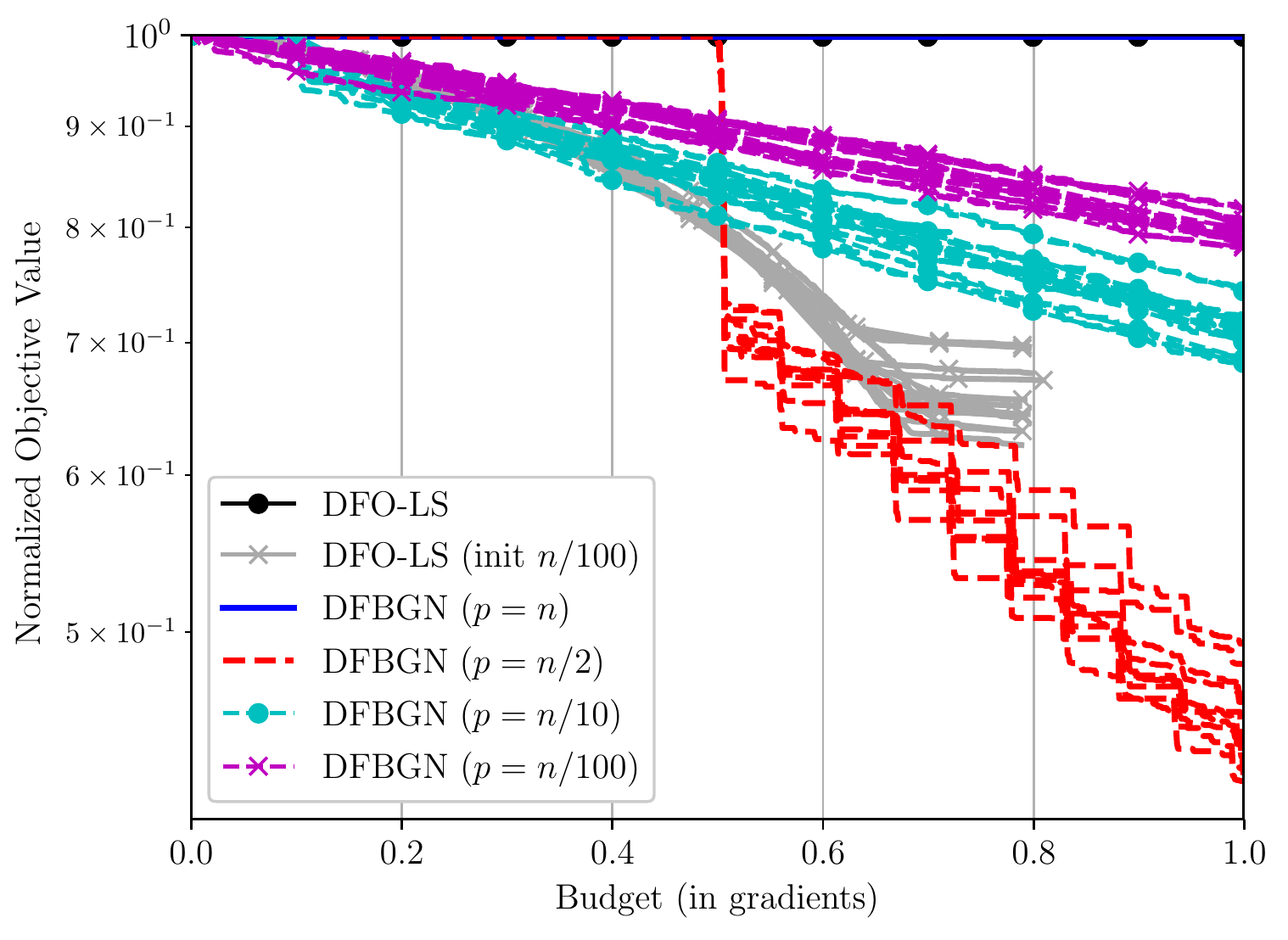}
		\caption{{\sc chandheq}, $n=1000$}
	\end{subfigure}
	~
	\begin{subfigure}[b]{0.48\textwidth}
		\includegraphics[width=\textwidth]{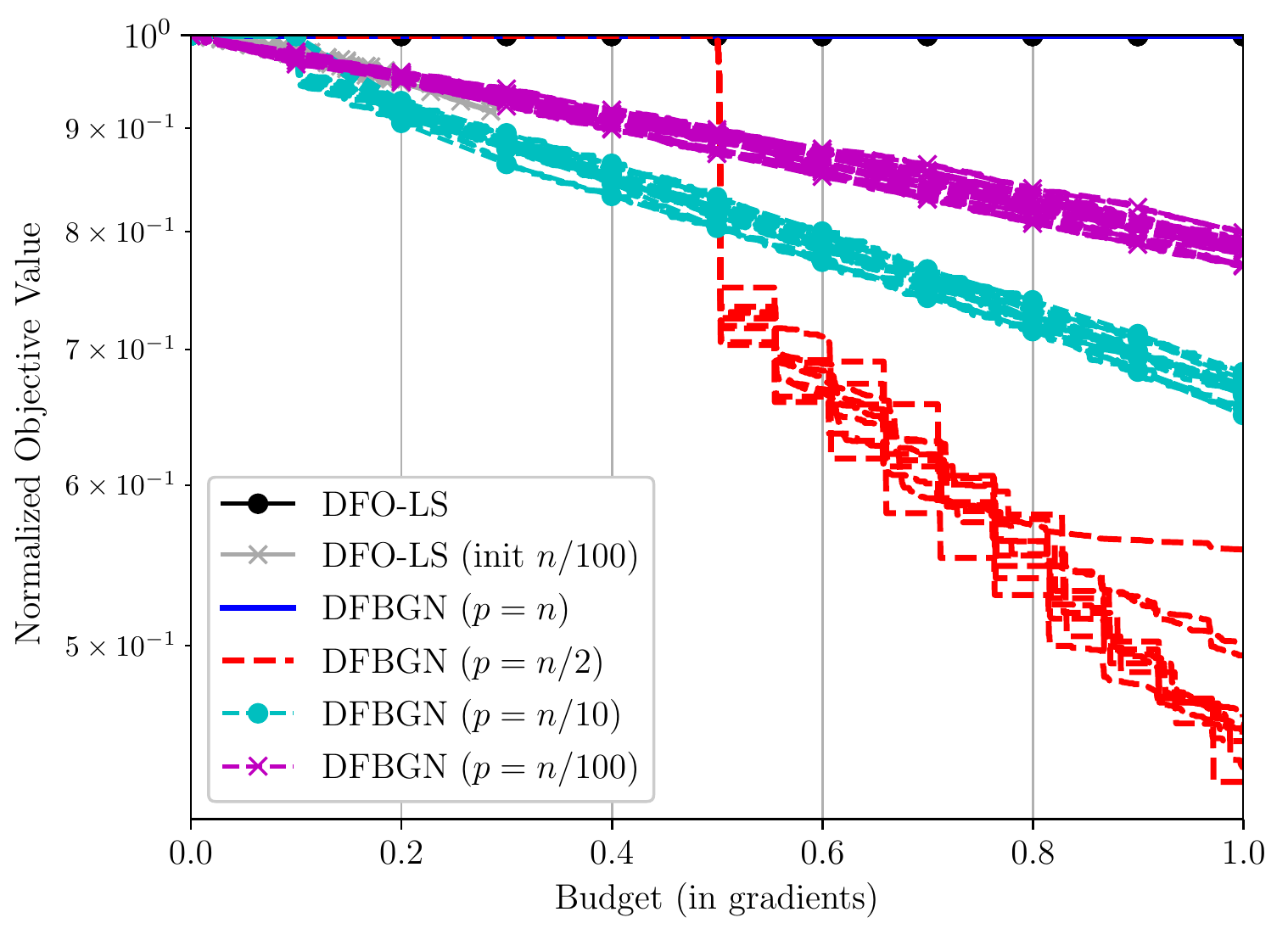}
		\caption{{\sc chandheq}, $n=2000$}
	\end{subfigure}
	\\
	\begin{subfigure}[b]{0.48\textwidth}
		\includegraphics[width=\textwidth]{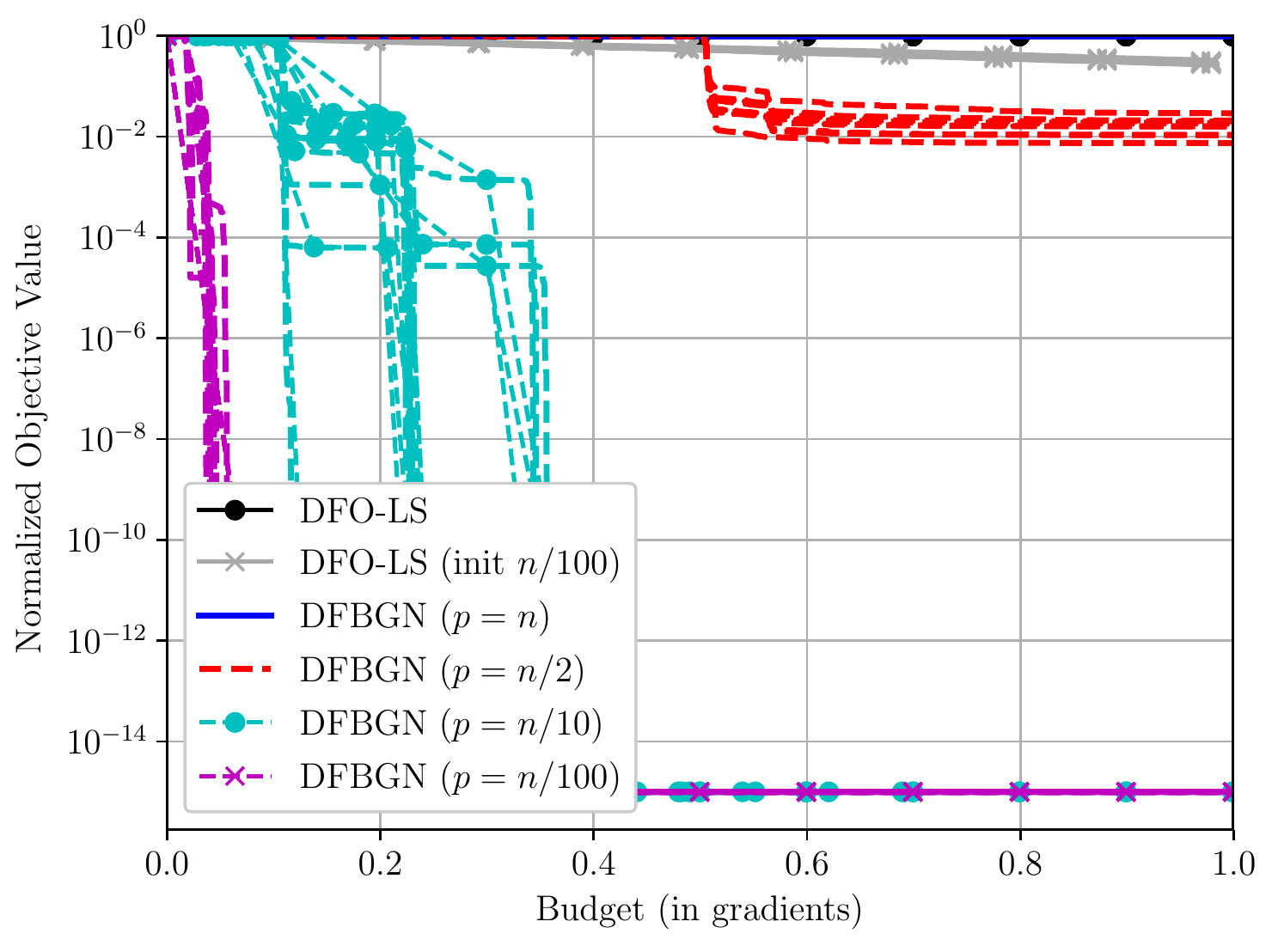}
		\caption{{\sc vardimne}, $n=1000$}
	\end{subfigure}
	~
	\begin{subfigure}[b]{0.48\textwidth}
		\includegraphics[width=\textwidth]{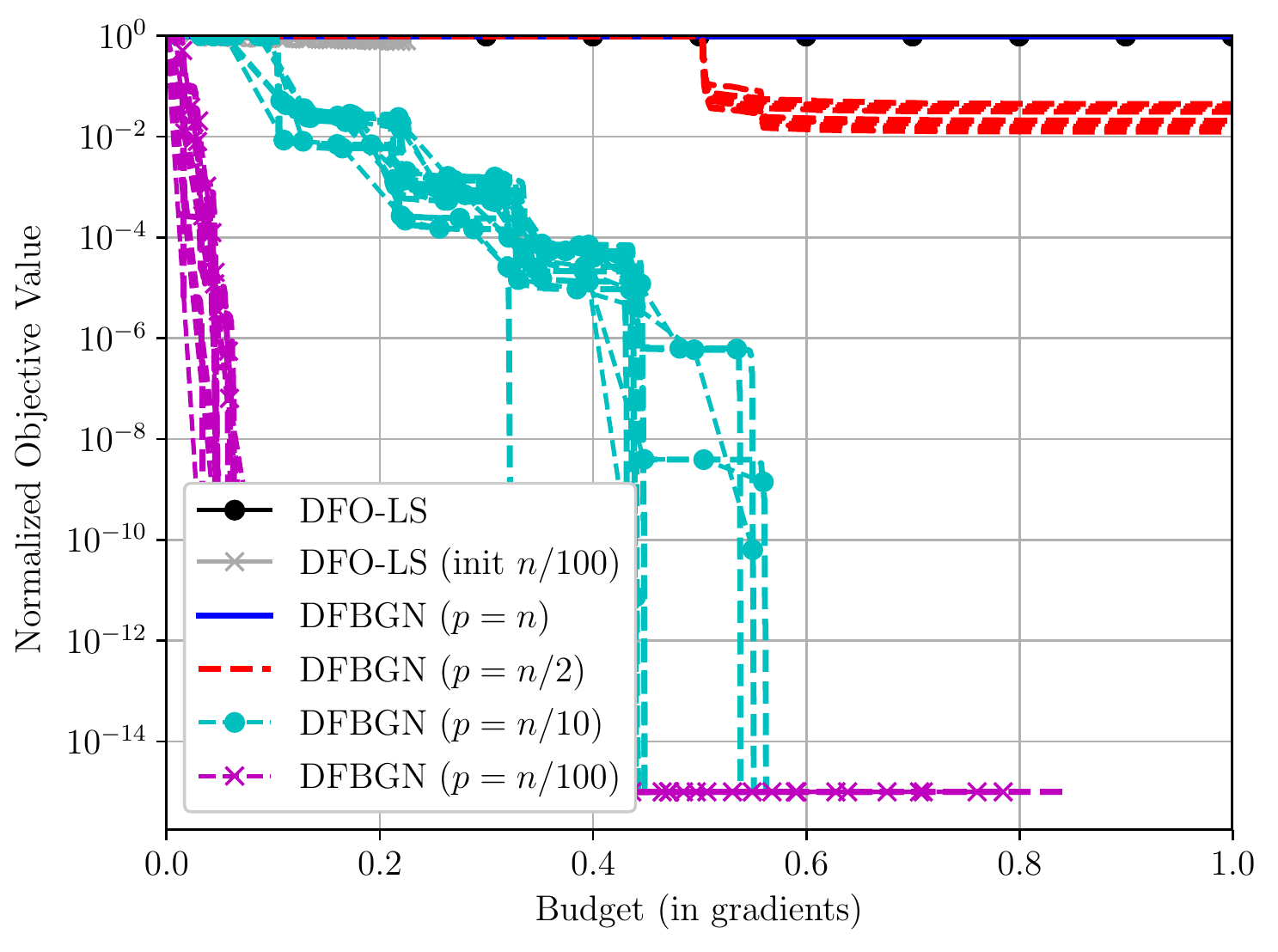}
		\caption{{\sc vardimne}, $n=2000$}
	\end{subfigure}
	\caption{Normalized objective value (versus evaluations) for 10 runs of DFO-LS and DFBGN on different CUTEst problems (all with $n=1000$ and $n=2000$). These results use a budget of $n+1$ evaluations and a 12 hour runtime limit per instance.}
	\label{fig_dfbgn_smallbudget}
\end{figure}

\subsection{Results for Small Budgets}
Another benefit of DFBGN is that it has a small initialization cost of $p+1$ evaluations.
When $n$ is large, it is more likely for a user to be limited by a budget of fewer than $n$ evaluations.
Here, we examine how DFBGN compares for small budgets to DFO-LS with reduced initialization cost.

We recall from \remref{rem_dfols_growing} that DFO-LS with reduced initialization cost progressively increases the dimension of the subspace of its interpolation model, until it reaches the whole space $\R^n$ (after approximately $n+1$ evaluations), while in DFBGN we restrict the dimension at all iterations.

In \figref{fig_dfbgn_smallbudget} we consider three variable-dimensional CUTEst problems from (CR) and (CR-large), all using $n=1000$ and $n=2000$.
We show the objective decrease against budget for 10 runs of each solver, restricted to a maximum of $n+1$ evaluations.
We see that the smaller $p$ used in DFBGN, the faster DFBGN is able to make progress (due to the lower number of initial evaluations).
However, this is offset by the faster objective decrease achieved by larger $p$ values (after the higher initialization cost)---if the user can afford a larger $p$, both in terms of linear algebra and initial evaluations, then this is usually a better option.
An exception to this is the problem {\sc vardimne}, where its simple structure means DFBGN with small $p$ solves the problem to very high accuracy with very few evaluations, substantially outperforming both DFBGN with larger $p$, and DFO-LS with reduced initialization cost.

In \figref{fig_dfbgn_smallbudget} we also show the decrease for DFO-LS with full initialization cost and DFBGN with $p=n$, but they use the full budget on initialization, and so make no progress.
However, in addition, we show DFO-LS with a reduced initialization cost of $n/100$ evaluations.
This variant performs well, in most cases matching the decrease of DFBGN with $p=n/100$ initially, but achieving a faster objective reduction against budget---this matches with our previous observations.
However, the extra cost of the linear algebra means that DFO-LS with reduced initialization does not end up using the full budget, instead hitting the 12 hour timeout.
This is most clear when comparing the results for $n=1000$ with $n=2000$, where DFO-LS with reduced initialization cost begins by achieving a similar decrease in both cases, but hits the timeout more quickly with $n=2000$, and so terminates after fewer objective evaluations (with a corresponding smaller objective decrease).

We analyze this more systematically in \figref{fig_dfbgn_profiles_smallbudget}, where we show data profiles (measured on budget) of DFBGN and DFO-LS on the (CR-large) problem collection, for low accuracy and small budgets.
These results verify our conclusions: DFBGN with small $p$ can make progress on many problems with a very short budget (fewer than $n+1$ evaluations), and outperform DFO-LS with reduced initialization cost due to its slow runtime.
However, once we reach a budget of more than $n+1$ evaluations, then DFO-LS and DFBGN with $p=n$ become the best-performing solvers (when measuring on evaluations only).
They are also able to achieve a higher level of accuracy compared to DFBGN with small $p$. 

\clearpage

\begin{figure}[tbh]
	\centering
	\begin{subfigure}[b]{0.48\textwidth}
		\includegraphics[width=\textwidth]{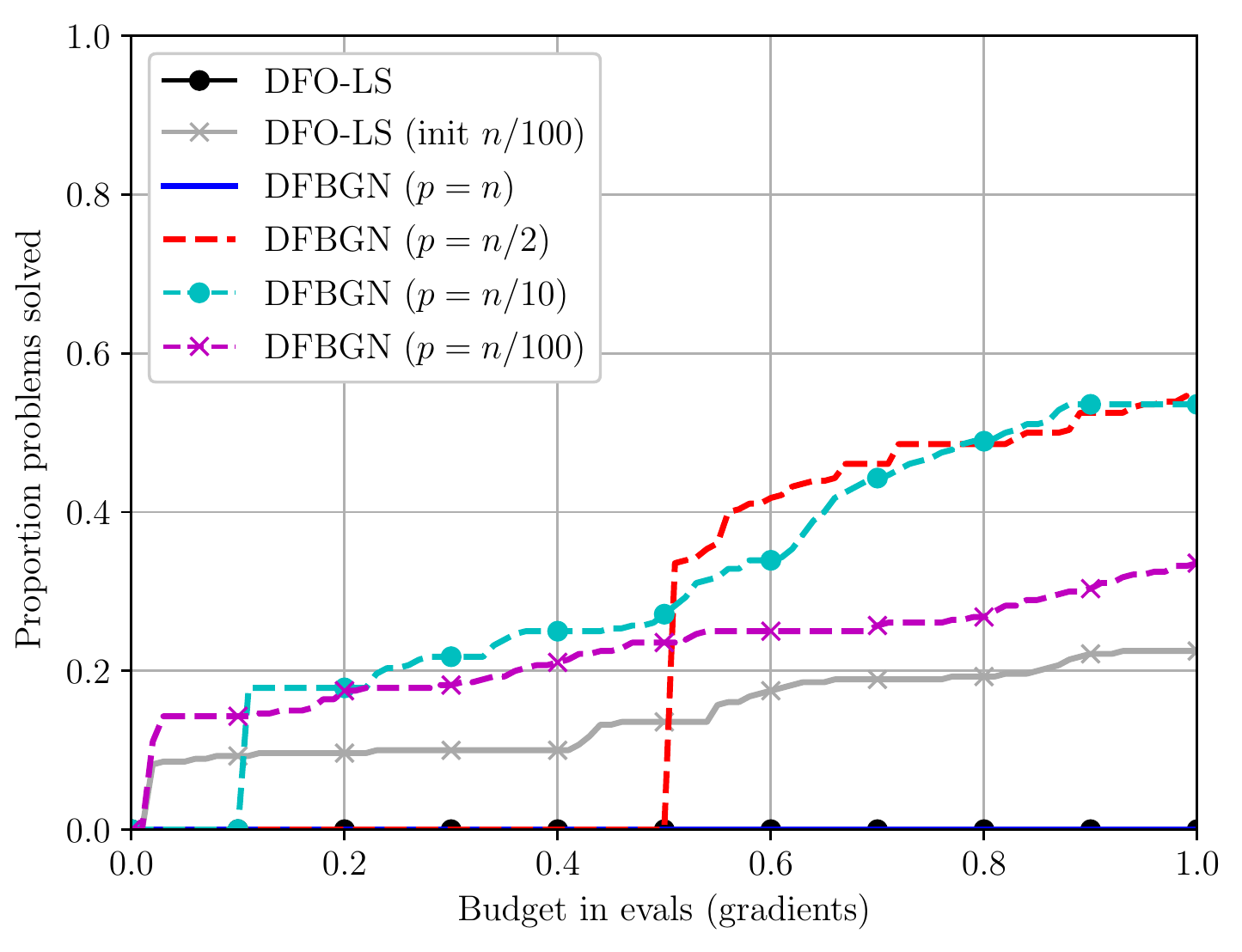}
		\caption{$\tau=0.5$, budget $n+1$ evaluations}
	\end{subfigure}
	~
	\begin{subfigure}[b]{0.48\textwidth}
		\includegraphics[width=\textwidth]{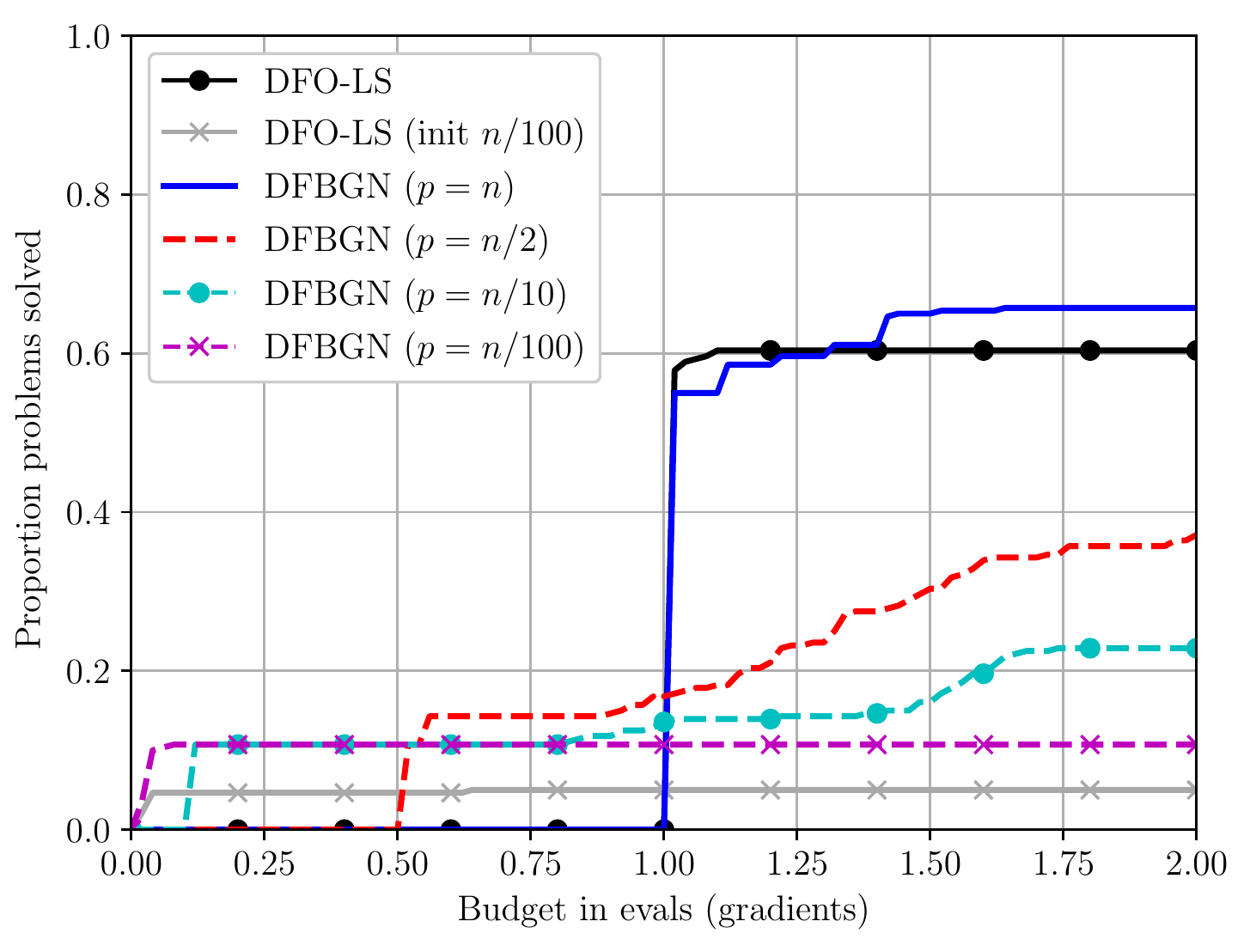}
		\caption{$\tau=0.1$, budget $2(n+1)$ evaluations}
	\end{subfigure}
	\caption{Data profiles (in evaluations) comparing DFO-LS (with and without reduced initialization cost) with DFBGN (various $p$ choices) for different accuracy levels and budgets. Results are an average of 10 runs for each problem, with a budget of $n+1$ or $2(n+1)$ evaluations and a 12 hour runtime limit per instance. The problem collection is (CR-large).}
	\label{fig_dfbgn_profiles_smallbudget}
\end{figure}

\begin{figure}[t]
	\centering
	\begin{subfigure}[b]{0.48\textwidth}
		\includegraphics[width=\textwidth]{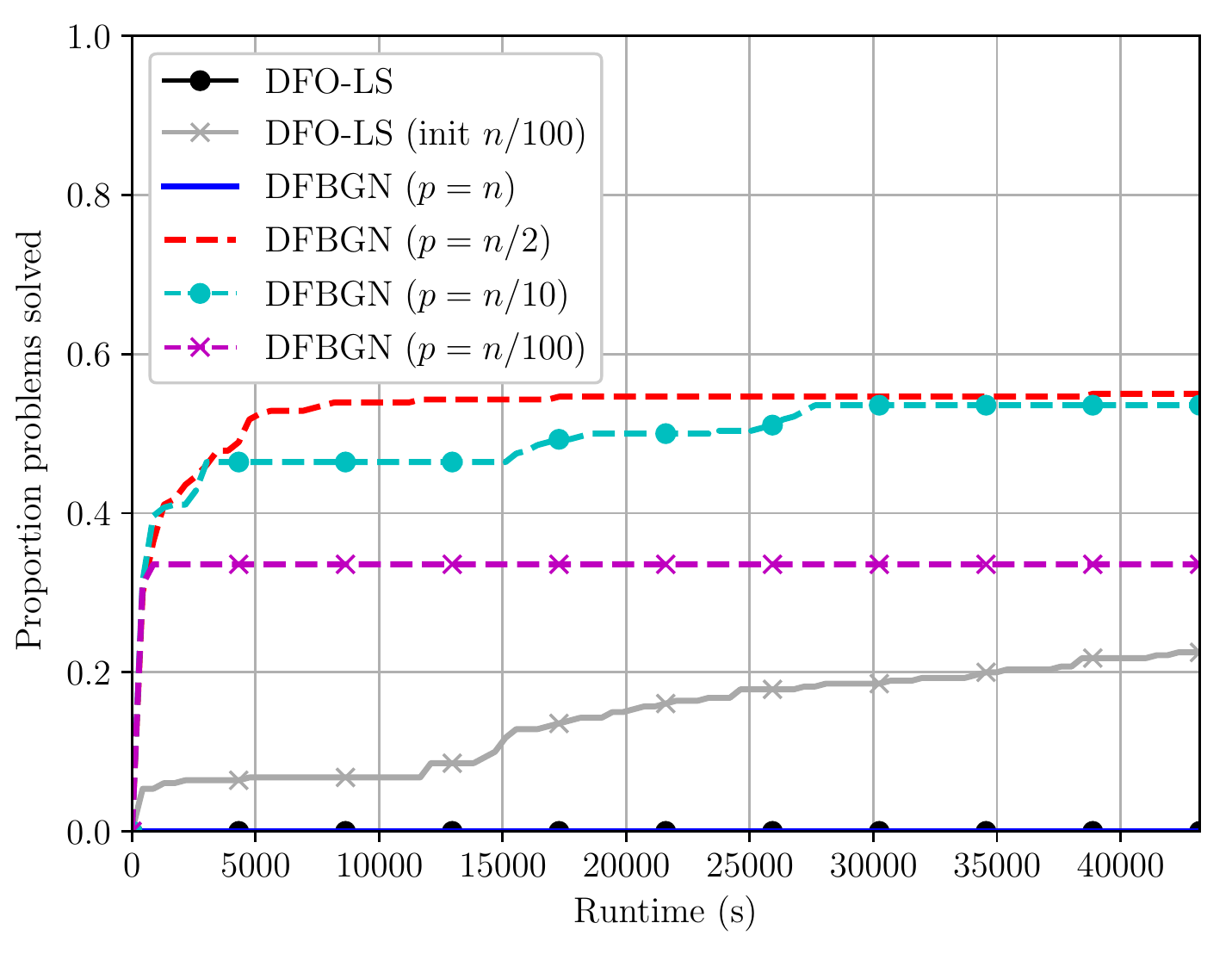}
		\caption{$\tau=0.5$, $n+1$ evaluations (vs.~runtime)}
	\end{subfigure}
	~
	\begin{subfigure}[b]{0.48\textwidth}
		\includegraphics[width=\textwidth]{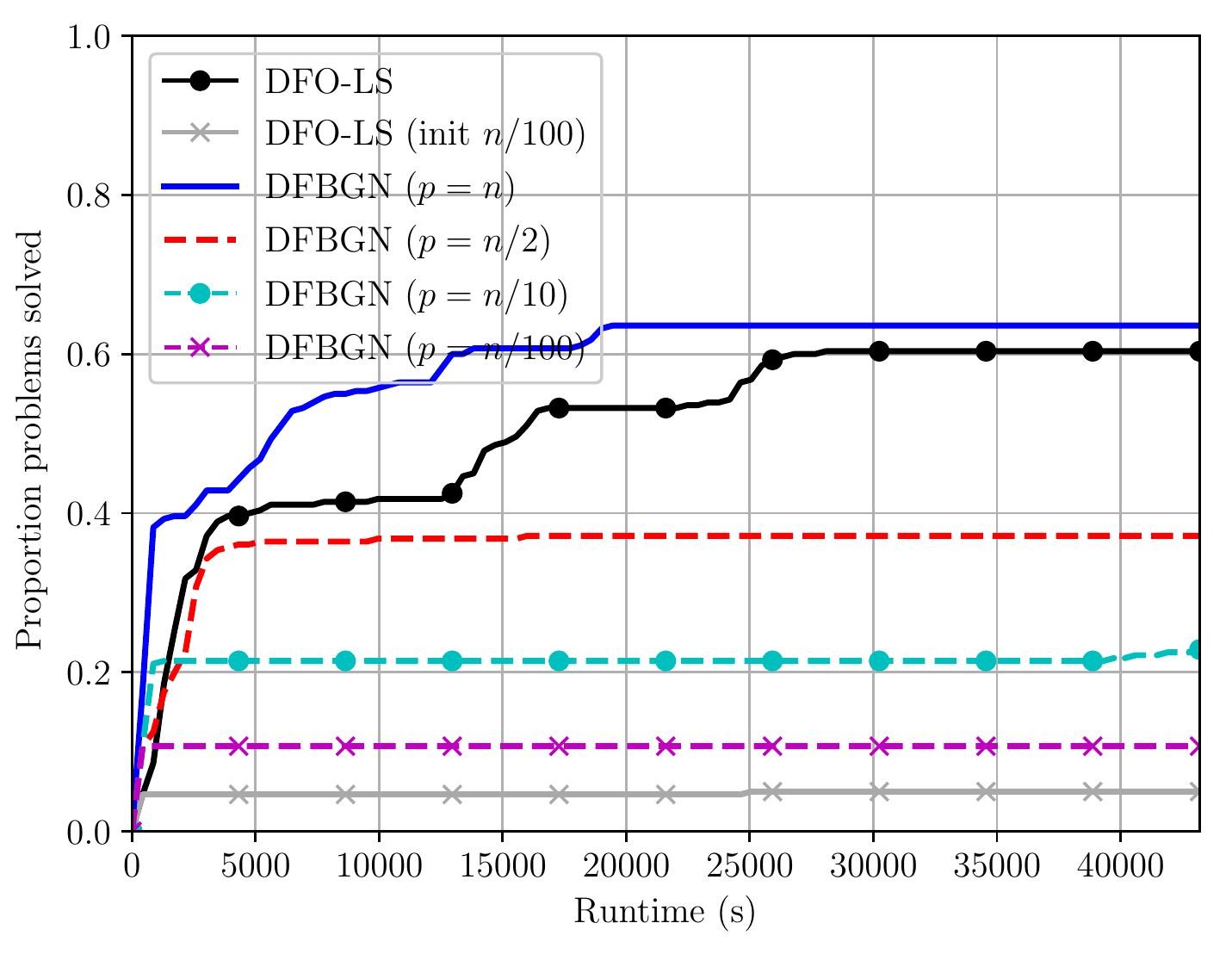}
		\caption{$\tau=0.1$, $2(n+1)$ evaluations (vs.~runtime)}
	\end{subfigure}
	\caption{Data profiles (in runtime) comparing DFO-LS (with and without reduced initialization cost) with DFBGN (various $p$ choices) for different accuracy levels and budgets. Results are an average of 10 runs for each problem, with a budget of $n+1$ or $2(n+1)$ evaluations and a 12 hour runtime limit per instance. The problem collection is (CR-large).}
	\label{fig_dfbgn_profiles_smallbudget_runtime}
\end{figure}

Lastly, in \figref{fig_dfbgn_profiles_smallbudget_runtime} we show the same results as \figref{fig_dfbgn_profiles_smallbudget}, but showing profiles measured on runtime.
We note that we are only measuring the linear algebra costs, as the cost of objective evaluation for our problems is negligible.
Here, the benefits of DFBGN with small $p$ are not seen.
This is because the problems that can be solved by DFBGN with small $p$ using very few evaluations are likely easier, and so can likely be solved by DFBGN with large $p$ in few iterations.
Thus, the runtime requirements for DFBGN with large $p$ to solve the problem are not large---even though they have a higher per-iteration cost, the number of iterations is small.
In this setting, therefore, the benefit of DFBGN with small $p$ is not lower linear algebra costs, but fewer evaluations---which is likely to be the more relevant issue in this small-budget regime.

\section{Conclusions and Future Work}
The development of scalable derivative-free optimization algorithms is an active area of research with many applications.
In model-based DFO, the high per-iteration linear algebra cost associated (primarily) with interpolation model creation and point management is a barrier to its utility for large-scale problems.
To address this, we introduce three model-based DFO algorithms for large-scale problems.

First, RSDFO is a general framework for model-based DFO based on model construction and minimization in random subspaces, and is suitable for general smooth nonconvex objectives.
This is specialized to nonlinear least-squares problems in RSDFO-GN, a version of RSDFO based on Gauss-Newton interpolation models built in subspaces.
Lastly, we introduce DFBGN, a practical implementation of RSDFO-GN. 
In all cases, the scalability of these methods arises from the construction and minimization of models in $p$-dimensional subspaces of the ambient space $\R^n$.
The subspace dimension can be specified by the user to reflect the computational resources available for linear algebra calculations.

We prove high-probability worst-case complexity bounds for RSDFO, and show that RSDFO-GN inherits the same bounds.
In terms of selecting the subspace dimension, we show that by using matrices based on Johnson-Lindenstrauss transformations, we can choose $p$ to be independent of the ambient dimension $n$.
Our analysis extends to DFO the techniques in \cite{Fowkes2020,Shao2021}, and yields similar results to probabilistic direct search \cite{Gratton2015} and standard model-based DFO \cite{Garmanjani2016,Cartis2019a}.
Our results also imply almost-sure global convergence to first-order stationary points.

Our practical implementation of RSDFO-GN, DFBGN, has very low computational requirements: asymptotically, linear in the ambient dimension rather than cubic for standard model-based DFO.
After extensive algorithm development described here, our implementation is simple and combines several techniques for modifying the interpolation set which allows it to still make progress with few objective evaluations (an important consideration for DFO techniques).
A Python version of DFBGN is available on Github.\footnote{\url{https://github.com/numericalalgorithmsgroup/dfbgn}}

For medium-scale problems, DFBGN operating in the full ambient space ($p=n$) has similar performance to DFO-LS \cite{Cartis2018} when measured by objective evaluations, validating the techniques introduced in the practical implementation.
However, DFBGN (with any choice of subspace dimension) has substantially faster runtime, which means it is much more effective than DFO-LS at solving large-scale problems from CUTEst, even when working in a very low-dimensional subspace.
Further, in the case of expensive objective evaluations, working a subspace means that DFBGN can make progress with very few evaluations, many fewer than the $n+1$ needed for standard methods to build their initial model.
Overall, the implementation of DFBGN is suitable for large-scale problems both when objective evaluations are cheap (and linear algebra costs dominate) or when evaluations are expensive (and the initialization cost of standard methods is impractical).

Future work will focus on extending the ideas from the implementation DFBGN to the case of general objectives with quadratic models. 
This will bring the available software in line with the theoretical guarantees for RSDFO.
We note that model-based DFO for nonlinear least-squares problems has been adapted to include sketching methods, which use randomization to reduce the number of residuals considered at each iteration \cite{Cartis2020}.
We also delegate to future work the development of techniques for nonlinear least-squares problems which combine sketching (i.e.~dimensionality reduction in the observation space) with our subspace approach (i.e.~dimensionality reduction in variable space), and further study of methods for adaptively selecting a subspace dimension (c.f.~\secref{sec_dfbgn_full_algo}).

\subsection{Acknowledgements}
The authors would like to acknowledge the use of the University of Oxford Advanced Research Computing (ARC) facility in carrying out this work.\footnote{\url{http://dx.doi.org/10.5281/zenodo.22558}}

\addcontentsline{toc}{section}{References} 
\bibliographystyle{siam}
\bibliography{refs} 

\clearpage
\appendix

\section{Extended Results for DFBGN} \label{app_dfbgn_extra_numerics}
In \figref{fig_dfbgn_mw}, we show performance profiles comparing DFBGN with DFO-LS on the (MW) problem collection.
Since these problems are low-dimensional ($n\leq 12$), they do not represent the setting for which DFBGN is designed, however we include them here for completeness.

Similar to \figref{fig_dfbgn_cutest100}, we see that DFBGN performs better (in terms of evaluations) the larger the subspace size $p$, with the performance with $p=n$ similar to DFO-LS.
For very low accuracy $\tau=0.5$, DFBGN with $p<n$ sometimes outperforms DFO-LS (with the full initialization cost), but not DFO-LS with reduced initialization cost.
\vspace{-1em}

\begin{figure}[H]
	\centering
	\begin{subfigure}[b]{0.43\textwidth}
		\includegraphics[width=\textwidth]{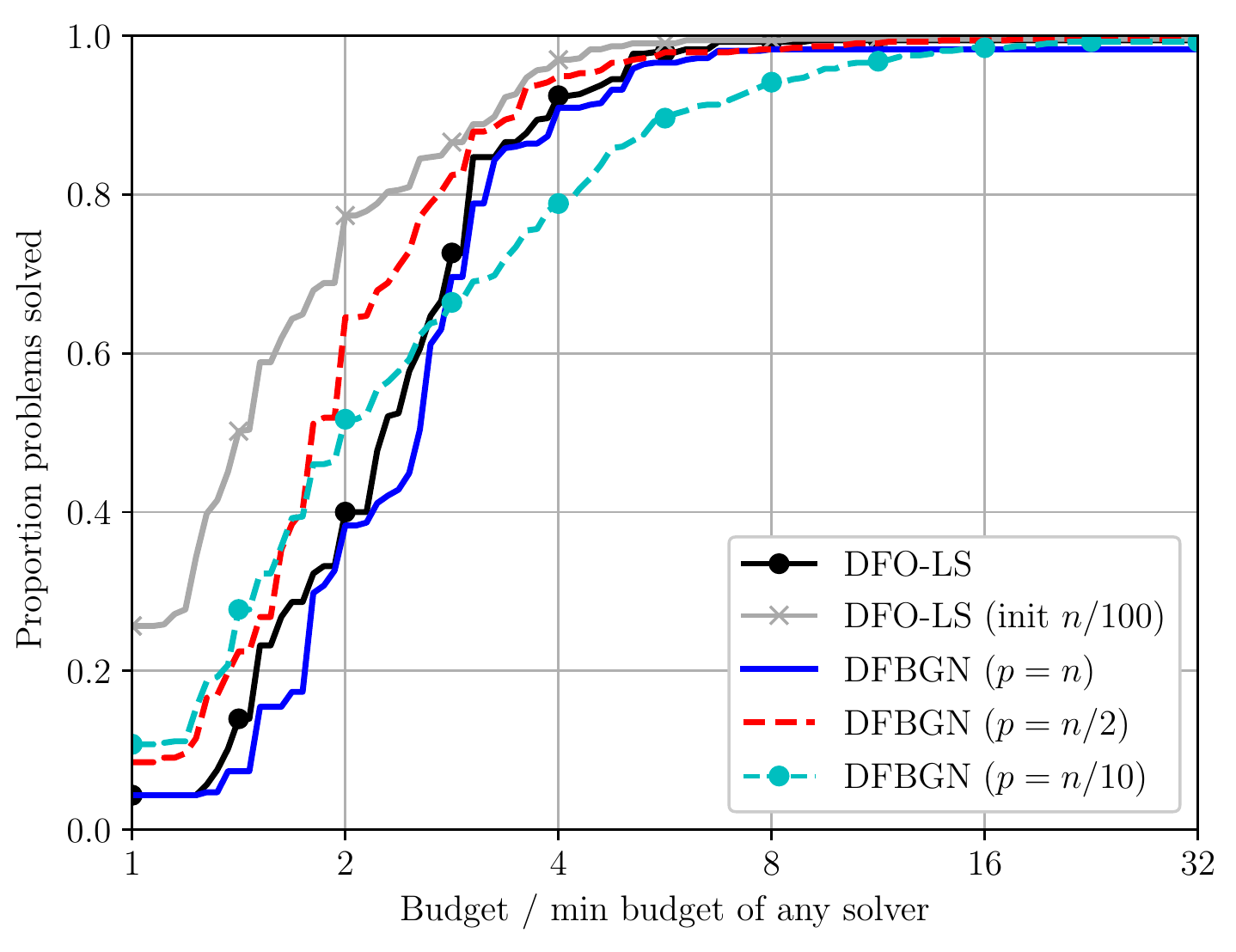}
		\caption{$\tau=0.5$}
	\end{subfigure}
	~
	\begin{subfigure}[b]{0.43\textwidth}
		\includegraphics[width=\textwidth]{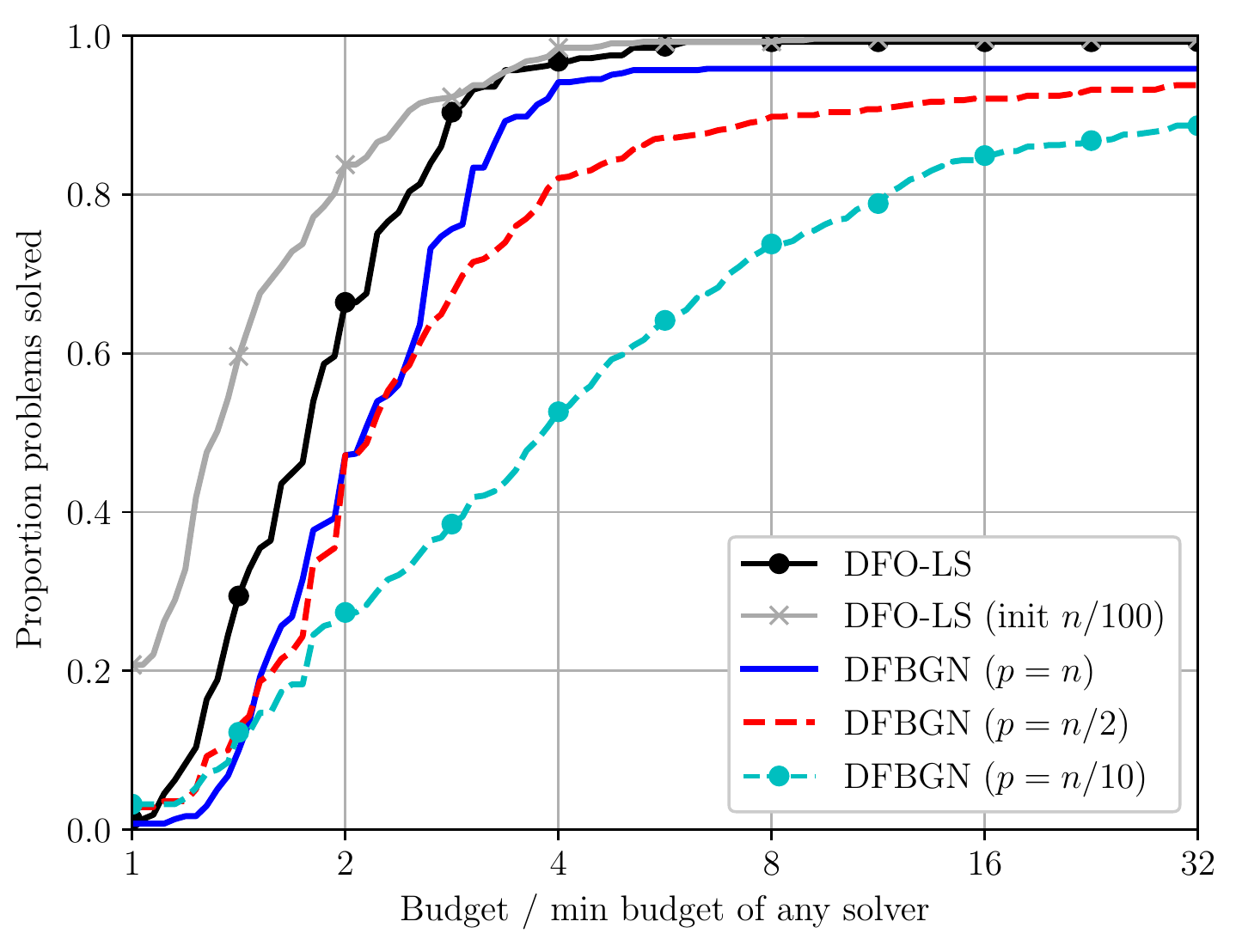}
		\caption{$\tau=10^{-1}$}
	\end{subfigure}
	\\
	\begin{subfigure}[b]{0.43\textwidth}
		\includegraphics[width=\textwidth]{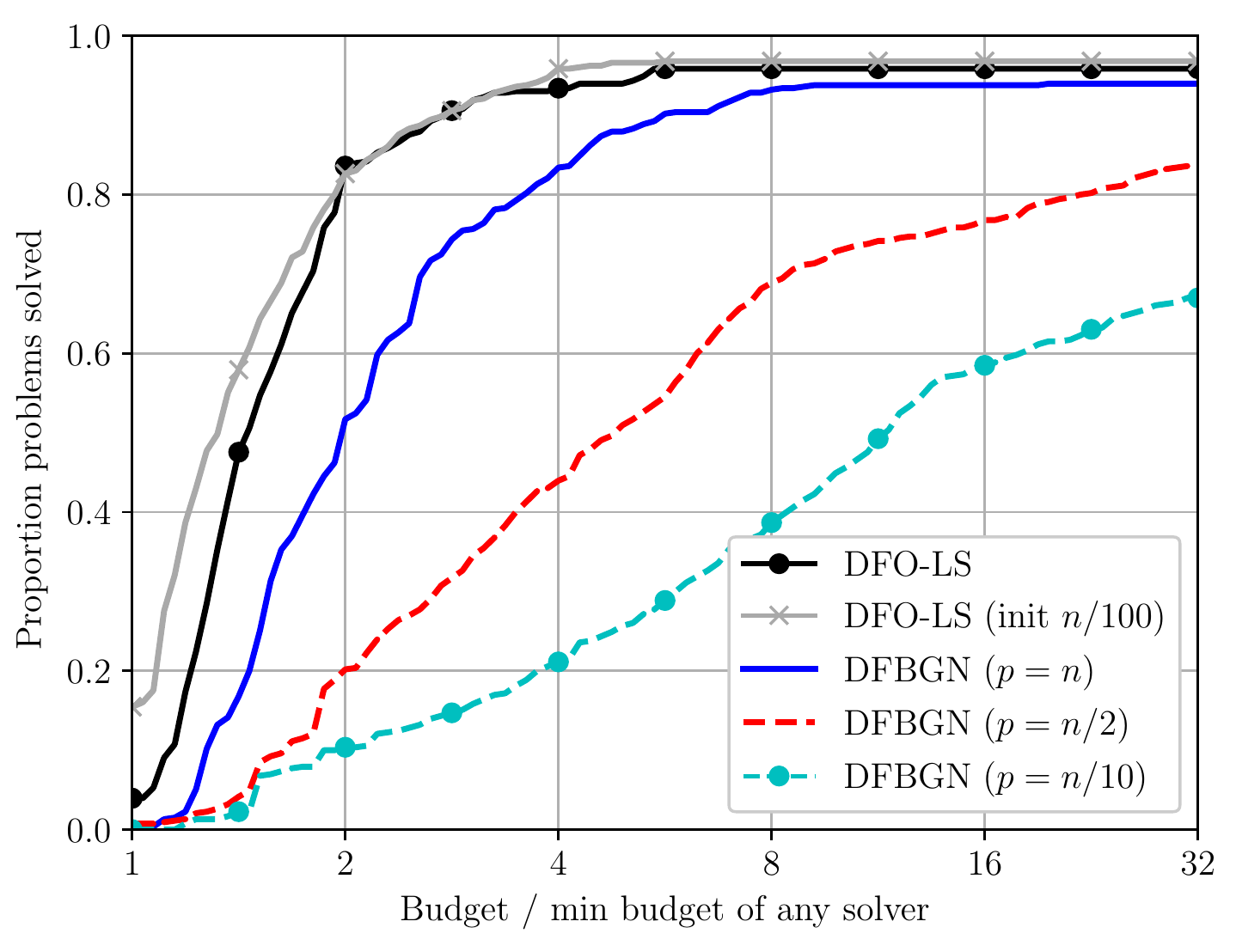}
		\caption{$\tau=10^{-3}$}
	\end{subfigure}
	~
	\begin{subfigure}[b]{0.43\textwidth}
		\includegraphics[width=\textwidth]{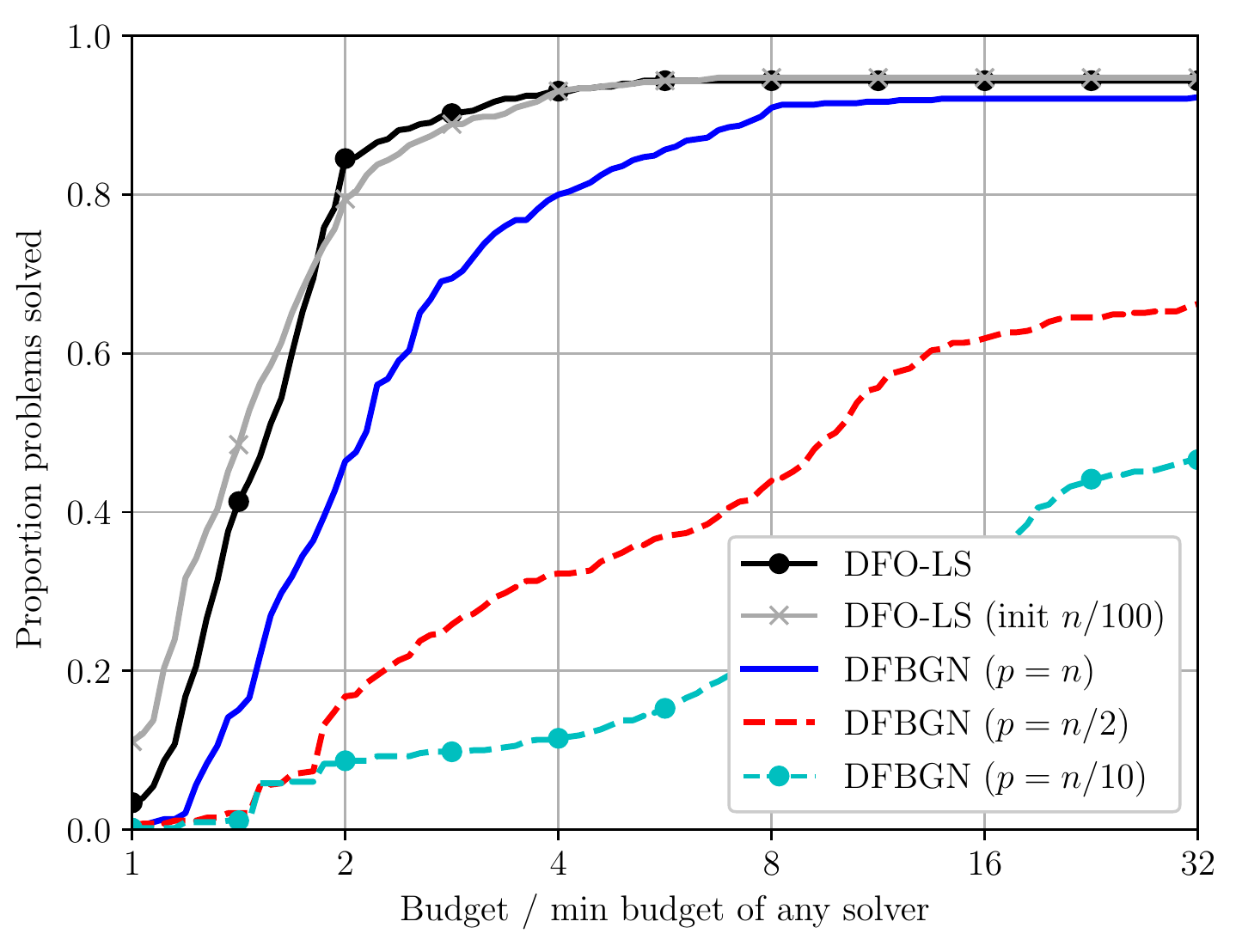}
		\caption{$\tau=10^{-5}$}
	\end{subfigure}
	\caption{Performance profiles (in evaluations) comparing DFO-LS (with and without reduced initialization cost) with DFBGN (various $p$ choices) for different accuracy levels. Results are an average of 10 runs for each problem, with a budget of $100(n+1)$ evaluations and a 12 hour runtime limit per instance. The problem collection is (MW).}
	\label{fig_dfbgn_mw}
\end{figure}

\section{Large-Scale Test Problems (CR-large)} \label{app_cr_large_problems}
\begin{table}[H]
	\centering
	\small{
	\begin{tabular}{rlccccc} 
		\hline\noalign{\smallskip}
		\# & Problem & $n$ & $m$ & $2f(\b{x}_0)$ & $2f(\bx^*)$ & Parameters \\ \noalign{\smallskip}\hline\noalign{\smallskip}
		1 & ARGLALE & 2000 & 4000 & 10000 & 2000 & $(N,M)=(2000,4000)$ \\
		2 & ARGLBLE & 2000 & 4000 & $8.545072\times 10^{22}$ & 999.6250 & $(N,M)=(2000,4000)$ \\
		3 & ARGTRIG & 1000 & 1000 & 333.0006 & 0 & $N=1000$ \\
		4 & ARTIF & 5000 & 5000 & 1827.355 & 0 & $N=5000$ \\
		5 & ARWHDNE & 5000 & 9998 & 24995 & 1396.793 & $N=5000$ \\
		6 & BDVALUES & 1000 & 1000 & $1.996774\times10^{4}$ & 0 & $NDP=1002$ \\
		7 & BRATU2D & 4900 & 4900 & $3.085195\times 10^{-3}$ & 0 & $P=72$ \\
		8 & BRATU2DT & 4900 & 4900 & $8.937521\times 10^{-3}$ & $7.078014\times 10^{-11}$ & $P=72$ \\
		9 & BRATU3D & 3375 & 3375 & 2.386977 & 0 & $P=17$ \\
		10 & BROWNALE & 1000 & 1000 & $2.502498\times 10^{8}$ & 0 & $N=1000$ \\
		11 & BROYDN3D & 1000 & 1000 & 1011 & 0 & $N=1000$ \\
		12 & BROYDNBD & 5000 & 5000 & 124904 & 0 & $N=5000$ \\
		13 & CBRATU2D & 2888 & 2888 & $1.560446\times 10^{-2}$ & 0 & $P=40$ \\
		14 & CHANDHEQ & 1000 & 1000 & 69.41682 & 0 & $N=1000$ \\
		15 & EIGENB & 2550 & 2550 & 99 & 0 & $N=50$ \\
		16 & FREURONE & 5000 & 9998 & $5.0485565\times 10^{6}$ & $6.081592\times 10^{5}$ & $N=5000$ \\
		17 & INTEGREQ & 1000 & 1000 & 5.678349 & 0 & $N=1000$ \\
		18 & MOREBVNE & 1000 & 1000 & $3.961509\times 10^{-6}$ & 0 & $N=1000$ \\
		19 & MSQRTA & 4900 & 4900 & $7.975592\times 10^{4}$ & 0 & $P=70$ \\
		20 & MSQRTB & 1024 & 1024 & 7926.444 & 0 & $P=32$ \\
		21 & OSCIGRNE & 1000 & 1000 & $6.120720\times 10^{8}$ & 0 & $N=1000$ \\
		22 & PENLT1NE & 1000 & 1001 & $1.114448\times 10^{17}$ & $9.686272\times 10^{-8}$ & $N=1000$ \\
		23 & POWELLSE & 1000 & 1000 & 418750 & 0 & $N=1000$ \\
		24 & SEMICN2U & 1000 & 1000 & $1.960620\times 10^{4}$ & 0 & $(N,LN)=(1000,900)$ \\
		25 & SPMSQRT & 1000 & 1664 & 797.0033 & 0 & $M=334$ \\
		26 & VARDIMNE & 1000 & 1002 & $1.241994\times 10^{22}$ & 0 & $N=1000$ \\
		27 & YATP1SQ & 2600 & 2600 & $5.184111\times 10^{7}$ & 0 & $N=50$ \\
		28 & YATP2SQ & 2600 & 2600 & $2.246192\times 10^{7}$ & 0 & $N=50$ \\
		\noalign{\smallskip}\hline
	\end{tabular}
	} 
	\caption{Details of large-scale test problems from the CUTEst test set (showing $2f(\b{x}_0)$ and $2f(\bx^*)$ as the implementations of DFO-LS and DFBGN do not have the $1/2$ constant factor in \eqref{eq_ls_definition}). The set of problems are taken from those in \cite[Table 3]{Cartis2019a}; the relevant parameters yielding the given $(n,m)$ are provided. The value of $n$ shown excludes fixed variables.}
	\label{tab_cutest_large_problems}
\end{table}

\end{document}